\documentclass[12pt]{amsart}
\usepackage[utf8]{inputenc}
\usepackage[left=2.cm,right=2.cm,top=3cm,bottom=2cm]{geometry}
\usepackage{graphicx}   
\usepackage{amssymb}
\usepackage{graphicx}
\usepackage{epstopdf}
\usepackage{color}
\usepackage{pgf,tikz}
\usetikzlibrary{arrows}
\usepackage{hyperref} 
\usepackage{mdframed}
\usepackage{bm}
\usepackage{shuffle}

\newtheorem{theo}{Theorem}
\newtheorem*{theo*}{Theorem}
\newtheorem{cl}{Claim}
\newtheorem{lem}{Lemma}
\newtheorem{cor}{Corollary}
\newtheorem{prop}{Propositon}
\newtheorem{conj}{Conjecture}
\newtheorem{prob}{Problem}

\newcommand{\SYT}{\mathop{\mathrm{SYT}}}
\newcommand{\Pa}{\mathop{\mathcal{P}}}
\newcommand{\sch}{\mathop{\mathrm{Sch}}}
\newcommand{\schP}{\mathop{\widetilde{\mathrm{Sch}}}}
\newcommand{\schT}{\mathop{\overline{\mathrm{Sch}}}}
\newcommand{\schB}{\mathop{\underline{\mathrm{Sch}}}}
\newcommand{\A}{\mathop{\mathrm{area}}}
\newcommand{\B}{\mathop{\mathrm{bounce}}}
\newcommand{\N}{\mathop{\mathrm{numph}}}
\newcommand{\dinv}{\mathop{\mathrm{dinv}}}
\newcommand{\R}{\mathop{\mathrm{read}}}

\newcommand{\dd}{\mathop{\mathrm{des}}}
\newcommand{\DD}{\mathop{\mathrm{Des}}}
\newcommand{\m}{\mathop{\mathrm{maj}}}

\newcommand{\inv}{\mathop{\mathrm{inv}}}
\newcommand{\co}{\mathop{\mathrm{co}}}
\newcommand{\hooks}{\mathop{\mathrm{hooks}}}
\newcommand{\hook}{\mathop{\mathrm{hook}}}

\newcommand{\pp}{\mathop{\mathrm{Part}}}
\newcommand{\TT}{\mathop{\mathrm{Touch}}}
\newcommand{\ttt}{\mathop{\mathrm{touch}}}
\usepackage{ stmaryrd }

\title[Toward a Schurification]{Toward a Schurification of Parking Function Formulas via bijections with Young Tableaux}
\author{\href{wallace.nancy@courrier.uqam.ca}{Nancy Wallace}}

\begin{document}
\maketitle
\tableofcontents
\footnote{This work was supported by a scholarship from the NSERC.}
\begin{abstract}
This paper contains a partial answer to the open problem 3.11 of \cite{[H2008]}. That is to find an explicit bijection on Schr\"oder paths that inverts the statistics area and bounce. This paper started as an attempt to write the sum over $m$-Schr\"oder paths with a fix number of diagonal steps into Schur functions in the variables $q$ and $t$. Some  results have been generalized to parking functions, and some bijections were made with standard Young tableaux giving way to partial combinatorial formulas in the basis $s_\mu(q,t)s_\lambda(X)$ for $\nabla(e_n)$ (respectively, $\nabla^m(e_n)$), when $\mu$ and $\lambda$ are hooks (respectively, $\mu$ is of length one). We also give an explicit algorithm that gives all the Schr\"oder paths related to a Schur function $s_\mu(q,t)$ when $\mu$ is of length one. In a sense, it is a partial decomposition of Schr\"oder paths into crystals.
\end{abstract}
%%%%%%%%%%%%%%%%   Introduction
\section{Introduction}
In this paper, Proposition~\ref{Prop : prob 3.1}  gives a partial answer to the open problem 3.11 of \cite{[H2008]}. The problem asks for an explicit bijection on Schr\"oder paths that inverts the statistic area and bounce. But the aim of this paper is to decompose parking functions and Schr\"oder paths in terms of the basis $s_\mu(q,t)s_\lambda(X)$. It is then used in \cite{[Wal2019a]} to give explicit combinatorial formulas for the modules of multivariate diagonal harmonics. In other words, the combinatorics of parking functions are used to elevate the understanding of the structure of the modules of diagonal harmonics. This combinatorial representation was first known as the Shuffle Conjecture. It was introduced in \cite{[HHLRU2005]} and proven by Carlson and Mellit in \cite{[CM2015]}, \cite{[M2016]}. It was shown beforehand in \cite{[GH1996a]} and \cite{[H2002]} that the Frobenius transformation of its graded characters may be expressed as $\nabla^m(e_n)$, where $\nabla$ is the Macdonald eigenoperator introduced in \cite{[BG1999]}, and $e_n$ is the $n$-th elementary symmetric function, both recalled in Section~\ref{Sec : tools}, along with classical combinatorial tools. 
More precisely, we will give a partial decomposition of parking functions and Schr\"oder paths in terms of the basis $s_\mu(q,t)s_\lambda(X)$. By proving the following:
\begin{theo}\label{The : main} If  $\mu\in\{(d,1^{n-d})~|~1\leq d\leq n\}$ and $\nu\vdash n$, then:
  \begin{equation}\label{Eq : 1 de the principale} \langle \nabla(e_n), s_{\mu}\rangle|_{\hooks}=\sum_{\tau\in \SYT(\mu)}  s_{\m(\tau)}(q,t)+\sum_{i=2}^{\dd(\tau)} s_{\m(\tau)-i,1}(q,t),
  \end{equation}
    \begin{equation}\label{Eq : 2 de the principale}\langle \nabla^m(e_n), s_\nu \rangle|_{1 \pp}=\sum_{\tau\in \SYT(\nu)}  s_{m\binom{n}{2}-\m(\tau')}(q,0)=\sum_{\tau\in \SYT(\nu)}  s_{m\binom{n}{2}-\m(\tau')}(q,0),
  \end{equation}
  and:
  \begin{equation}\label{Eq : 3 de the principale}\langle \nabla^m(e_n), e_n\rangle|_{\hooks}=  s_{m\binom{n}{2}}(q,t)+\sum_{i=2}^{n-1} s_{m\binom{n}{2}-i,1}(q,t).
  \end{equation}
\end{theo}
This will be done by characterizing particular parking functions, in Section~\ref{Sec : parking}, which leads to Equation\eqref{Eq : 2 de the principale}. In Section~\ref{Sec : schroder}, we restrict the characterization on Schr\"oder paths. In Section~\ref{Sec : bijections}, we give bijections between subsets of Schr\"oder paths and Standard Young tableaux, and use them to prove Equation~\eqref{Eq : 1 de the principale} and Equation~\eqref{Eq : 3 de the principale}.
Moreover, in Section~\ref{Sec : algo}, we exhibit an explicit algorithm that gives all the Schr\"oder paths associated to a Schur function in the variables $q$ and $t$ when $\mu$ is of length one. We will briefly explain, in Section~\ref{Sec : cristaux}, what it means in terms of Crystal decomposition. We end with a list of problems to solve in Section~\ref{Sec : conclu}. Finally, in Section~\ref{Sec : combi chemins}, we will recall notions on path combinatorics.
 %%%%%%%%%%%%%%%%%%%%%%%%%%%%%%%%%%%%%%%%%%%%%%%
 %%%%%%%%%%%%%%%    combinatorial tools
  %%%%%%%%%%%%%%%%%%%%%%%%%%%%%%%%%%%%%%%%%%%%%%%
  \section{Combinatorial Tools}\label{Sec : tools}
 
The notions discussed in this section are classical and are recalled to set notations.
An alphabet, $A$, is a set. The elements of that set are called \begin{bf}letters\end{bf}. A \begin{bf}word\end{bf} is a finite sequence of elements of $A$, we usually omit the parentheses and the commas. The empty word is denoted $\varepsilon$. The number of letters in a word $w$ is called the \begin{bf}length\end{bf}, denoted $|w|$, the number of occurrences of the letter $a$ in $w$ is denoted $|w|_a$. The set of words of length $n$ in the alphabet $A$ is denoted $A^n$, we denote $A^*$ the set $\cup_n A^n$. A \begin{bf}factor\end{bf} of $w$ is a consecutive subsequence of $w$. Additionally, if we are interested in word ending with a fixed factor $u$, we will denote the set $A^*u$, and $u$ is called a suffix. If we want those words to be of length, $n+|u|$ we will denote the set $A^nu$. Likewise, a factor at the beginning of a word is called a prefix, and the set of words with prefix $u$ is denoted $uA^*$. For a word $w=w_1w_2\cdots w_k$, $w^n$ is the concatenation on $m$ copies of $w$, and $w^{-1}=w_k\cdots w_2 w_1$. For two words $u$ and $w$ the set $u\shuffle w$ is the set containing all words such that $u$ and $w$ are subsequences. We call these words \begin{bf}shuffles\end{bf}. 
%Finally, for a set a variables $\{x_1,\cdots,x_n\}$ and $w\in \{1,\ldots,n\}^k$ we will use the notation $x_w$ for the monomial $x_{w_1}\cdots x_{w_n}$.
A \begin{bf}permutation\end{bf} of $n$ can be represented as words of $\{1,\ldots,n\}^n$ with all distinct letters. The \begin{bf}descent set of a permutation $w=w_1\cdots w_n$, \end{bf} denoted $\DD(w)$, is the set of $i$'s such that $w_i>w_{i+1}$. The cardinality of the set will be denoted $\dd(w)$. The \begin{bf}major index of a permutation\end{bf}, denoted $\m(w)$, is by definition $\m(w)=\sum_{i\in\DD(w)} i$. To avoid confusion we will write the \begin{bf}inverse of a permutation\end{bf} $w$, $\textrm{inv}(w)$.
 
 A \begin{bf}partition\end{bf} of $n$ is a decreasing sequence of positive integers it can be represented by a Ferrers diagram (see Figure~\ref{Fig: diagram}). Each number in the sequence is called a \begin{bf}part\end{bf}, and, if it has $k$ parts, it is of \begin{bf}length\end{bf} $k$ denoted $\ell(\lambda)=k$. If $\lambda=\lambda_1,\cdots,\lambda_k$ and $n=\sum_i\lambda_i$, we say $\lambda$ is of \begin{bf}size\end{bf} $n$, denoted $|\lambda|=n$. Although the notation $|\cdot|$ is used for words and partitions, it will be clear by context which one is used.  For $\lambda$ a  \begin{bf}Ferrers diagram \end{bf} of shape $\lambda=\lambda_1,\lambda_2,\cdots,\lambda_k$ is a left justified pile of boxes having $\lambda_i$ boxes in the $i$-th row. We will use the French notation; hence, the second row lies on top of the first row (see Figure~\ref{Fig: diagram}). We can see them as a subset of $\mathbb{N}\times\mathbb{N}$ if we put the bottom left corner of the diagram to the origin. In this setting, we can associate the bottom left corner of a box to the coordinate it lies on. We say a partition is  \begin{bf}hook-shaped\end{bf} if it has the shape $(a,1, \cdots,1)=(a,1^k)$, where $a, k \in \mathbb{N}$. The  \begin{bf}conjugate \end{bf} of a partition $\lambda$, (or a diagram) is denoted $\lambda'$, and is its reflection through the line $x=y$ (see Figures~\ref{Fig: conjugate}). 
\begin{figure}[!htb]
\begin{minipage}{4.7cm}
\centering
\begin{tikzpicture}[scale=.5]
\draw (0,0)--(4,0)--(4,1)--(2,1)--(2,3)--(1,3)--(1,5)--(0,5)--(0,0);
\draw (1,0)--(1,3);
\draw(2,0)--(2,1);
\draw (3,0)--(3,1);
\draw (4,0)--(4,1);
\draw (0,1)--(2,1);
\draw (0,2)--(2,2);
\draw (0,3)--(1,3);
\draw (0,4)--(1,4);
%\node(e) at (-2,2.5){};
\end{tikzpicture}
\caption{
$\lambda=42211$}
\label{Fig: diagram}
\end{minipage}
\begin{minipage}{4.7cm}
\centering
\begin{tikzpicture}[scale=.5]
\draw (0,0)--(5,0)--(5,1)--(3,1)--(3,2)--(1,2)--(1,4)--(0,4)--(0,0);
\draw (1,0)--(1,2);
\draw(2,0)--(2,2);
\draw (3,0)--(3,1);
\draw (4,0)--(4,1);
\draw (0,1)--(3,1);
\draw (0,2)--(1,2);
\draw (0,3)--(1,3);
\draw (0,4)--(1,4);
\end{tikzpicture}
\caption{
$\lambda'=5311$}
\label{Fig: conjugate}
\end{minipage}
\begin{minipage}{7.7cm}
\centering
\begin{tikzpicture}[scale=.5]
\draw (0,0)--(5,0)--(5,1)--(3,1)--(3,2)--(1,2)--(1,4)--(0,4)--(0,0);
\draw (1,0)--(1,2);
\draw(2,0)--(2,2);
\draw (3,0)--(3,1);
\draw (4,0)--(4,1);
\draw (0,1)--(3,1);
\draw (0,2)--(1,2);
\draw (0,3)--(1,3);
\draw (0,4)--(1,4);
\node(1) at (.5,.5){$1$};
\node(2) at (1.5,.5){$2$};
\node(3) at (.5,1.5){$3$};
\node(4) at (.5,2.5){$4$};
\node(5) at (2.5,.5){$5$};
\node(6) at (3.5,.5){$6$};
\node(7) at (4.5,.5){$7$};
\node(8) at (1.5,1.5){$8$};
\node(9) at (2.5,1.5){$9$};
\node(10) at (.5,3.5){$10$};
\end{tikzpicture}
\caption{$\tau \in \SYT(5311)$, with descent set $\{2,3,7,9\}$ and major index $25$}\label{Fig : tableau}
\end{minipage}
\end{figure}
A \begin{bf}tableau\end{bf} is a filling of a diagram by positive integers, the number in each box is called an \begin{bf}entry\end{bf}. The \begin{bf}size\end{bf} of a tableau relates to the size of the diagram it fills. It is said to be a  \begin{bf}semi-standard Young tableau\end{bf} if  all entries are weakly increasing in rows and strictly increasing in columns. A \begin{bf}standard Young tableau\end{bf} is a tableaux of  size $n$, such that all numbers from $1$ to $n$ appear exactly once and all entries are strictly increasing in rows and columns. If a tableau is a filling of the diagram associated to the  partition $\lambda$,   it is said to be of \begin{bf}shape\end{bf} $\lambda$. The set of standard Young tableaux of shape $\lambda$ is denoted $\SYT(\lambda)$. The \begin{bf}descent set of a tableau\end{bf} $\tau$, denoted $\DD(\tau)$ is the set of entries $i$ such that $i+1$ lies in a higher row. The cardinality of the descent set of $\tau$ is denoted $\dd(\tau)$, and the sum of the elements in the descent set is the \begin{bf}major index\end{bf} denoted $\m(\tau)$ (see Figure~\ref{Fig : tableau}). Again, it will be clear by context if the descent set and the major index are used on words or tableaux. Since each box of $\tau$ is associated to its own entry, we will write $c\in \tau$ when we refer to the entry $c$ in the tableau $\tau$. We will use the notation $x_\tau$ for the monomial $\prod_{c\in\tau}x_c$.

For a possibly infinite set of variables, $X=\{x_1,\ldots,x_n\}$, the \begin{bf}elementary symmetric functions\end{bf} $e_n(X)$ are the sum of all square-free monomials of degree $n$ in the set of variables $X$. The symmetric function $e_\lambda$ is simply $e_{\lambda_1}e_{\lambda_2}\cdots e_{\lambda_{\ell(\lambda)}}$.  The elementary symmetric functions form a basis of the symmetric functions. Another basis is the set of Schur functions. For $\lambda$ a partition the \begin{bf}Schur function\end{bf} $s_\lambda(X)=\sum x_\tau$, where the sum is over all semi-standard Young tableau of shape $\lambda$. The Schur basis in the $X$ variables is self-dual for the Hall scalar product, denoted $\langle \--,\-- \rangle$. We will use this notation when we want to display the coefficient of a particular Schur function. Note that the Schur functions in the variables $q$ and $t$ are coefficients and can go in and out of the scalar product. We will sometimes call Schur functions index by partitions that are hook-shaped, hook-shaped Schur functions or, simply, hook Schur functions. It will also be useful to remember that $e_n=s_{1^n}$. Furthermore, the \begin{bf}complete homogeneous symmetric functions\end{bf} are a basis such that $h_n(X)=s_{(n)}(X)$ and $h_\lambda:=h_{\lambda_1}h_{\lambda_2}\cdots h_{\lambda_{\ell(\lambda)}}$. We simply write $e_\lambda$ for $e_n(X)$, $s_\lambda$ for $s_\lambda(X)$ and $h_\lambda$ for $h_\lambda(X)$ (not for $e_\lambda(q,t)$, $h_\lambda(q,t)$ or $s\lambda(q,t)$). A curious reader could look at \cite{[Mac1995]}. 

The modified Macdonald polynomials $\tilde{H}_\mu(X;q,t)$ form another base of the ring of symmetric functions. In \cite{[BG1999]} Bergeron and Garsia introduce the operator $\nabla$ defined with the modified Macdonald polynomials as eigenfunctions, with eigenvalues $\prod_{(i,j)\in\mu}q^it^ j$. The Shuffle Theorem, proven by Carlson and Mellit (see \cite{[CM2015]} and \cite{[M2016]}), gives a combinatorial formula for $\nabla^m(e_n)$. This formula uses path combinatorics.
  %%%%%%%%%%%%%%%%%%%%%%%%%%%%%%%%%%%%%%%%%%%%%%%
 %%%%%%%%%%%%%%%  Path combinatorics
  %%%%%%%%%%%%%%%%%%%%%%%%%%%%%%%%%%%%%%%%%%%%%%%
\section{Path Combinatorics}\label{Sec : combi chemins}
 
 Before we can state the Shuffle Theorem, we need more classical definitions, relating to path combinatorics. More details on these classical notions can be found in \cite{[H2008]}.
 
The following $q$-analogues will be very useful:
\begin{equation*}
[n]_q:=1+q+q^2+\cdots+q^{n-1}, \vspace{20pt}  [n]!_q:=\prod{i=1}^n [i]_q, \hspace{20pt} \text{and}\hspace{20pt} \begin{bmatrix}n\\k\end{bmatrix}_q:=\frac{[n]!_q}{[n-k]!_q[k]!_q}.
\end{equation*}
Let $\mathcal{C}^n_k$ be the set of paths composed of north and east steps, in an $(n-k)\times k$ grid starting at the bottom left corner. The \begin{bf}area\end{bf} of a path is the number of boxes under the path. A classical result relates the Gaussian Polynomials to path combinatorics. Indeed, $\begin{bmatrix}n\\k\end{bmatrix}_q=\sum_{\gamma\in \mathcal{C}^n_k}q^{\A(\gamma)}$. A path $\gamma \in \mathcal{C}_k^n$ can be identified with a word, $w_\gamma$ in $\{N,E\}^n$ such that $|w_\gamma|_E=k$. To facilitate reading we will frequently refer to $\gamma$ when we talk about $w_\gamma$ (see Figure~\ref{Fig : aire} for an example). 
\begin{figure}[!htb]
\centering
\begin{tikzpicture}[scale=.75]
\filldraw[pink] (0,0)--(2,0)--(2,1)--(3,1)--(3,5)--(4,5)--(4,0)--(0,0);
\draw[gray,very thin] (0,0) grid (4,5);
\draw[very thick] (0,0)--(2,0)--(2,1)--(3,1)--(3,5)--(4,5);
\end{tikzpicture}
\caption{Path of area $6$ in a $5\times 4$ grid, with word representation $EENENNNNE$}\label{Fig : aire}
\end{figure}
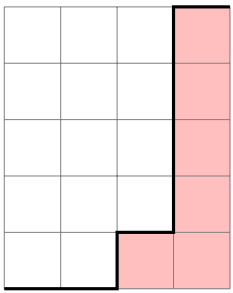

In an $n\times m$ grid, \begin{bf}the main diagonal\end{bf} is the diagonal starting at the bottom left corner and finishing at the top right corner. A \begin{bf}Dyck path\end{bf} of size $(n,m)$ is a path composed of north and east steps, starting at the bottom left corner  in an $n\times m$ grid such that the path always stays over the main diagonal. The set of such paths is denoted $\mathcal{D}_{n,m}$. A classical result makes it possible to represent Dyck paths by words in $\{N,E\}^*$ such that for all $\gamma_i$, prefix of $\gamma$, we have $|\gamma_i|_N\geq \frac{m}{n} |\gamma_i|_E$. The lines of the grid are numbered from bottom to top. A line $i$ is said to \begin{bf}contain an east step\end{bf} if the factor starting with the $i$-th north step and ending the letter before the $i+1$-th north step contains an east step. A \begin{bf}column of a path\end{bf}, $\gamma$, is a factor $N^jE^k$ such that $\gamma=uN^jE^kw$, where $u,w \in N\{N,E\}E\cup\{\varepsilon\}$. For example, in Figure~\ref{Fig : aire cat}, the path has 3 columns.

The \begin{bf}area of a Dyck path\end{bf} will be the number of boxes under the path and over the main diagonal (see Figure~\ref{Fig : aire cat} for an example). The area of the line $i$, denoted $a_i$, is the number of boxes in the line $i$ that are between the path and the main diagonal. Obviously, the area of a path is the sum of the $a_i$'s. The path $\gamma$ is said to have a \begin{bf}return to the main diagonal\end{bf} if there is $\gamma_i$, a non-trivial prefix of $\gamma$, such that $\gamma_i$ is a Dyck path and the end point of $\gamma_i$ lies on the main diagonal of $\gamma$.  The \begin{bf}touch sequence of a path\end{bf} $\gamma$, denoted $\TT(\gamma)$, is defined  as a sequence $(\gamma_1,\ldots,\gamma_k)$ of factors of $\gamma$ such that $\gamma=\gamma_1\cdots\gamma_k$, all $\gamma_i$ are Dyck paths, all  $\gamma_i$ contain no return to the main diagonal, and the end points of $\gamma_i$ returns to the main diagonal. Usually, the sequence $(\frac{1}{2}|\gamma_1|,\ldots,\frac{1}{2}|\gamma_k|)$ defines the touch vector $\ttt(\gamma)$. The touch vector contains all the  \begin{bf}touch points\end{bf}. For example, in Figure~\ref{Fig : aire cat}, $\TT(NNENNEEENE)=(NNENNEEE,NE)$ and $\ttt(NNENNEEENE)=(4,1)$. The \begin{bf}bounce path\end{bf} of a path $\gamma \in \mathcal{D}_{n,n}$  will be the Dyck path that remains under the path $\gamma$ and changes direction if and only if it touches the path $\gamma$ or the main diagonal. The \begin{bf}bounce vector\end{bf} is the vector containing the positions of the return to the main diagonal, starting from the top, of  the bounce path. For the bounce vector, the lines are numbered from the top starting at $0$. Finally, the \begin{bf}bounce statistic\end{bf} is the sum of the integer in the bounce vector minus $n$. It is, usually, simply called bounce (see Figure~\ref{Fig : bounce cat} for an example) . Note that the bounce statistic is not defined for Dyck paths in an $n\times m$ grid with $m\not=n$. In theses cases we use the diagonal inversion statistic which will be discussed at the end of this section.
\begin{figure}[!htb]
\begin{minipage}{8.2cm}
\centering
\begin{tikzpicture}[scale=.5]
\filldraw[pink] (0,1)--(0,2)--(1,2)--(1,4)--(3,4)--(3,3)--(2,3)--(2,2)--(1,2)--(1,1)--(0,1)--(0,2);
\draw[gray,very thin] (0,0) grid (5,5);
\draw[very thick] (0,0)--(0,2)--(1,2)--(1,4)--(4,4)--(4,5)--(5,5)--(0,0);
\node(1) at (6,.5){$1$};
\node(2) at (6,1.5){$2$};
\node(3) at (6,2.5){$3$};
\node(4) at (6,3.5){$4$};
\node(5) at (6,4.5){$5$};
\node(i) at (6,5.5){$i$};
\node(a1) at (7,.5){$0$};
\node(a2) at (7,1.5){$1$};
\node(a3) at (7,2.5){$1$};
\node(a4) at (7,3.5){$2$};
\node(a5) at (7,4.5){$0$};
\node(ai) at (7,5.5){$a_i$};
\end{tikzpicture}
\caption{Dyck path of area $4$, and word representation $NNENNEEENE$}\label{Fig : aire cat}
\end{minipage}
\begin{minipage}{7.7cm}
\centering
\begin{tikzpicture}[scale=.5]
\draw[gray,very thin] (0,0) grid (5,5);
\draw[very thick] (0,0)--(0,2)--(1,2)--(1,4)--(4,4)--(4,5)--(5,5)--(0,0);
\draw[red, thick, dashed] (0.1,0)--(0.1,1.9)--(2,2)--(2,3.9)--(4.1,3.9)--(4.1,4.9)--(5,4.9);
\node(5) at (6,0){$5$};
\node(4) at (6,1){$4$};
\node(3) at (6,2){$3$};
\node(2) at (6,3){$2$};
\node(1) at (6,4){$1$};
\node(0) at (6,5){$0$};
\node(r) at (10,2.8){\textcolor{blue}{return to diagonal}};
\node(r1) at (2,2){\textcolor{blue}{$\bullet$}};
\node(r2) at (4,4){\textcolor{blue}{$\bullet$}};
\node(p) at (-2,3){\textcolor{red}{peak}};
\node(p1) at (0,2){\textcolor{red}{$\bullet$}};
\node(p2) at (2,4){\textcolor{red}{$\bullet$}};
\draw[->,red] (p)--(p1);
\draw[->,red] (p)--(p2);
\draw[->,blue] (r)--(r1);
\draw[->,blue] (r)--(r2);
\end{tikzpicture}
\caption{Bounce vector is $(0,1,3,5)$ and bounce is $4$.}\label{Fig : bounce cat}
\end{minipage}
\end{figure}

A \begin{bf}Schr\"oder path\end{bf} of size $(n,rn)$  is a path composed of north, east and diagonal steps in an $n\times rn$ grid such that the path always stay over the main diagonal starting at the bottom left corner. In respect to Cartesian coordinates, a diagonal step corresponds to adding $(1,1)$ . The set of paths containing $d$ diagonal steps is denoted $\sch_{n,d}^{(r)}$. These paths can also be represented by words in the alphabet $\{N,E,D\}^*$ such that for all prefix $\gamma_i$ of $\gamma$ we have $|\gamma_i|_N\geq r|\gamma|_E$. Clearly $\mathcal{D}_{n,rn}=\sch_{n,0}^{(r)}$. Moreover, the path obtained by deleting all diagonal steps in a Schr\"oder path is a Dyck path. For a Schr\"oder path, $\pi$ this new path will be denoted $\Gamma(\pi)$. For example, the path $\pi$ in Figure ~\ref{Fig : aire sch}, is such that $\Gamma(\pi)$ is the path seen in Figure~\ref{Fig : aire cat}. We will also frequently use another subset of Schr\"oder paths:
\begin{equation*} 
 {\schP}_{n,d}=\{ \gamma \in {\sch}_{n,d} ~|~ \gamma=wNE, w \in \{D,N,E\}^*\}={\sch}_{n,d}\cap\{D,N,E\}^*NE
 \end{equation*} 
 
 The \begin{bf}area statistic of a Schr\"oder path\end{bf} is fairly the same as the other definitions of the area statistic. Instead of counting the squares, we count the number of \begin{bf}lower triangles\end{bf} under the path and over the main diagonal. Where a lower triangle is the lower half of a square cut in two starting by the botom left corner and ending at the top right corner (see Figure~\ref{Fig : aire sch} for an example). 
 
In \cite{[H2008]}, Haglund defines a bounce statistic for  Schr\"oder paths in an $n\times n$ grid. We first define the set of \begin{bf}peaks\end{bf} of the path, $\Gamma(\gamma)$. These are the set of lattice points at the beginning of an east step such that the bounce path of  $\Gamma(\gamma)$ switches from a north step to an east step. By extension the peaks of $\gamma$ are the lattice points found by reinserting the diagonal steps in $\Gamma(\gamma)$. The number of peaks of the path $\gamma$, with multiplicity, that lie under each diagonal step is the statistic \begin{bf}numph\end{bf}, denoted $\N(\gamma)$. The \begin{bf}bounce statistic\end{bf} will be extended to a Schr\"oder path, $\gamma$,  by the formula (see Figure~\ref{Fig : bounce sch} for an example):
\begin{equation*}
\B(\gamma)=\B(\Gamma(\gamma))+\N(\gamma)
\end{equation*}
Finally, touch points can be defined for Schr\"oder path, simply change Dyck path for Schr\"oder paths in the definition.
\begin{figure}[!htb]
\begin{minipage}{8.2cm}
\centering
\begin{tikzpicture}[scale=.5]
\filldraw[pink] (0,1)--(1,1)--(1,2);
\filldraw[pink] (1,2)--(2,2)--(2,3);
\filldraw[pink] (2,3)--(3,3)--(3,4);
\filldraw[pink] (3,4)--(4,4)--(4,5);
\filldraw[pink] (2,4)--(3,4)--(3,5);
\filldraw[pink] (4,5)--(5,5)--(5,6);
\filldraw[pink] (3,5)--(4,5)--(4,6);
\filldraw[pink] (5,6)--(6,6)--(6,7);
\filldraw[pink] (4,6)--(5,6)--(5,7);
\draw[gray,very thin] (0,0) grid (9,9);
\draw[very thick] (0,0)--(0,1)--(1,2)--(1,3)--(2,3)--(2,4)--(2,5)--(3,5)--(5,7)--(7,7)--(8,8)--(8,9)--(9,9)--(0,0);
\node(1) at (10,.5){$1$};
\node(2) at (10,1.5){$2$};
\node(3) at (10,2.5){$3$};
\node(4) at (10,3.5){$4$};
\node(5) at (10,4.5){$5$};
\node(6) at (10,5.5){$6$};
\node(7) at (10,6.5){$7$};
\node(8) at (10,7.5){$8$};
\node(9) at (10,8.5){$9$};
\node(i) at (10,9.5){$i$};
\node(a1) at (11,.5){$0$};
\node(a2) at (11,1.5){$1$};
\node(a3) at (11,2.5){$1$};
\node(a4) at (11,3.5){$1$};
\node(a5) at (11,4.5){$2$};
\node(a6) at (11,5.5){$2$};
\node(a7) at (11,6.5){$2$};
\node(a8) at (11,7.5){$0$};
\node(a9) at (11,8.5){$0$};
\node(ai) at (11,9.5){$a_i$};
\end{tikzpicture}
\caption{Schr\"oder path of area $9$, and word representation $NDNENNEDDEEDNE$}\label{Fig : aire sch}
\end{minipage}
\begin{minipage}{8.7cm}
\centering
\begin{tikzpicture}[scale=.5]
\draw[gray,very thin] (0,0) grid (9,9);
\draw[very thick] (0,0)--(0,1)--(1,2)--(1,3)--(2,3)--(2,4)--(2,5)--(3,5)--(5,7)--(7,7)--(8,8)--(8,9)--(9,9)--(0,0);
\node(p) at (-1.5,5){\textcolor{red}{peak}};
\node(p1) at (1,3){\textcolor{red}{$\bullet$}};
\node(p2) at (5,7){\textcolor{red}{$\bullet$}};
\draw[->,red] (p)--(p1);
\draw[->,red] (p)--(p2);
\end{tikzpicture}
\caption{For this path, $\gamma$, $\N(\gamma)=4$ and $\B(\gamma)=8$.}\label{Fig : bounce sch}
\end{minipage}
\end{figure}
 
The generating function of the Schr\"oder pathsis defined by:
\begin{equation*} 
{\sch}_{n,d}(q,t)=\sum_{\gamma\in {\sch}_{n,d}}q^{\B(\gamma)}t^{\A(\gamma)},
\end{equation*}
and the generating function of the Schr\"oder paths ending with $NE$ is defined by:
\begin{equation*} 
{\schP}_{n,d}(q,t)=\sum_{\gamma\in {\schP}_{n,d}}q^{\B(\gamma)}t^{\A(\gamma)}.
\end{equation*}
Since the subset $\schP$ is chosen to work with the bounce statistic which is not defined for $n\times nm$ grids when $m\not=1$, we will not define $\sch_{n,d}^{(m)}$. We will define ${\schP}_{n,d}^{(m)}(q,t)$ as follows:
\begin{equation*}
 {\schP}_{n,d}^{(m)}(q,t)=\sum_{k=d}^n (-1)^{k-d}{\schP}_{n,k}^{(m)}(q,t)
\end{equation*}
The reason is du to the fact that $s_{d+1,1^{n-d-1}}=\sum_{k=d}^n (-1)^{k-d}e_{n-d}h_{d}$, which is used in Equation~\eqref{Eq : nabla schur} due to Haglund and Equation~\eqref{Eq : nabla m schur} due to Mellit.
An \begin{bf}$(n,mn)$-parking function\end{bf} is a pair consisting of and a  $(n,mn)$-Dyck path and  a permutation of $n$, $w$, for which we write $w_i$ on the line $i$ of the Dyck path. Moreover, all factors of $w$ in a given column of the path must contain no descents (see Figure~\ref{Fig : exemple park} and Figure~\ref{Fig : nonexemple park} for examples). The set of all $(n,mn)$-parking function is denoted $\Pa_{n,mn}$ . 
\begin{figure}[!htb]
\begin{minipage}{9.5cm}
\centering
\begin{tikzpicture}[scale=.5]
\draw[gray,very thin] (0,0) grid (9,9);
\draw[very thick] (0,0)--(0,2)--(1,2)--(1,3)--(2,3)--(2,4)--(2,5)--(3,5)--(3,6)--(4,6)--(4,7)--(5,7)--(7,7)--(7,8)--(8,8)--(8,9)--(9,9);
\draw[very thick, gray] (0,0)--(9,9);
\node(1) at (.5,.5){$1$};
\node(2) at (8.5,8.5){$2$};
\node(3) at (1.5,2.5){$3$};
\node(4) at (2.5,3.5){$4$};
\node(5) at (2.5,4.5){$5$};
\node(6) at (4.5,6.5){$6$};
\node(7) at (3.5,5.5){$7$};
\node(8) at (.5,1.5){$8$};
\node(9) at (7.5,7.5){$9$};
\end{tikzpicture}
\caption{Here a parking function with $w=183457692$. The factors in each column are $18$, $3$, $45$, $7$, $6$, $9$, $2$, and contain no descents.}\label{Fig : exemple park}
\end{minipage}
\begin{minipage}{7.5cm}
\centering
\begin{tikzpicture}[scale=.5]
\draw[gray,very thin] (0,0) grid (9,9);
\draw[very thick] (0,0)--(0,2)--(1,2)--(1,3)--(2,3)--(2,4)--(2,5)--(3,5)--(3,6)--(4,6)--(4,7)--(5,7)--(7,7)--(7,8)--(8,8)--(8,9)--(9,9);
\draw[very thick, gray] (0,0)--(9,9);
\node(1) at (.5,.5){$8$};
\node(2) at (8.5,8.5){$2$};
\node(3) at (1.5,2.5){$1$};
\node(4) at (2.5,3.5){$5$};
\node(5) at (2.5,4.5){$4$};
\node(6) at (4.5,6.5){$6$};
\node(7) at (3.5,5.5){$7$};
\node(8) at (.5,1.5){$3$};
\node(9) at (7.5,7.5){$9$};
\end{tikzpicture}
\caption{This is NOT a parking function, since $w=831457692$. The factor in the first column is $83$ and has a descent.}\label{Fig : nonexemple park}
\end{minipage}
\end{figure}

The \begin{bf}reading word\end{bf} is obtained by reading the letters of $w$ (which are written immediately to the right of each north step) in regard to the diagonals parallel to the main diagonal starting from top right corner to the bottom left corner and starting with the diagonal that is the farthest from the main diagonal. For example, the reading word in Figure~\ref{Fig : exemple park} is $675438291$. The reading word of the parking function $(\gamma,w)$ is denoted $\R(\gamma,w)$.

The \begin{bf}area of a parking function\end{bf} is the area of its Dyck path. The \begin{bf}diagonal inversion statistic\end{bf}, of a parking function in $\mathcal{P}_{n,mn}$, (sometimes called $\dinv$ for short) is given by the formula $\sum_{i<j}d_i(j)$, where:
\begin{equation*}
d_i(j)=\begin{cases} \chi(w_i<w_j)\max(0,r-|a_i-a_j|)+\chi(w_i>w_j)\max(0,m-|a_j-a_i+1|) &\text{ if } i<j  
                                  \\  0 &\text{ if } i\geq j.
                                  \end{cases}
\end{equation*}
The diagonal inversion statistic of the parking function $(\gamma,w)$ is denoted $\dinv(\gamma,w)$. Note that all the definition work if $w$ is not a permutation (some authors use words but these can be regrouped with permutations as representatives).
Equivalently, for a $(\gamma,w)\in\Pa_{n,mn}$ we can consider the diagonal inversion of $(\tilde\gamma,\tilde w)$, where $\tilde\gamma$ is the $(mn,mn)$-Dyck path obtained by repeating all north steps $m$ times and for $w=w_1\cdots w_n$ we have $\tilde w=w_1^m\cdots w_n^m$ (here $\tilde w$ is not a permutation). In this case we can consider the sum $d_i(j)=\sum_{t=1}^m d_i^t(j)$, where $d_i^t(j)$ is calculated with $\tilde\gamma$, for the $t$-th copy of $w_i$. 

A visual representation of the diagonal inversion statistic for $(\gamma,w)\in\Pa_{n,n}$ is done considering one diagonal parallel to the main diagonal on each north step. For the north step on line $j$ if the diagonal crosses the north step on line $i$, with $i<j$ and $w_i<w_j$, then the pair $(i, j)$ contributes one to the diagonal inversion statistic. If the diagonal immediately over the diagonal crossing the north step on line $j$ crosses the line $i$, with $i<j$ and $w_i>w_j$, then the pair $(i, j)$ contributes one to the diagonal inversion statistic (see Figure~\ref{Fig : dinv diago} and Figure~\ref{Fig : dinv n-mn}).
\begin{figure}[!htb]
\begin{minipage}{8.5cm}
\centering
\begin{tikzpicture}[scale=.75]
\draw (0,0)--(0,1)--(1,1);
\draw (3,2)--(3,3)--(4,3);
\draw[dashed] (0,-.5)--(4,3.5);
\draw[dashed] (-1,-.5)--(3,3.5);
\draw (3,-.5)--(6,2.5);
\node(i) at (.5,.5){$w_i$};
\node(j) at (3.5,2.5){$w_j$};
\node(ai) at (1.75,0){$a_i$};
\node(aj) at (4.25,2){$a_j$};
\node(1) at (-.5,-.5){\tiny{$1$}};
\draw[->] (ai)--(0.1,0);
\draw[->] (ai)--(3.4,0);
\draw[->] (aj)--(3.1,2);
\draw[->] (aj)--(5.4,2);
\draw[->] (-.4,-.5)--(-0.1,-.5);
\draw[->] (-.6,-.5)--(-.9,-.5);
\node(m) at (7,1){\textcolor{blue}{main diagonal}};
\draw[->, blue] (m)--(4.6,1);
\end{tikzpicture}
\caption{The pair $(i,j)$ contributes $1$ if $i<j$ and $w_i>w_j$.}\label{Fig : dinv diago}
\end{minipage}
\begin{minipage}{8.5cm}
\centering
\begin{tikzpicture}[scale=.75]
\draw (1,0)--(1,1)--(2,1);
\draw (3,2)--(3,3)--(4,3);
\draw[dashed] (0,-.5)--(4,3.5);
\draw (3,-.5)--(6,2.5);
\node(i) at (1.5,.5){$w_i$};
\node(j) at (3.5,2.5){$w_j$};
\node(ai) at (2.5,0){$a_i$};
\node(aj) at (4.25,2){$a_j$};
\draw[->] (ai)--(1.1,0);
\draw[->] (ai)--(3.4,0);
\draw[->] (aj)--(3.1,2);
\draw[->] (aj)--(5.4,2);
\node(m) at (7,1){\textcolor{blue}{main diagonal}};
\draw[->, blue] (m)--(4.6,1);
\end{tikzpicture}
\caption{The pair $(i,j)$ contributes $1$ if $i<j$ and $w_i<w_j$.}\label{Fig : dinv n-mn}
\end{minipage}
\end{figure}
The Schr\"oder paths in an $n\times mn$ grid with $d$ diagonal steps can be represented by parking functions $(\gamma,w)$ such that $\R(\gamma,w) \in \{n-d+1,\cdots,n\}\shuffle \{n-d,\cdots,1\}$. Indeed, by definition of parking functions, if $w_i$ is in $ \{n-d+1,\cdots,n\}$, then the north step at line $i$ is followed by an east step. Therefore, for all $w_i$ in $ \{n-d+1,\cdots,n\}$ one can change the factor $NE$ on line $i$ for a $D$ and unlabel the path. This procedure gives us a Schr\"oder path with $d$ diagonal steps. Conversely, all $D$ steps of a Schr\"oder path can be changed for $NE$ factors and tagged in the reading order by the letters in $ \{n-d+1,\cdots,n\}$ and all the north steps can be tagged in the reading order by letters in $\{n-d,\cdots,1\}$. This bijection will be mostly used for proofs. Hence, we will often refer to Schr\"oder paths by their parking function description.

In \cite{[TW2018]}, Thomas and Williams proved that the zeta map, denoted $\zeta$, is a bijection on rational parking functions, that preserves statistics. In the $n\times n$ case it is such that $\dinv(\gamma,w)=\A(\zeta(\gamma,w))$ and $\A(\gamma,w)=\B(\zeta(\gamma,w))$. This will be used implicitly in the following way: if one as a decomposition of Schr\"oder paths with $d$ diagonal steps in Schur functions, in the variables $q$ and $t$, in terms of area and bounce, the decomposition in terms of diagonal inversions and area is the same. 
For more on $m$-Schr\"oder paths see \cite{[H2008]} and \cite{[S2005]}.

In this paper we will give explicit decomposition in Schur function in the variables $q$ and $t$ for $\langle\nabla^m (e_n), e_{n} \rangle|_{\hook}$, $\langle\nabla^m e_n,s_\mu \rangle|_{1 \pp}$ and $\langle\nabla e_n,s_{d+1,1^{n-d-1}} \rangle|_{\hook}$ by using Corollary 2.4 in \cite{[H2004]}:
\begin{theo*}[Haglund]\label{The : Hag} Let $n$, $d$ be positive integers such that $n\geq d$. Then:
\begin{equation}\label{Eq : nabla schur}{\schP}_{n,d}(q,t)=\langle\nabla e_n,s_{d+1,1^{n-d-1}} \rangle,
\end{equation}
and:
\begin{equation*}{\sch}_{n,d}(q,t)=\langle\nabla e_n,e_{n-d}h_d \rangle .
\end{equation*}
\end{theo*}
The following equalities will also be used and can be inferred from Mellit's proof found in \cite{[M2016]} of the compositional shuffle conjecture of \cite{[BGLX2016]}. Let $n$, $d$, $m$ be positive integer such that $n\geq d$. Then:
\begin{equation}\label{Eq : nabla m schur}{\schP}_{n,d}^{(m)}(q,t)=\langle\nabla^m e_n,s_{d+1,1^{n-d-1}} \rangle
\end{equation}
and:
\begin{equation*}\nabla^m(e_n)=\sum_{(\gamma,w)\in \Pa_{n,nm}} t^{\dinv(\gamma,w)}q^{\A(\gamma)}F_{\co(\DD(\inv(\R(\gamma,w))))}(X),
\end{equation*}
where $F_c$ is the fundamental quasisymmetric function index by the composition $c$ and for $S$ a subset of $\{1,\ldots,n-1\}$, $\textrm{co}(S)$  is the composition associated to $S$. 
We can infer the last result from \cite{[S1979]} and \cite{[H2002]}:
\begin{equation}\label{Eq : S-L,H}\nabla(e_n)|_{q=0}=\sum_{\tau\in \SYT(n)} t^{\m(\tau)} s_{\lambda(\tau)}
\end{equation}
%%%%%%%%%%%%%%%%%%%%%%%%%%%%%%%%%%%%%%%%%%%%%%%
%%%%%%%%%%%%%%% Algorithm on Schr\"oder paths related to Schur functions index by one part 
 %%%%%%%%%%%%%%%%%%%%%%%%%%%%%%%%%%%%%%%%%%%%%%%
\section{Algorithm on Schr\"oder Paths Related to Schur Functions Index by One Part Partitions}\label{Sec : algo}
It was proven in \cite{[H2002]} that $\nabla(e_n)$ is the character of the $GL_2\times \mathbb{S}_n$-module of diagonal harmonics. Hence, the polynomials $\sch_{n,d}(q,t)$ and $\schP_{n,d}(q,t)$ are symmetric in $q,t$ and can be written as a sum of Schur functions evaluated in $q,t$. 
The restriction of a symmetric function to the sum of Schur functions indexed by only one part (respectively,hook-shaped Schur functions)  will be denoted by $|_{1 \pp}$ (respectively, $|_{\hooks}$). For example, if $f=\sum_{\lambda\in C}c_\lambda s_\lambda$, then the restriction to one part is $f|_{1 \pp}= \sum_{\lambda\in C, \ell(\lambda)=1}c_\lambda s_\lambda$ (respectively, the restriction to hooks is $f|_{\hooks}= \sum_{\lambda\in C, \lambda=(a,1^b)}c_\lambda s_\lambda$).
In this section, we will give a simple formula for the Schur functions indexed by one part partitions contained in the development of ${\sch}_{n,d}(q,t)$. This will be done by proving an algorithm that allows us to describe all the paths of $\schP_{n,d}$ relating to the restrictions to Schur functions indexed by one part in the Schur function decomposition of $\schP_{n,d}(q,t)$. 

Let us first notice that Schur functions on a set of $k$ ordered variables are indexed by a partition with  length smaller or equal to $k$. This follows from the combinatorial definition of Schur functions, since  the filling of the first column of the semi-standard Young tableau must be strictly increasing and the first column as the same number of boxes to fill than the number of parts of the partition. Hence, Schur functions in two variables have at most two parts. Furthermore, a Schur function in two variables is such that $s_{a,b}(q,t)=q^{a-b}t^{a-b}(q^b+q^{b-1}t+\cdots+qt^{b-1}+t^b)$. Ergo, for $c$ an integer, the monomial $q^c$ as a non-zero coefficient in the decomposition of a symmetric function $f(q,t)$ if and only if the decomposition in Schur function contains the term $s_c(q,t)$. 
This is equivalent to stating that the Schur functions in the variables $q$ and $t$ appearing in $\sch_{n,d}(q,t)|_{1 \pp}$ are the same than the Schur functions in the variable $q$ appearing in $\sch_{n,d}(q,0)$. The only paths that contribute to monomials with non-zero coefficients in $\sch_{n,d}(q,0)$ are the paths $\gamma$ such that $\A(\gamma)=0$. Hence, these are the set of paths $\{NE,D\}^n$.
\\

The algorithm $\varphi$ takes a path, $\gamma$ in $\{NE,D\}^*$ for input, and returns a sequence of paths $(\gamma_0, \gamma_1, \ldots,\gamma_{\B(\gamma)})$, obtained as follows:
\begin{align*}
&\text{First set } \varphi(\gamma)=(\gamma_0), ~k=|\gamma|_E+|\gamma|_N+|\gamma|_D.
\\&\text{For $v=1$ to  $\B(\gamma)$;}
\\ &\hspace{20pt} \text{Let $\gamma_{v-1}=w_1w_2\cdots w_{k}$. }
\\ &\hspace{20pt} \text{Let $i$ be such that $w_i=E$, $w_{i+1}\not=E$ and $w_j=E$ implies $w_{j+1}=E$ or $j\leq i$.}
\\ &\hspace{20pt} \text{Set $\gamma_v=w_1\cdots w_{i-1}w_{i+1}w_iw_{i+2} \cdots w_{k}$.}
\\ &\hspace{20pt}    \text{Append $\gamma_v$ to the sequence $\varphi(\gamma)$.}
\\&\text{repeat ;}
\\&\text{return } \varphi(\gamma).
\end{align*}
\begin{figure}[!htb]
\centering
\begin{tikzpicture}[scale=.3]
\draw (0,-1)--(1,0)--(2,1)--(2,2)--(3,2)--(4,3)--(4,4)--(5,4)--(5,5)--(6,5);
\node(1) at (8,2.5){$\rightsquigarrow$};
\filldraw[pink] (14,4)--(14,5)--(15,5)--(15,4)--(14,4);
\draw (10,-1)--(11,0)--(12,1)--(12,2)--(13,2)--(14,3)--(14,5)--(16,5);
\node(2) at (18,2.5){$\rightsquigarrow$};
\filldraw[pink] (24,4)--(24,5)--(25,5)--(25,4)--(24,4);
\filldraw[pink] (22,2)--(23,2)--(23,3)--(22,2);
\draw (20,-1)--(21,0)--(22,1)--(22,2)--(23,3)--(24,3)--(24,5)--(26,5);
\node(3) at (28,2.5){$\rightsquigarrow$};
\filldraw[pink] (34,4)--(34,5)--(35,5)--(35,4)--(34,4);
\filldraw[pink] (33,3)--(33,4)--(34,4)--(34,3)--(33,3);
\filldraw[pink] (32,2)--(33,2)--(33,3)--(32,2);
\draw (30,-1)--(31,0)--(32,1)--(32,2)--(33,3)--(33,4)--(34,4)--(34,5)--(36,5);
\node(4) at (38,2.5){$\rightsquigarrow$};
\filldraw[pink] (43,3)--(44,3)--(44,4)--(45,4)--(45,5)--(43,5)--(43,3);
\filldraw[pink] (42,2)--(43,2)--(43,3)--(42,2);
\draw (40,-1)--(41,0)--(42,1)--(42,2)--(43,3)--(43,5)--(46,5);
\end{tikzpicture}
\caption{The sequence $\varphi(DDNEDNENE)$. }
\label{Fig : phi sur Schroder aire 0}
\end{figure}
We first need to prove this algorithm provides us with a sequence of Schr\"oder paths.
\begin{lem}For all $\gamma \in \{NE,D\}^*$, the elements of the sequence $\varphi(\gamma)$ are Schr\"oder paths. Moreover, if $\gamma \in \sch_{n,d}$, then $\varphi(\gamma) \subseteq \sch_{n,d}$.
\end{lem}
\begin{proof}Recall that for a path $\gamma$ to be a Schr\"oder path we must have $|\omega|_N\geq|\omega|_E$ for all prefix $\omega$ of $\gamma$. Because $\gamma_0$ is a Schr\"oder path, it is sufficient to show that if $\gamma_i$ is a Schr\"oder path, then the path $\gamma_{i+1}$, obtained by parsing one time through the algorithm, is also a Schr\"oder path. Let $\gamma_i=w_1w_2\cdots w_k$, then $\gamma_{i+1}=w_1\cdots w_{i-1}w_{i+1}w_iw_{i+2} \cdots w_{k}$ and the only prefixes that are different are $w_1\cdots w_{i-1}w_{i+1}$ compared to $w_1\cdots w_{i-1}w_{i}$ and $w_1\cdots w_{i-1}w_{i+1}w_i$ compared to $w_1\cdots w_{i-1}w_iw_{i+1}$. The last pair has the same letters ordered differently, so $|w_1\cdots w_{i-1}w_{i+1}w_i|_N\geq |w_1\cdots w_{i-1}w_{i+1}w_i|_E$ if and only if $|w_1\cdots w_{i-1}w_iw_{i+1}|_N\geq|w_1\cdots w_{i-1}w_iw_{i+1}|_E$. Now for the first pair of prefixes we have:
 \begin{align*}
 |w_1\cdots w_{i-1}w_{i+1}|_N &\geq |w_1\cdots w_{i-1}w_{i}|_N, \text{ since $w_i=E$ and $w_{i+1}\in\{N,D\}$,}
 \\      &\geq |w_1\cdots w_{i-1}w_{i}|_E, \text{ since $\gamma_i$ is a Schr\"oder path, }
 \\      & > |w_1\cdots w_{i-1}w_{i+1}|_E, \text{ since $w_i=E$ and $w_{i+1}\not=E$.}
 \end{align*}
 Therefore, $\gamma_{i+1}$ is indeed a Schr\"oder path.
 
 Finally, for $\gamma$ a Schr\"oder path we can move east steps to the left at least a number of times equal to bounce. Indeed, the peaks are associated to an east step and the numph statistic gives the number of diagonal steps over that east step (to the right of that $E$ in the word representation). Because the bounce path associated to $\Gamma(\gamma)$ changes direction at the peak when it hits an east step, except for the last entry, the vector associated to $\B(\Gamma(\gamma))$ gives the number of north steps over that east step (to the right of that $E$ in word representation).
 
\end{proof}
We give an example of $\varphi(\gamma)$ for $\gamma$ a Schr\"oder path not in $\{NE,D\}$.
\begin{figure}[!htb]
\centering
\begin{tikzpicture}[scale=.3]
\filldraw[pink] (13,3)--(13,4)--(14,4)--(14,3)--(13,3);
\filldraw[pink] (12,2)--(13,2)--(13,3)--(12,2);
\draw (10,-1)--(11,0)--(12,1)--(12,2)--(13,3)--(13,4)--(15,4)--(15,5)--(16,5);
\node(2) at (18,2.5){$\rightsquigarrow$};
\filldraw[pink] (23,3)--(23,4)--(24,4)--(24,3)--(23,3);
\filldraw[pink] (24,4)--(24,5)--(25,5)--(25,4)--(24,4);
\filldraw[pink] (22,2)--(23,2)--(23,3)--(22,2);
\draw (20,-1)--(21,0)--(22,1)--(22,2)--(23,3)--(23,4)--(24,4)--(24,5)--(26,5);
\node(4) at (28,2.5){$\rightsquigarrow$};
\filldraw[pink] (33,3)--(34,3)--(34,4)--(35,4)--(35,5)--(33,5)--(33,3);
\filldraw[pink] (32,2)--(33,2)--(33,3)--(32,2);
\draw (30,-1)--(31,0)--(32,1)--(32,2)--(33,3)--(33,5)--(36,5);
\end{tikzpicture}
\caption{The sequence $\varphi(DDNDNEENE)$. }
\label{Fig : phi sur Schroder quelconque}
\end{figure}
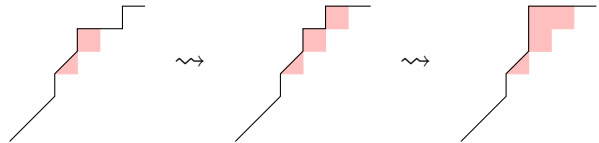
In Figure~\ref{Fig : phi sur Schroder quelconque} the $\B$ statistic does not decrease evenly throughout the iterations of the algorithm. In Figure~\ref{Fig : phi sur Schroder aire 0} each iteration increases the area statistic by exactly one and decreases the bounce statistic by exactly one. We will show in Lemma~\ref{Lem : +1-1} that this is not a coincidence if $\gamma_0$ is a Schr\"oder path of area $0$. But first, we need to show a result on the prefixes rellating to the paths in $\varphi(\gamma)$.
\begin{lem}\label{Lem : prefix algorithm}Let $\gamma$ be in  $\{NE,D\}^*$ and $\varphi(\gamma)=(\gamma_0,\ldots,\gamma_{\B(\gamma)})$. If $\gamma_i$ as a prefix $\alpha$ in $\{NE,D\}$, then $\alpha$ is a prefix of $\gamma$. Moreover, if $\alpha$ is the longest prefix of $\gamma_i$ such that $\alpha$ is in $\{NE,D\}$, then for $\beta$ such that $\alpha\beta=\gamma_i$ we have $\beta=\omega E^{|\beta|_E-1}$, where $\omega$ is a word in the alphabet $\{N,E,D\}^*$. Finally, if $\gamma_{i}=w_1\cdots w_{j-2}w_{j}w_{j-1}w_{j+1} \cdots w_{k}$ and $\gamma_{i-1}=w_1\cdots w_{j-2}w_{j-1}w_{j}w_{j+1} \cdots w_{k}$, then $w_{j-1}$ is a letter in $\omega$.
\end{lem}
\begin{proof}By induction on $i$. If $i=0$, then $\alpha=\gamma_0=\gamma$ and $\beta=\varepsilon$. For $i>0$ let $\alpha$ (respectively,$\alpha'$) be the longest prefix of $\gamma_i$ (respectively,$\gamma_{i-1}$) such that $\alpha$ is in $\{NE,D\}^*$ (respectively,$\alpha' \in \{NE,D\}^*$)  and $\beta$ (respectively,$\beta'$) be such that $\alpha\beta=\gamma_i$ (respectively,$\alpha'\beta'=\gamma_{i-1}$). By induction we know that $\alpha'$ is a prefix of $\gamma$, and $\beta'=\omega'E^{|\beta'|_E-1}$. 
By definition of each iteration of the algorithm, there is $j$ such that $\gamma_{i}=w_1\cdots w_{j-2}w_{j}w_{j-1}w_{j+1} \cdots w_{k}$ and $\gamma_{i-1}=w_1\cdots w_{j-2}w_{j-1}w_{j}w_{j+1} \cdots w_{k}$. 

If $|\alpha'|=l\leq j-2$, then $\beta'=w_{l+1}\cdots w_{j-2}w_{j-1}w_{j}w_{j+1} \cdots w_{k}$ and $\alpha=\alpha'$. By definition of the algorithm $w_j\in \{N,D\}$. Due to $\beta'=\omega'E^{|\beta'|_E-1}$, we must have that $w_j$ and $w_{j-1}$ are both letters of $\omega'$. Ergo, the suffix $E^{|\beta'|_E-1}$ of $\beta'$ is unchanged in $\beta$. Hence, $\beta=\omega E^{|\beta'|_E-1}=\omega E^{|\beta|_E-1}$ and $|\omega|_E=|\omega'|_E$.

If $|\alpha'|=l > j-2$, then by definition of the algorithm $w_{j-1}=E$, because $\alpha' \in\{NE,D\}^*$, $w_{j-2}=N$. Consequently,$\gamma_{i}=w_1\cdots w_{j-3}Nw_{j}Ew_{j+1} \cdots w_{k}$ and $\alpha=w_1\cdots w_{j-3}$ is in $\{NE,D\}^*$ and is a prefix of $\alpha'$. Thus, it is a prefix of $\gamma$. This means $\beta=Nw_{j}Ew_{j+1} \cdots w_{k}$. Each iteration of the algorithm swaps the rightmost east step that is not followed by an east step to the right. In consequence, $w_{j}\in\{N,D\}$ and we must have $\omega$ such that $\beta=\omega E^{|\beta|_E-1}$, with $w_{j-1}$ is a letter of $\omega$; otherwise the letters $w_j$ and $w_{j-1}$ would not have been swapped.
\end{proof}
\begin{lem}\label{Lem : +1-1}Let $\gamma$ be in  $\{NE,D\}^*$ and $\varphi(\gamma)=(\gamma_0,\ldots,\gamma_{\B(\gamma)})$. Then, for all $i$, such that $0\leq i < \B(\gamma)$, the following equalities hold:
\begin{align*}\A(\gamma_i)+1&=\A(\gamma_{i+1}), 
\\  \B(\gamma_i)&=\B(\gamma_{i+1})+1, 
\\  \A(\gamma_i)&=\B(\gamma_{\B(\gamma)-i}), \text{ and,}
\\  \B(\gamma_i)&=\A(\gamma_{\B(\gamma)-i}).
\end{align*}
\end{lem} 
\begin{proof}Let us first notice that the algorithm changes $EN$ for $NE$ or $ED$ for $DE$. In both cases, this adds exactly one lower triangle under the path. Therefore, $\A(\gamma_i)+1=\A(\gamma_{i+1})$. 

For the second condition, let $\gamma_{i}=w_1\cdots w_{j-1}w_{j}w_{j+1}w_{j+2} \cdots w_{k}$. By definition of the algorithm $\varphi$, we know that $\gamma_{i+1}=w_1\cdots w_{j-1}w_{j+1}w_{j}w_{j+2} \cdots w_{k}$, $w_j=E$ and $w_{j+1}\in \{N,D\}$. By Lemma~\ref{Lem : prefix algorithm}, we know $\gamma_{i+1}=\alpha\beta$ with $\alpha\in\{NE,D\}^*$, so there is a return to the main diagonal of the bounce path between $\alpha$ and $\beta$. Hence, the following east step is associated to a peak. By Lemma~\ref{Lem : prefix algorithm}, $\beta=\omega E^{|\beta|_E-1}$, $w_j$ is a letter of $\omega$ and $\omega$ contains exactly one east step. Consequently, there is a peak at $w_j$ in $\gamma_{i+1}$. Let $\gamma_i=\alpha'\beta'$, if $w_j$ is a letter in $\omega'$, the same reasoning leads to a peak at $w_j$ in $\gamma_i$. If $w_j$ is a letter in $\alpha'$, there is a peak at $w_j$ in $\gamma_i$, since  $\alpha'\in\{NE,D\}^*$. Thus, $NE=w_{j-1}w_j \in \TT(\alpha)$. 

If $w_{j+1}=D$, then $\B(\Gamma(\gamma_i))=\B(\Gamma(\gamma_{i+1}))$ because $\Gamma$ discards the diagonal steps. Recall that the peaks of a Schr\"oder path, $\gamma$, are also obtained from $\Gamma(\gamma)$, in consequence, the peaks in $\gamma_i$ and $\gamma_{i+1}$ are associated to the same east steps. Recall that numph is the number of diagonal steps, with multiplicity, positioned after a peak (higher if you consider the path itself rather than the word representation). The diagonal step $w_{j+1}$ is after the peak at $w_j$ in $\gamma_i$ and before the peak at $w_j$ in $\gamma_{i+1}$. All other peaks and diagonal steps remain unchanged. Hence, $\N(\gamma_i)=\N(\gamma_{i+1})+1$. 

If $w_{j+1}=N$, then the peak at $w_j$ moves one position to the right in the word representation (one line higher if you consider the path itself). By definition, at this point, the bounce path returns to the main diagonal and goes north to the next east step,
which, by Lemma~\ref{Lem : prefix algorithm}, are all after $\omega$. Therefore, the peak at $w_j$ does not move to a line already containing a peak unless $\omega$ ends with $w_{j}D^l$. In this last case, the peak at $w_j$ contributed $1$ to $\B(\gamma_i)$. In any case, the peak on the first line of $\Gamma(\gamma_{i+1})$, contributes $0$ to bounce. Consequently, in both cases, $\B(\Gamma(\gamma_i))=\B(\Gamma(\gamma_{i+1}))+1$. Additionally, all the east steps keep the same number of diagonal steps positioned after them. Hence, all the peaks keep the same number of diagonal steps positioned after them and $\N(\gamma_i)=\N(\gamma_{i+1})$.

Considering the path $\gamma$ in  $\{NE,D\}^*$ has an area equal to zero, the third and fourth conditions follows from the first two conditions. 
\end{proof}
We now present a map that will be useful for the discussion on crystals in Section~\ref{Sec : cristaux}.
For $\gamma$ in $\{NE,D\}^*$ we have $\varphi(\gamma)=(\gamma_0,\ldots,\gamma_{\B(\gamma)})$. With this notation we define the map:
\begin{align*}\label{Eq : tilde varphi}\tilde\varphi : \{\gamma \in\sch{}_{n,d-1} ~|~\A(\gamma)=0\}& \rightarrow\{\gamma \in \sch{}_{n,d-1} ~|~\A(\gamma)=1\}
\\  \gamma=\gamma_0&\mapsto\gamma_1
\end{align*}
This next lemma will be used in the proof of Theorem~\ref{The : main}. 
\begin{lem}\label{Lem :  ens 1 part aire=1} Let $d$ be an integer such that $1\leq d\leq n-1$, then the image of the map $\tilde\varphi$ is given by the set $\{uD^jNNEE, vNDED^jNE ~|~ u\in \{NE,D\}^{n-d-2}, v\in\{NE,D\}^{n-d-1}\}$.
\end{lem}
\begin{proof}Follows from the definition of $\varphi$.
\end{proof}
In order to give the decomposition in Schur functions evaluated in the variables $q$ and $t$, for $\langle \nabla(e_n), s_{d+1,1^{n-d-1}}\rangle|_{1 \pp}$, we will show that for a path $\gamma$ in $\{NE,D\}^*$, the sum $\sum_{\pi\in\varphi(\gamma)}q^{\B(\pi)}t^{\A(\pi)}$ is a Schur function in the variables $q$ and $t$. For the general result to hold, we need the intersection of these sets to be empty.
\begin{lem}\label{Lem : intersection vide}Let $\gamma$, $\pi$ be in $\{NE,D\}^*$ such that $\gamma\not=\pi$ then $\varphi(\gamma) \cap \varphi(\pi)=\emptyset$.
\end{lem}
\begin{proof} Let us first notice that the algorithm changes $EN$ for $NE$ or $ED$ for $DE$. Therefore, the relative order of the north and diagonal steps does not change for all paths in $\varphi(\gamma)$ and $ \varphi(\pi)$. Hence, $\gamma_0$ and $\pi_0$ have the same relative order in regard to the north and diagonal steps which uniquely determine paths of $\{NE,D\}^*$. Consequently, $\varphi(\gamma) \cap \varphi(\pi)=\emptyset$.
\end{proof}
We can now display a bijection that inverts the statistics area and bounce. This partially solves open problem 3.11 of \cite{[H2008]}.
\begin{prop}\label{Prop : prob 3.1}Let $n$ be a positive integer and $\varphi(\{NE,D\}^n)$ be the set $\cup_{\gamma\in\{NE,D\}^n} \varphi(\gamma)$. There  is a bijection, $\Omega_n$,  of $\varphi(\{NE,D\}^n)$ onto itself. For all $n\geq 1$ we have, $\A(\gamma_i)=\B(\Omega_n(\gamma_i))$ and $\B(\gamma_i)=\A(\Omega_n(\gamma_i))$. 
\end{prop}
\begin{proof}By Lemma~\ref{Lem : intersection vide}, for $\gamma \in \varphi(\{NE,D\}^n)$ there is a unique $\gamma_0$ and a unique $i$ such that $\gamma \in \varphi(\gamma_0)$ and $\gamma=\gamma_i$. Thus, we can define $\Omega_n(\gamma)=\gamma_{\B(\gamma_i)+\A(\gamma_i)-i}$. The result is a consequence of Lemma~\ref{Lem : +1-1}.
\end{proof}
The following Lemma gives us a full set representatives of Schur functions indexed by one part. 
\begin{lem}\label{Lem : bij schroder aire 0 et rectangle}
Let $A_d=\{ \gamma \in \{NE,D\}^n ~|~ |\gamma|_D=d \}$, there is a bijection $\theta: A_d \rightarrow \mathcal{C}^n_d$ such that $\theta(NE)=N$ and $\theta(D)=E$. Moreover, for $\gamma \in A_d$ we have $\B(\gamma)=\A(\theta(\gamma))+\binom{n-d}{2}$.
\end{lem}
\begin{proof}
In $A_d$, the factor $NE$, can be changed for a letter. A path $\gamma$ in $A_d$ is a word of length $n$ with $d$ occurrences of one letter and $n-d$ occurrences of the other letter. A path in $ \mathcal{C}^n_d$ can be represented by a word with $d$ occurrences of the letter  $E$ and $n-d$ occurrences of the letter $N$. Hence, $\theta$ merely relabels the letters and is a bijection.

Furthermore, in $A_d$ all east steps are associated to a peak; therefore, the $i$-th diagonal steps contribute the number of factors $NE$  before  the $i$-th diagonal steps to $\N$. But that number is the number of boxes under the $i$-th north step in $\theta(\gamma)$. Thus, $\N(\gamma)=\A(\theta(\gamma))$. Finally, $\B(\Gamma(\gamma))=\binom{n-d}{2}$ for all paths $\gamma$ in $A_d$ because all the $n-d$ east steps return to the main diagonal.
\end{proof}
The next proposition will be generalized for parking functions by Proposition~\ref{Prop : 1 part parking} and generalized for the restriction to Schur functions indexed by a hook-shaped partition, evaluated in the variables $q$ and $t$ by Theorem~\ref{The : main}. Although the generalizations will not account for all the paths related to each Schur functions.
\begin{prop}\label{Prop : sum algo 1 part} For $\gamma$ in $\{NE,D\}^*$, we have:
\begin{equation}\label{Eq : bounce}
\sum_{\gamma_i \in \varphi(\gamma)}q^{\B(\gamma_i)}t^{\A(\gamma_i)}=s_{\B(\gamma)}(q,t),
\end{equation}
\begin{equation}\label{Eq : 1 part 1-schroder}
\langle\nabla e_n,e_{n-d}h_d\rangle|_{1 \pp}=\underset{|\gamma|_D=d}{\sum_{\gamma \in \{NE,D\}^{n}}} s_{\B(\gamma)}(q,t)
=\sum_{\gamma\in \mathcal{C}^{n}_d} s_{\A(\gamma)+\binom{n-d}{2}}(q,t), \text{ and,}
\end{equation}
\begin{equation}\label{Eq : 1 part 1-schroder schur}
\langle\nabla e_n,s_{d+1,1^{n-d-1}} \rangle|_{1 \pp}=\underset{|\gamma|_D=d}{\sum_{\gamma \in \{NE,D\}^{n-1}NE}} s_{\B(\gamma)}(q,t)=\sum_{\gamma\in \mathcal{C}^{n-1}_d} s_{\A(\gamma)+\binom{n-d}{2}}(q,t).
\end{equation}
\end{prop}
\begin{proof}Equation~\eqref{Eq : bounce} follows from Lemma~\ref{Lem : +1-1}. For the first equality of Equation~\eqref{Eq : 1 part 1-schroder}, we notice that $s_a(q,t)=q^a+q^{a-1}t+\cdots+qt^{a-1}+t^a$, and, thus, by \hyperref[The : Hag]{Haglund's} Theorem, a Schur function indexed by a one part partition in $\langle\nabla e_n,e_{n-d}h_d \rangle$ can be associated to a path $\gamma$ in $\sch_{n,d}$ such that $\A(\gamma)=0$. But these are in $\{NE,D\}^n$ and have $d$ diagonal steps. For this reason,  by Equation~\eqref{Eq : bounce}, the equality holds.
The second equality of Equation~\eqref{Eq : 1 part 1-schroder} Follows from Lemma~\ref{Lem : bij schroder aire 0 et rectangle}. Finally, for Equation~\eqref{Eq : 1 part 1-schroder schur} we only need to notice that paths of $\mathcal{C}^n_d$ ending with a north step are in bijection with paths of $\mathcal{C}^{n-1}_d$ and have the same area. The result is a consequence of  \hyperref[The : Hag]{Haglund's} Theorem, Lemma~\ref{Lem : bij schroder aire 0 et rectangle} and Equation~\eqref{Eq : bounce}.
\end{proof}
We end this section with a result needed for the generalization of Theorem~\ref{The : main}.
\begin{cor}\label{Cor : aire 1 vs 1 part}
Let $d$ be an integer and $\gamma$ be a path in $\schP_{n,d}$. Then, the path $\gamma$ is such that $\gamma=\gamma'NDED^jNE$ or $\gamma=\gamma'NED^jNNEE$, with $\gamma'\in  \{NE,D\}^*$ if and only if $\A(\gamma)=1$ and $\gamma$ contributes to a Schur function  indexed by a partition of length $1$ in $ \langle \nabla(e_n), s_{d+1,1^{n-d-1}}\rangle$.
\end{cor}
\begin{proof} If  $\gamma=\gamma'NDED^jNE$ or $\gamma=\gamma'NED^jNNEE$), with $\gamma'\in  \{NE,D\}^*$, then, by Lemma~\ref{Lem :  ens 1 part aire=1}, $\gamma$ is in the image of $\tilde\varphi$ and the result follows from Proposition~\ref{Prop : sum algo 1 part}. 
 If $\A(\gamma)=1$ and $\gamma$ contributes to a Schur function  indexed by a partition of length $1$, by Lemma~\ref{Lem :  ens 1 part aire=1}, $\gamma$ is in the image of $\tilde\varphi$ and  the result follows from Proposition~\ref{Prop : sum algo 1 part}.
\end{proof}
%%%%%%%%%%%%%%%%%%%%%%%%%%%%%%%%%%%%%%%%%%%%%%%
%%%%%%%%%%%%%%%        From Parking Functions Formulas to Schur Functions   
%%%%%%%%%%%%%%%%%%%%%%%%%%%%%%%%%%%%%%%%%%%%%%%
\section{From Parking Functions Formulas to Schur Functions}\label{Sec : parking}
The aim of this section is to give a combinatorial formula for $\nabla^m(e_n)$ restricted to Schur functions indexed by one part partitions in the variables $q$ and $t$. We will denote this restriction   $\nabla^m(e_n)|_{1 \pp}$. In this section we will be using the diagonal inversion statistic, since bounce is not defined for parking functions. This is the main obstacle to knowing all path related to each Schur functions in the variables $q$ and $t$ in the formula. 
To obtain a formula for $\nabla^m(e_n)|_{1 \pp}$, we will give the necessary and sufficient conditions for a parking function to have a diagonal inversion statistic of $0$. We will also determine the necessary and sufficient conditions for a parking function to have a diagonal inversion statistic of $1$ if $m$ is greater or equal to $2$.  

The next three results are technical and will mostly be used to discard some reoccurring cases.
\begin{cl}\label{Cl : min dinv Park}Let $(\gamma,w)$ be in $\Pa_{n,nm}$ such that $\gamma$ has a factor $\gamma'=NE^pN$, with $ 1\leq p\leq m$, and with its north steps associated to $w_i$ and $w_{i+1}$. Then, if $w_i>w_{i+1}$, $d_i(i+1)=p-1$ and if $w_i<w_{i+1}$, $d_i(i+1)=p$.
\end{cl}
\begin{proof} By definition:
\begin{equation*}
d_i(i+1)=\chi(w_i<w_{i+1})\max(0,m-|a_i-a_{i+1}|)+\chi(w_i>w_{i+1})\max(0,m-|a_{i+1}-a_i+1|)
 \end{equation*}
Hence, by hypothesis $a_i=a_{i+1}+p-m$. Therefore:
 \begin{equation*}
 d_i(i+1)=\chi(w_i<w_{i+1})(p)+\chi(w_i>w_{i+1})(p-1)
 \end{equation*}
\end{proof}
\begin{lem}\label{Lem : dinv Park p>m}Let $(\gamma,w)$ be in $\Pa_{n,mn}$ such that $\gamma$ has a factor $\gamma'=NE^pN$, with $p > m$. Then,  $\dinv(\gamma,w)\geq m-1$. 
\end{lem}
\begin{proof}Let $w_i$ and $w_{i+1}$ be the letters of $w$ associated to the factor $\gamma'$. Since the path is continuous and over the main diagonal, there exists $j_1,\ldots,j_m$ such that the $k$-th copy (from the top) of $w_{i+1}$ is to the north on the same diagonal than the north step associated to the letter $w_{j_k}$ in $w$. Note that the $j_k$'s are not necessarily distinct, ergo, $d_{j_k}(i+1)=d_s^t(i+1)$ for $s=j_k$ and $j_k$ is the $t$-th copy (from the top) of $w_{s}$. Which means that for $1\leq k\leq m-1$ when $w_{j_k}>w_{i+1}$ we get $d_{j_k}(i+1)\geq 1$ (see Figure~\ref{Fig : djk(i+1) meme diago}) and when $w_{j_k}<w_{i+1}$  the copy $k+1$ of $w_{i+1}$ is to the north and one diagonal lower than $w_{j_k}$. Therefore, $d_{j_k}(i+1)\geq 1$ (see Figure~\ref{Fig : djk(i+1) pas meme diago}). Hence:
 \begin{equation*}\dinv(\gamma,w)= \sum_{s=1}^{n-1}\sum_{t=1}^m \sum_{r>l}^n d_s^t(r)\geq \sum_{k=1}^{m-1} d_{j_k}(i+1)\geq m-1.
 \end{equation*}
\end{proof}
\begin{figure}[!htb]
\begin{minipage}{8cm}
\centering
\begin{tikzpicture}[scale=.75]
\draw (0,0)--(0,4);
\draw (2,5)--(2,6)--(7,6)--(7,10);
\draw[dashed, red] (-.5,1.5)--(7.75,8.75);
\node(i) at (2.5,5.5){$w_i$};
\node(i1) at (8.5,9.75){$w_{i+1}$~copy $1$};
\node(ik) at (8.5,8.25){$w_{i+1}$~copy $k$};
\node(im) at (8.55,6.25){$w_{i+1}$~copy  $m$};
\node(s1) at (1.25,3.75){$w_{s}$~copy $1$};
\node(st) at (1.95,2.1){$w_{j_k}=w_{s}$~copy $t$};
\node(sm) at (1.35,.25){$w_{s}$~copy $m$};
\end{tikzpicture}
\caption{}\label{Fig : djk(i+1) meme diago}
\end{minipage}
\begin{minipage}{8cm}
\centering
\begin{tikzpicture}[scale=.75]
\draw (0,0)--(0,4);
\draw (2,5)--(2,6)--(7,6)--(7,10);
\draw[dashed, red] (-.5,1.5)--(7.75,8.75);
\draw[dashed, red] (4,4.95)--(7.75,8.25);
\node(i) at (2.5,5.5){$w_i$};
\node(i1) at (8.5,9.75){$w_{i+1}$~copy $1$};
\node(ik) at (8.5,8.25){$w_{i+1}$~copy $k$};
\node(ik1) at (9.,7.75){$w_{i+1}$~copy $k+1$};
\node(im) at (8.55,6.25){$w_{i+1}$~copy  $m$};
\node(s1) at (1.25,3.75){$w_{s}$~copy $1$};
\node(st) at (1.95,2.1){$w_{j_k}=w_{s}$~copy $t$};
\node(sm) at (1.35,.25){$w_{s}$~copy $m$};
\end{tikzpicture}
\caption{}\label{Fig : djk(i+1) pas meme diago}
\end{minipage}
\end{figure}
\begin{lem}\label{Lem : m=1 p>2 local}Let  $(\gamma,w)$ be in $\Pa_{n,mn}$. If there is a factor $\gamma$, $\gamma'=NE^pN$ associated to the lines $i-1$ and $i$ such that $p\geq 2$, then there is $k$ such that $d_k(i)=1$, if $m=1$ and $d_k(i)\geq 1$, if $m>1$. 
\end{lem}
\begin{proof}We will work with $\tilde\gamma$ and $\tilde w$. Therefore we can use a Dyck path in an $mn\times mn$ grid. Let us suppose there is no such $k$.
 Dyck paths have the property of always having more north steps than east steps for all prefixes. Hence, there is a line $j$ in $\gamma$ such that $j<i$ and the north step on line $j_s$ is on the same diagonal than the north step on line $i_1$. We can assume $j_s$ is the upper bound of such lines. 
 
 By hypothesis, $d_j(i)=0$ for $(\gamma,w)$, in consequence, $d_{j_s}(i_1)=0$  and we must have $\tilde w_{j,s}> \tilde w_{i,1}$, the contrary would lead to $d_{j_s}(i_1)=1$.  Since $p\geq 2$, there is $l$ such that $j\leq l<i$ and the line $l_r$ is one diagonal over the diagonal passing through the north step at line $i_1$ in $(\tilde \gamma,\tilde w)$. We can assume $l$ is the smallest line satisfying these properties. Note that if $j=l$, then $r=s-1$ and if $j\not=l$, then $s=1$ (see Figure~\ref{Fig : p>2 j=l} and Figure~\ref{Fig : p>2 s=m}). Again,  $\tilde w_{l,r} < \tilde w_{i,1}$ or else we would have $d_{l_r}(i_1)=1$. This means $l\not=j$ and $l\not=j+1$, since $w_j>w_{i}>w_l$. So there must be at least one east step between the lines $j_1$ and ${j+1}_m$. If there is just one, then the line ${j+1}_m$ is on the same diagonal than the lines $i_1$ and $j_1$ contradicting  that $j$ is the upper bound. If there are two or more east steps between the lines $j_1$ and ${j+1}_m$, then the path goes under the diagonal passing through the north steps at the line $j_1$ and at the line $i_1$. But it must cross it again before the line $l_r$ because the path is continuous and the line $l_r$  is over the diagonal. Which contradicts again that $j$ is the upper bound.
\end{proof}
\begin{figure}[!htb]
\begin{minipage}{8cm}
\centering
\begin{tikzpicture}[scale=.75]
\draw (0,0)--(0,1);
\draw[dotted] (0,1)--(0,2);
\draw (0,2)--(0,4);
\draw (2,5)--(2,6)--(3,6);
\draw[dotted] (3,6)--(5,6);
\draw (5,6)--(6,6)--(6,7);
\draw[dotted] (6,7)--(6,8);
\draw[dashed,red] (-.5,2)--(6.5,9);
\draw (6,8)--(6,9);
\node(i1) at (6.5,8.75){$w_{i,1}$};
\node(js-1) at (2,3.5){$w_{j,s-1}$ or $w_{j+1,m}$};
\node(js) at (1.5,2.75){$w_{j,s}$ or $w_{j,1}$};
\end{tikzpicture}
\caption{}\label{Fig : p>2 j=l}
\end{minipage}
\begin{minipage}{8cm}
\centering
\begin{tikzpicture}[scale=.75]
\draw (0,0)--(0,1);
\draw[dotted] (0,1)--(0,2);
\draw (0,2)--(0,3.25)--(1,3.25);
\draw (2,6)--(2,7)--(3,7);
\draw[dotted] (3,7)--(5,7);
\draw (5,7)--(6,7)--(6,7.5);
\draw[dotted] (6,7.5)--(6,8);
\draw[dashed,red] (-.5,2)--(6.5,9);
\draw (6,8)--(6,9);
\draw (1.15,4.25)--(1.15,5.25)--(2.15,5.25);
\node(i1) at (6.5,8.75){$w_{i,1}$};
\node(lr) at (1.7,4.75){$w_{l,r}$};
\node(js) at (.6,2.75){$w_{j,1}$};
\end{tikzpicture}
\caption{}\label{Fig : p>2 s=m}
\end{minipage}
\end{figure}
We can now state a first condition.
\begin{lem}\label{Lem : read=w-1 m=1}Let  $(\gamma,w)$ be in $\Pa_{n,n}$. If $\dinv(\gamma,w)=0$, then $\R(\gamma,w)=w^{-1}$. Additionally, all non-trivial factors of $\gamma$, $\gamma'=NE^pN$ are such that $p=1$. 
\end{lem}
\begin{proof}If $p\leq 1$, then by Claim~\ref{Cl : min dinv Park}, for all factors $NE^pN$ of $\gamma$ we must have $p=1$ when $w_i>w_{i+1}$ and $p=0$ when $w_i<w_{i+1}$, since $\dinv(\gamma,w)=0$.
If $\gamma'=NE^pN$ is a factor of $\gamma$ associated to lines $i$ and $i+1$ such that and $p>1$, then, by Lemma~\ref{Lem : m=1 p>2 local}, there is $k$ such that $d_k(i+1)=1$. Therefore, $p\not>1$ because $\dinv(\gamma,w)=0$. Finally, $\R(\gamma,w)=w^{-1}$, is a direct consequence of $p\leq 1$.
\end{proof}
The same result is also true for general parking functions when $m\geq2$. The following gives somewhat of a generalization.
\begin{lem}\label{Lem : read=w-1}Let $a$ and $m$ be integers such that $2\leq a \leq m$ and  $(\gamma,w)$ be in $\Pa_{n,nm}$. If $\dinv(\gamma,w)=a-2$, then $\R(\gamma,w)=w^{-1}$. Additionally, all factors of $\gamma$, $\gamma'=NE^pN$ are such that $p\leq m$. 
\end{lem}
\begin{proof} By hypothesis $\dinv(\gamma,w)=a-2$. If $p > m$, then, by Lemma~\ref{Lem : dinv Park p>m}, $a-2\geq m-1$, which contradicts $a\leq m$. Hence, all factors $\gamma'=NE^pN$ of $\gamma$ are such that $p\leq m$. Thereafter, all factors $\gamma_{i,j}=NE^{p_i}NE^{p_{i+1}}\cdots NE^{p_{j-1}}N$, with $i<j$, satisfy $|\gamma_{i,j}|_E =\sum_{k=i}^{j-1}p_k\leq (j-i)m=m|\gamma_{i,j}|_N$. Consequently, for all $i<j$ we read $w_j$ before $w_i$ in $\R(\gamma,w)$ and $\R(\gamma,w)=w^{-1}$ as stated.
\end{proof}
Obviously, $\R(\gamma,w)\not=w^{-1}$ in general. But sometimes $\R(\gamma,w)=w^{-1}$ even without the condition on diagonal inversions. The last part of the proof gave us a weaker yet more general statement.
\begin{cl}\label{Cl : read=w-1} Let $m$ be an integer  and  $(\gamma,w)$ be in $\Pa_{n,nm}$. If all factors of $\gamma$, $\gamma'=NE^pN$ are such that $p\leq m$, then $\R(\gamma,w)=w^{-1}$.
\end{cl}
Sadly Lemma~\ref{Lem : read=w-1} does not apply for $m=2$, when the diagonal inversion statistic has value $1$. Therefore, we have to prove it separately.
\begin{lem}\label{Lem : read=w-1 m=2}Let $(\gamma,w)$ be in $\Pa_{n,2n}$. If $\dinv(\gamma,w)=1$, then $\R(\gamma,w)=w^{-1}$. Additionally, all factors of $\gamma$, $\gamma'=NE^pN$ are such that $p\leq 2$. 
\end{lem}
\begin{proof} Suppose there is a factor $NE^pN$ of $\gamma$ such that $p>2$, associated to lines ${j-1}$ and $j$. We can assume $j$ to be the smallest line satisfying that property. Let us consider $\tilde\gamma$ the path found by doubling each north step in $\gamma$ and $\tilde w$ the word $w_{1,1}w_{1,2}w_{2,1}w_{2,2}\cdots w_{n,1}w_{n,2}$. Considering the path is continuous and $p>2$, the path goes over the diagonal passing through $w_{j,1}$. The path ends under that diagonal, ergo there must be $i$ such that $w_{j,2}$ is on the same diagonal as $w_{i,1}$ or $w_{i,2}$ and $w_{i+1,1}$ is strictly over the diagonal passing through $w_{j,1}$. The three cases possible are illustrated by Figure~\ref{Fig : cas 1}, Figure~\ref{Fig : cas 2} and Figure~\ref{Fig : cas 3}.

For the first case, if $w_i <w_j$, then the pairs $(i_1,j_1)$ and $(i_2,j_2)$ both contribute to dinv. If $w_i >w_j$, then the pairs $(i+1_2,j_1)$ and $(i_1,j_2)$ both contribute to dinv. Thus, the diagonal inversion statistic cannot be equal to $1$.

For the second case, if $w_{i+1} <w_j$, then the pairs $(i+1_2,j_1)$ and $(i_1,j_2)$ both contribute to dinv, since $w_i<w_{i+1}$. If $w_{i+1} >w_j$, then the pairs $(i+1_1,j_1)$ and $(i+1_2,j_2)$ both contribute to dinv. Hence, the diagonal inversion statistic cannot be equal to $1$.

For the last case, if $w_i <w_j$, then the pairs $(i_1,j_1)$ and $(i_2,j_2)$ both contribute to dinv. If $w_i >w_j$, then the pairs $(i+1_2,j_2)$ and $(i_1,j_2)$ both contribute to dinv because $w_i<w_{i+1}$. Ergo, the diagonal inversion statistic cannot be equal to $1$.
Therefore, $p\leq 2$ for all factors $NE^pN$ of $\gamma$ and, by Claim~\ref{Cl : read=w-1}, we get $\R(\gamma,w)=w^{-1}$.
\end{proof}
\begin{figure}[!htb]
\begin{minipage}{5.5cm}
\centering
\begin{tikzpicture}[scale=.75]
\draw (1.4,1)--(1.4,4);
\draw (3.5,5.5)--(5.75,5.5)--(5.75,6.75);
\draw[->, red, dashed] (.75,1.25)--(6,6.5);
\draw[->, red,dashed] (.75,.65)--(6,5.9);
\node(st) at (6.2,6.5){$w_{j,1}$};
\node(ik) at (6.2,5.75){$w_{j,2}$};
\node(s1) at (2.2,3.5){$w_{i+1,1}$};
\node(so) at (2.2,2.75){$w_{i+1,2}$};
\node(sm) at (1.9,2){$w_{i,1}$};
\node(sn) at (1.9,1.25){$w_{i,2}$};
\end{tikzpicture}
\caption{Case 1}\label{Fig : cas 1}
\end{minipage}
\begin{minipage}{5.5cm}
\centering
\begin{tikzpicture}[scale=.75]
\draw (1.4,.25)--(1.4,3.25);
\draw (3.5,5.5)--(5.75,5.5)--(5.75,6.75);
\draw[->, red,dashed] (.75,1.25)--(6,6.5);
\draw[->, red,dashed] (.75,.65)--(6,5.9);
\node(st) at (6.2,6.6){$w_{j,1}$};
\node(ik) at (6.2,5.85){$w_{j,2}$};
\node(s1) at (2.2,2.85){$w_{i+1,1}$};
\node(so) at (2.2,2.1){$w_{i+1,2}$};
\node(sm) at (1.9,1.35){$w_{i,1}$};
\node(sn) at (1.9,.5){$w_{i,2}$};
\end{tikzpicture}
\caption{Case 2}\label{Fig : cas 2}
\end{minipage}
\begin{minipage}{5.5cm}
\centering
\begin{tikzpicture}[scale=.75]
\draw (1.4,1)--(1.4,2.5)--(2.4,2.5)--(2.4,4);
\draw (3.5,5.5)--(5.75,5.5)--(5.75,6.75);
\draw[->, red,dashed] (.75,1.25)--(6,6.5);
\draw[->, red,dashed] (.75,.65)--(6,5.9);
\node(st) at (6.2,6.5){$w_{j,1}$};
\node(ik) at (6.2,5.75){$w_{j,2}$};
\node(s1) at (3.2,3.7){$w_{i+1,1}$};
\node(so) at (3.2,3.1){$w_{i+1,2}$};
\node(sm) at (1.9,2.1){$w_{i,1}$};
\node(sn) at (1.9,1.35){$w_{i,2}$};
\end{tikzpicture}
\caption{Case 3}\label{Fig : cas 3}
\end{minipage}
\end{figure}
By definition, it is fairly easy to see that for $(\gamma,w)$ the descent set of $w$ is related to the number of columns of $\gamma$. We state the following claim is order to avoid repetition.
\begin{cl}\label{Cl : max descent par colonne}
Let $(\gamma,w)$ be in $\Pa_{n,nm}$, $1\leq m$. Then, the number of descents of $w$ plus 1 is smaller or equal to  the number of distinct columns. Additionally, the descents are at the top of a column.
\end{cl}
\begin{proof}If $w_i$ and $w_{i+1}$ are in the same columns, then by definition of parking functions we must have $w_i<w_{i+1}$. Therefore, descents must be at the top of a column. The last column cannot have descents, since the top of that column is $w_n$ and we know the last letter of a permutation cannot be a descent. 
\end{proof}
The last result relates the number of distinct columns to the descent set of $w$. But the following relates  the number of distinct columns to the descent set of $\R(\gamma,w)^{-1}$.
\begin{lem}\label{Lem : dinv vs top colonnes}Let $m$ and $n$ be integers, $(\gamma,w)$ be in $\Pa_{n,nm}$ and $T(\gamma)$ be the number of distinct columns. Let $\sigma$ be the permutation such that $\sigma.(\R(\gamma,w)^{-1})=w$. Then:
\begin{equation*}
 \dinv(\gamma,w)\geq \begin{cases}T(\gamma)-\dd(\R(\gamma,w)^{-1})-1 &\text{ if } \sigma(n)=n,
 \\      T(\gamma)-\dd(\R(\gamma,w)^{-1})-2 &\text{ if } \sigma(n)\not=n.
 \end{cases}
\end{equation*}
\end{lem}
\begin{proof}We will show that the letter at the top of a column contribute at least 1 to $\dinv$ unless they are in the descent set of $\R(\gamma,w)^{-1}$, in the last column, or the last letter of $\R(\gamma,w)^{-1}$. Notice that if $\sigma(n)=n$, then last letter of $\R(\gamma,w)^{-1}$ is in the last column, so we only need to subtract it once.
Let $\R(\gamma,w)^{-1}=v_1v_2\cdots v_n$ and let  $v_i$ be at the top of a column. If $v_i$ is not in the last column and $i\not=n$, then by definition of the reading word and because the path is continuous, we have , in $(\tilde\gamma,\tilde w)$, these three cases: the last copy from the top of $v_{i+1}$ is to the north and on the same diagonal than a copy of $v_i$, let us say the $k$-th copy from the top (see Figure~\ref{Fig : cas 1 dans read}),  the letter $v_{i+1}$ is to  south and on one of the diagonals crossing one of the $m-1$ first copies of  $v_i$ (see Figure~\ref{Fig : cas 2 dans read}), let us say the $p$-th copy, or $v_{i+1}$ is to the south one diagonal higher than the first copy of $v_i$ (see Figure~\ref{Fig : cas 3 dans read}). 
Let $\sigma$ be the permutation that send $\R(\gamma,w)^{-1}$ to $w$. When $i$ is not in the descent set of  $\R(\gamma,w)^{-1}$, our first case yields $d_{\sigma(i)}(\sigma(i+1)) = k$, by definition of the diagonal inversion statistics. The same reasoning shows $d_{\sigma(i+1)}(\sigma(i))= p+1$ for the second case and $d_{\sigma(i+1)}(\sigma(i)) = 1$ for the last case.
\end{proof}
\begin{figure}[!htb]
\begin{minipage}{5.5cm}
\centering
\begin{tikzpicture}[scale=.75]
\draw (1.6,.75)--(1.6,2.75)--(1.8,2.75);
\draw (5.75,5.5)--(5.75,7.5);
\draw[->, red,dashed] (.75,.75)--(6,6);
\draw[->, red,dashed] (1.5,.75)--(8,7.25);
\node(i1) at (7.1,7.25){$v_{i+1}$~copy $1$};
\node(ik) at (7.2,5.75){$v_{i+1}$~copy $m$};
\node(s1) at (2.75,2.5){$v_{i}$~copy $1$};
\node(st) at (2.75,1.75){$v_{i}$~copy $k$};
\node(sm) at (2.85,1){$v_{i}$~copy $m$};
\end{tikzpicture}
\caption{}\label{Fig : cas 1 dans read}
\end{minipage}
\begin{minipage}{5.5cm}
\centering
\begin{tikzpicture}[scale=.75]
\draw (1.4,1.75)--(1.4,3.75);
\draw (5.75,5.5)--(5.75,7.5);
\draw[->, red,dashed] (.75,1.25)--(6,6.5);
\draw[->, red,dashed] (.75,.75)--(8,8);
\node(i1) at (6.9,7.25){$v_{i}$~copy $1$};
\node(st) at (6.9,6.5){$v_{i}$~copy $p$};
\node(ik) at (7.,5.75){$v_{i}$~copy $m$};
\node(s1) at (2.75,3.5){$v_{i+1}$~copy $1$};
\node(sm) at (2.85,2){$v_{i+1}$~copy $m$};
\end{tikzpicture}
\caption{}\label{Fig : cas 2 dans read}
\end{minipage}
\begin{minipage}{5.5cm}
\centering
\begin{tikzpicture}[scale=.75]
\draw (1.4,1.75)--(1.4,3.75);
\draw (5.75,4)--(5.75,6);
\draw[->, red,dashed] (.75,1.25)--(6,6.5);
\draw[->, red,dashed] (.75,.75)--(7,7);
\node(i1) at (6.9,5.75){$v_{i}$~copy $1$};
\node(st) at (6.9,5){$v_{i}$~copy $p$};
\node(ik) at (7.,4.25){$v_{i}$~copy $m$};
\node(s1) at (2.75,3.5){$v_{i+1}$~copy $1$};
\node(sm) at (2.85,2){$v_{i+1}$~copy $m$};
\end{tikzpicture}
\caption{}\label{Fig : cas 3 dans read}
\end{minipage}
\end{figure}
We can now state necessary and sufficient conditions for the diagonal inversion statistic to be equal to zero.
\begin{prop}\label{Prop : critere dinv=0}Let $(\gamma,w)$ be in $\Pa_{n,nm}$, $1\leq m$. Then, $\dinv(\gamma,w)=0$ if and only if the following conditions applies:
\\
$\bullet$ The path $\gamma$ can be written as $\gamma'E^j$ where all factors of $\gamma'$ of length 2 have at most one east step.
\\
$\bullet$ If $\{i_1,\ldots,i_k,n\}$ is the set of all lines containing an east step, then $\{i_1,\ldots,i_k\}$ is the descent set.
\\
$\bullet$ $\R(\gamma,w)=w^{-1}$. 
\end{prop}
\begin{proof} If $\dinv(\gamma,w)=0$, then by Claim~\ref{Cl : min dinv Park} and Lemma~\ref{Lem : dinv Park p>m}, we have the first condition. If $\gamma'=NE^pN$ is a factor of $\gamma$ associated to $w_i$ and $w_{i+1}$, then by now proven first condition and  Claim~\ref{Cl : min dinv Park}, $w_i>w_{i+1}$. Hence, the position $i$ is a descent. But, by Claim~\ref{Cl : max descent par colonne}, we know that the number of descents plus $1$ is greater or equal to the number of columns and that $w_n$ contains an east step, ergo the second condition. The last condition follows from Lemma~\ref{Lem : read=w-1 m=1} and Lemma~\ref{Lem : read=w-1}. Therefore, $\dinv(\gamma,w)=0$ does imply the stated conditions. 

Conversely, by Claim  ~\ref{Cl : max descent par colonne}, the descents are at the top of each column, since $n$ cannot be a descent. Therefore, by the first condition and by Claim~\ref{Cl : min dinv Park}, we know that, for all lines $i$, with an east step, we are at the top of a column and $d_i(i+1)=0$. For all lines $i$ with an east step, and, all lines $j$ such that $j>i+1$, we know that  $a_i+(j-i)m\geq a_j\geq a_i+(j-i)(m-1)$, since there is at most one east step between each north step. This leads to:
\begin{equation*}
(j-i)m+1\geq |a_j-a_i+1|\geq (j-i)(m-1)+1\geq 2(m-1)+1\geq m,
\end{equation*}
and:
\begin{equation*}
(j-i)m\geq |a_j-a_i|\geq (j-i)(m-1)\geq 2(m-1)\geq m,
\end{equation*}
Therefore, $d_i(j)=0$ for all $i$ and $j$. Hence, $\dinv(\gamma,w)=0$. 
\end{proof}
Note that the second statement of the previous proposition implies that the number of descents in $w$ is equal to the number of distinct columns plus $1$. 
Looking at the specialization $q=0$ is equivalent to looking only at the parking functions $(\gamma,w)$ such that $\dinv(\gamma,w)=0$, and, thus, we need the area of theses parking functions.
\begin{prop}\label{Prop : aire parking}Let $(\gamma,w)$ be in $\Pa_{n,nm}$, $1\leq m$ such that $\dinv(\gamma,w)=0$, then:
\begin{equation*} 
\A(\gamma,w)=m\binom{n}{2}-\dd(w)n+\m(w)
\end{equation*}
\end{prop}
\begin{proof}Let $\{ i_1,\ldots,i_k,n\}$ be the lines with east steps. We know that these are tops of columns and by the previous proposition we know that $\{ i_1,\ldots,i_k\}$ is the descent set. By the previous proposition, we also know that:
\begin{align*}
\A(\gamma,w) &=m\binom{n}{2}-\sum_{j=1}^k (n-i_j)
\\   &=m\binom{n}{2}-n\sum_{j=1}^k 1+\sum_{j=1}^ki_j
\\   &=m\binom{n}{2}-\dd(w)n+\m(w)
\end{align*}
\end{proof}
This last proposition allows us to give a proper formula for $\nabla^m(e_n)|_{q=0}$. Which is just an extension of the Stanley-Lusztig formula.
\begin{prop}\label{Prop : 1 part parking} For integers $n$, $m$ such that $1\leq n,m$
\begin{equation}\label{Eq : 1 part parking}\nabla^m(e_n)|_{1 \pp}=\sum_{\tau\in \SYT(n)}s_{\m(\tau)+(m-1)\binom{n}{2}}(q,t) s_{\lambda(\tau)}(X), \text{and},
\end{equation}
\begin{equation}\label{Eq : q=0 parking}\nabla^m e_n|_{q=0}=t^{(m-1)\binom{n}{2}}\sum_{w\in \mathbb{S}_n} t^{\binom{n}{2}-\dd(w)n+\m(w)}F_{\co(\DD(\inv(w^{-1})))}(X)=t^{(m-1)\binom{n}{2}}\sum_{\tau\in \SYT(n)} t^{\m(\tau)} s_{\lambda(\tau)}.
\end{equation}
\end{prop}
\begin{proof} The first equality of Equation~\eqref{Eq : q=0 parking} follows from Proposition~\ref{Prop : aire parking} and Proposition~\ref{Prop : critere dinv=0}. Consequently, by the Equation~inferred from \cite{[S1979]} and \cite{[H2002]} (see Equation~\eqref{Eq : S-L,H}), we have:
\begin{equation*}\sum_{w\in \mathbb{S}_n} t^{\binom{n}{2}-\dd(w)n+\m(w)}F_{\co(\DD(\inv(w^{-1})))}(X)=\sum_{\tau\in \SYT(n)} t^{\m(\tau)} s_{\lambda(\tau)}(X).
\end{equation*}
Therefore, the second equality of Equation~\eqref{Eq : q=0 parking} holds.  For Equation~\eqref{Eq : 1 part parking}, we only need to notice that $\nabla^m(e_n)$ is symmetric in $q$,$t$ and $s_{\lambda}(q,t)=0$ if $\ell(\lambda)>2$ and $s_{a,b}(q,t)=(qt)^b(q^{a-b}+q^{a-b-1}t+\cdots +qt^{a-b-1}+t^{a-b}$. Hence, $s_{a,b}(0,t)=0$ if $b\not=0$ and $s_{a}(0,t)=t^a$. Ergo, we have the stated result by Equation~\eqref{Eq : q=0 parking}.
\end{proof}
From this last proposition and Proposition~\ref{Prop : sum algo 1 part} we can obtain the following $q$-analogues that are used in \cite{[Wal2019d]}.
\begin{cor}Let $n$, $m$, $d$ be integers, then:
\begin{align*}
\langle\nabla^m(e_n), s_{d+1,1^{n-d-1}}\rangle|_{t=0}&=q^{(m-1)\binom{n}{2}+\binom{n-d}{2}}\begin{bmatrix} n-1\\d\end{bmatrix}_q=q^{m\binom{n}{2}-\binom{d+1}{2}}\begin{bmatrix} n-1\\d\end{bmatrix}_{q^{-1}}
\end{align*}
\end{cor}
\begin{proof}We recall that $\begin{bmatrix}n\\k\end{bmatrix}_q=\sum_{\gamma\in \mathcal{C}^n_k}q^{\A(\gamma)}$. In consequence, by Proposition~\ref{Prop : sum algo 1 part} and  Proposition~\ref{Prop : 1 part parking}, we have the first equality. The second equality follows from $\binom{n}{2}-\binom{d+1}{2}=d(n-d-1)+\binom{n-d}{2}$ and $q^{d(n-d-1)}\begin{bmatrix} n-1\\d\end{bmatrix}_{q^{-1}}=\begin{bmatrix} n-1\\d\end{bmatrix}_q$.
\end{proof}
This next lemma emulates Proposition~\ref{Prop : critere dinv=0} for diagonal inversion statistics values of one.
\begin{lem}\label{Lem : m>2 dinv=1}
Let $(\gamma,w)$ be in $\Pa_{n,nm}$, $2\leq m$ and $T(\gamma)$ be the number of distinct columns. Then, $\dinv(\gamma,w)=1$ if and only if one of the following conditions applies:
\\
$\bullet$ All factors, $NE^pN$ of $\gamma$ are such that $p\leq1$ and $T(\gamma)=\dd(w)+2$. 
\\
$\bullet$ Exactly one factor, $NE^pN$ of $\gamma$ is such that $p=2$  all other such factors satisfy $p\leq1$ and $T(\gamma)=\dd(w)+1$.
\end{lem}
\begin{proof}We will start by proving the statement for $m\geq3$.
If $\dinv(\gamma,w)=1$, by Lemma~\ref{Lem : read=w-1}, we have $\R(\gamma,w)=w^{-1}$. Thus,by Lemma~\ref{Lem : dinv vs top colonnes} and Claim~\ref{Cl : max descent par colonne}, $\dd(w)+1\leq T(\gamma)\leq \dd(w)+2$.
 
By Claim~\ref{Cl : max descent par colonne}, if $T(\gamma)=\dd(w)+1$, then $w_i>w_{i+1}$ for all $i$ at the top of a column. Hence, by Lemma~\ref{Lem : read=w-1}, $p\leq m$. In consequence, by Claim~\ref{Cl : min dinv Park}, there is exactly one factor, $NE^pN$ of $\gamma$ is such that $p=2$,  all other such factors satisfy $p\leq1$.  
Again, by Claim~\ref{Cl : max descent par colonne}, if $T(\gamma)=\dd(w)+2$, there is exactly one line $i$ containing an east step such that $w_i<w_{i+1}$. Ergo, by Claim~\ref{Cl : min dinv Park}, all factors, $NE^pN$ of $\gamma$ are such that $p\leq1$.

If all factors, $NE^pN$ of $\gamma$ are such that $p\leq1$ and $T(\gamma)=\dd(w)+2$, then, by Claim~\ref{Cl : max descent par colonne}, we know there is exactly one $i$ at the top of a column such that $w_i<w_{i+1}$. Moreover, $p\leq1$ for all factors $NE^pN$ of $\gamma$, so, by Claim~\ref{Cl : min dinv Park}, $\dinv(\gamma,w)=1$, if $d_i(j)=0$ for all $j\geq i+2$. But $p\leq 1$ means $m\geq a_{i+1}-a_i\geq m-1$. Hence, for $j>i$, $|a_j-a_i|\geq (m-1)(j-i)\geq m$ and $|a_j-a_i+1|\geq (m-1)(j-i)+1\geq m$, since $j-i\geq 2$ and $m\geq 3$. Consequently, $d_i(j)=0$ for all $j\geq i+2$ and $\dinv(\gamma,w)=1$.

If exactly one factor, $NE^pN$ of $\gamma$ is such that $p=2$  all other such factors satisfy $p\leq1$ and $T(\gamma)=\dd(w)+1$. By Claim~\ref{Cl : max descent par colonne}, for all $i$ at the top of a column $w_i>w_{i+1}$. Additionally, exactly one factor, $NE^pN$ of $\gamma$ is such that $p=2$,  all other such factors satisfy $p\leq1$. In consequence,  by Claim~\ref{Cl : min dinv Park}, $\dinv(\gamma,w)=1$ if $d_i(j)=0$ for all $j\geq i+2$. But $p\leq 1$ means $m\geq a_{i+1}-a_i\geq m-1$ and $p=2$ means $m\geq a_{i+1}-a_i\geq m-2$. Ergo, for $j>i$, $|a_j-a_i+1|\geq (m-1)(j-i)+1\geq m$ because $j-i\geq 2$ and $m\geq 3$. Thus, $d_i(j)=0$ for all $j\geq i+2$ and $\dinv(\gamma,w)=1$.

For $m=2$ the proof is the same, we only need to change references of Lemma ~\ref{Lem : read=w-1} to Lemma ~\ref{Lem : read=w-1 m=2}.
\end{proof}
Note that 
for $m=1$ nothing holds (see Figure~\ref{Fig : => contre-exemple m=1}, Figure~\ref{Fig : <= contre-exemple m=1 partie 1} and Figure~\ref{Fig : <= contre-exemple m=1 partie 2}) but,  in Section~\ref{Sec : bijections}, we manage to obtain a Proposition~\ref{Prop : 1 part parking} type formula  for the restriction to hook Schur functions in the variables $q$ and $t$, by using Schr\"oder paths and the $\B$ statistic.
\begin{figure}[!htb]
\begin{minipage}{8.5cm}
\centering
\begin{tikzpicture}[scale= .5]
\draw (0,0)--(0,3)--(3,3)--(3,4)--(4,4)--(0,0);
\draw[dashed,red] (0,.5)--(3,3.5);
\draw[gray] (1,1)--(1,3);
\draw[gray] (0,2)--(2,2);
\draw[gray] (2,2)--(2,3);
\draw[gray] (0,1)--(1,1);
\node(1) at (-1,.5){$1$};
\node(2) at (-1,1.5){$2$};
\node(3) at (-1,2.5){$3$};
\node(4) at (2.5,3.5){$4$};
\end{tikzpicture}
\caption{The diagonal inversion statistic is $1$ yet $p>2$.} \label{Fig : => contre-exemple m=1}
\end{minipage}
\begin{minipage}{8.5cm}
\centering
\begin{tikzpicture}[scale= .5]
\draw (0,0)--(0,2)--(1,2)--(1,3)--(2,3)--(2,4)--(4,4)--(0,0);
\draw[dashed,red] (0,1.7)--(2,3.7);
\draw[dashed,red] (0,1.3)--(1,2.3);
\draw[gray] (0,1)--(1,1)--(1,2)--(2,2)--(2,3)--(3,3)--(3,4);
\node(1) at (-.5,.5){$1$};
\node(2) at (-.5,1.5){$2$};
\node(4) at (.5,2.5){$4$};
\node(3) at (1.5,3.5){$3$};
\end{tikzpicture}
\caption{The diagonal inversion statistic is $2$ yet $p\leq1$ and $T(\gamma)=\dd(1243)+2=3$.} \label{Fig : <= contre-exemple m=1 partie 1}
\end{minipage}
\end{figure}
\begin{figure}[!htb]
\centering
\begin{tikzpicture}[scale= .5]
\draw (0,0)--(0,2)--(2,2)--(2,3)--(3,3)--(3,4)--(4,4)--(0,0);
\draw[dashed,red] (0,1.7)--(3,3.7);
\draw[dashed,red] (0,1.2)--(2,2.5);
\draw[gray] (0,1)--(1,1)--(1,2)--(2,2)--(2,3)--(3,3)--(3,4);
\node(1) at (-.5,.5){$1$};
\node(4) at (-.5,1.5){$4$};
\node(3) at (1.5,2.5){$3$};
\node(2) at (2.5,3.5){$2$};
\end{tikzpicture}
\caption{The diagonal inversion statistic is $2$ yet exactly one factor $NE^pN$ is such that $p=2$ all others are such that $p\leq1$ and $T(\gamma)=\dd(1432)+1=3$.} \label{Fig : <= contre-exemple m=1 partie 2}
\end{figure}
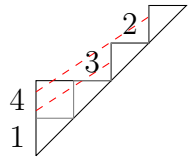
%%%%%%%%%%%%%%%%%%%%%%
%%%%%  Section Schroder
%%%%%%%%%%%%%%%%%%%%%%
\section{Restriction to $m$-Schr\"oder Paths}\label{Sec : schroder}
This section is dedicated to the restriction to Schr\"oder paths. With this restriction we can give the necessary and sufficient conditions for a path to have a diagonal inversion statistic value of one.
Proposition~\ref{Prop : critere dinv=0} can be restated in terms of Schr\"oder paths.
\begin{cor}\label{Cor : Schroder->des=d}Let $(\gamma,w)$ be in $\sch_{n,d}^{(m)}$, $1\leq m$. Then, $\dinv(\gamma,w)=0$ if and if, one of the following conditions applies:
\\
$\bullet$ The path $\gamma$ can be written as $\gamma'E^j$ where all factors of $\gamma'$ of length 2 have at most one east step.
\\
$\bullet$ If $\{i_1,\ldots,i_k,n\}$ is the set of all lines containing an east step, then $\{w_{i_1},\ldots,w_{i_k}\} \subseteq \{n-d+1,\ldots,n\}$.
\\
$\bullet$ $\R(\gamma,w)=w^{-1}$. 
\end{cor}
\begin{proof} Condition one and three are consequences of Proposition~\ref{Prop : critere dinv=0}. By definition of Schr\"oder paths $\R(\gamma,w) \in  \{n-d+1,\ldots,n\} \shuffle \{n-d,\ldots,1\}$. Therefore, $\R(\gamma,w)^{-1}=w\in  \{n,\ldots,n-d+1\} \shuffle \{1,\ldots,n-d\}$ and the descents of $w$ are the positions of $n-d+1,\ldots,n$ in $w$. Hence, the result follows from Proposition~\ref{Prop : critere dinv=0}.
\end{proof}
The restriction to $m$-Schr\"oder paths allow us to write the formula of Proposition ~\ref{Prop : 1 part parking} in terms of paths in a rectangular grid as we did for $m=1$ in  Proposition~\ref{Prop : sum algo 1 part}.
\begin{cor} Let $n$, $d$ and $m$ be positive integer such that $n\geq d$ and $m>1$. Then:
\begin{align*}{\sch}^m_{n,d}(q,t)|_{1 \pp}=&\langle\nabla^m e_n,e_{n-d}h_d \rangle|_{1 \pp}
\\      =&\sum_{\gamma\in \mathcal{C}^{n-1}_{d-1}} s_{(m\binom{n}{2}-\binom{d}{2}-\A(\gamma))}(q,t)+\sum_{\gamma\in \mathcal{C}^{n-1}_{d}} s_{(m\binom{n}{2}-\binom{d+1}{2}-\A(\gamma))}(q,t)
\end{align*}
Additionally:
\begin{equation}\label{Eq : 1 part m-schroder} {\sch}^m_{n,d}(q,t)|_{1 \pp}=\langle\nabla^m e_n,s_{d+1,1^{n-d-1}} \rangle|_{1 \pp}
=\sum_{\gamma\in \mathcal{C}^{n-1}_{d}} s_{(m\binom{n}{2}-\binom{d+1}{2}-\A(\gamma))}(q,t)
\end{equation}
\end{cor}
\begin{proof}Due to $\sum_{\gamma\in \mathcal{C}^{n-1}_{d}}\A(\gamma)=\sum_{\gamma\in \mathcal{C}^{n-1}_{d}} (n-1-d)d-\A(\gamma)$ and $\binom{n}{2}-\binom{d+1}{2}=\binom{n-d}{2}+(n-1-d)d$, the result follows from Proposition~\ref{Prop : sum algo 1 part} and Proposition~\ref{Prop : 1 part parking}.
\end{proof}
The main proof of this section is very technical case-by-case proof. It will be used to prove the main theorem via Corollary~\ref{Cor : m>=1 dinv=1}.
\begin{prop}\label{Prop : m=1 dinv=1}
Let $(\gamma,w)$ be in $\sch_{n,d}$ and let $T(\gamma)$ be the number of distinct columns. Then, $\dinv(\gamma,w)=1$ if and only if one of the following conditions applies:
\\
$\bullet$ All factors, $NE^pN$ of $\gamma$ are such that $p\leq1$ and $T(\gamma)=\dd(w)+2$. 
\\
$\bullet$ Exactly one factor, $\gamma'=NE^pN$ of $\gamma$ is such that $p=2$, and all other such factors satisfy $p\leq1$  and $T(\gamma)=\dd(w)+1$.
\\
$\bullet$ Exactly one factor, $\gamma'=NE^pN$ of $\gamma$ is such that $p > 2$ and $\gamma'$ is associated to lines $n-1$, $n$, $w_{n}\in  \{n-d+1,\ldots,n\}$ and all other such factors satisfy $p\leq1$ and $T(\gamma)=\dd(w)+2$ if $w_{n-1}\in  \{n-d,\ldots,1\}$ or $T(\gamma)=\dd(w)+1$ if $w_{n-1}\in  \{n-d+1,\ldots,n\}$.
\end{prop}
\begin{proof}If $\dinv(\gamma,w)=1$ and all factors $\gamma'=NE^pN$ are  such that $p\leq 1$, then, by Claim~\ref{Cl : min dinv Park}, there is exactly one line $i$ at the top of a column such that $w_i<w_{i+1}$. Additionally, by Claim~\ref{Cl : max descent par colonne}, we know that  $T(\gamma)\geq\dd(w)+1$. By Claim~\ref{Cl : read=w-1} and Lemma~\ref{Lem : dinv vs top colonnes}, we have $T(\gamma)\leq\dd(w)+2$. Since $w_n$ is at the top of its own column, $i$ and $n$ are not  descents and the top of all the other columns are descents. Thus, $T(\gamma)=\dd(w)+2$. 

For the remaining cases, if $\dinv(\gamma,w)=1$, then, by Lemma~\ref{Lem : m=1 p>2 local}, we know there is at most one factor of $\gamma$ say $\gamma'=NE^pN$ associated to the lines $i$ and $i+1$ such that $p >1$ and there is $k$ such that $d_k(i+1)=1$.

If $w_k>w_{i+1}$,  then the north step at line $k$ is on the diagonal above the north step at the line $i+1$. Moreover, the path is continuous and the north step at line $i$ is over the diagonal passing through the north steps at the line $k$ and $i+1$, and, thus,  there exist $l<k$ such that the north step at the line $i+1$ and the north step at the line $l$ are on the same diagonal. Assuming $l$ is the biggest such $l$. We know, $w_l>w_{i+1}$ because $d_l(i+1)=0$. This means $w_k$ is read before $w_{i+1}$ and $w_{i+1}$ is read before $w_l$. By definition of Schr\"oder paths $\R(\gamma,w)\in \{n-d+1,\ldots,n\}\shuffle \{n-d,\ldots,1\}$. Hence, $w_l\in\{n-d+1,\ldots,n\}$. There is at most one east step between the line $l$ and the line $l+1$, so $w_{l+1}$ is read before $w_{l}$. Ergo, $w_l>w_{l+1}$. Therefore, $w_{l+1}$ is not in the same column as $w_l$. Consequently,  $w_{l+1}$ is on the same diagonal as $w_l$, since there is at most one factor $NE^rN$ with $r>1$. This contradicts that $l$ is the highest line such that $l<k$, $l$ and $i+1$ are on the same diagonal. (See Figure~\ref{Fig : k diago au-dessus i+1}.)

If $w_k<w_{i+1}$, $p\geq 2$ and there is an east step between the lines $k$ and $k+1$, then $d_k(i+1)=1$ implies the lines $k$ and $i+1$ are crossed by the same diagonal. By Lemma~\ref{Lem : m=1 p>2 local}, there is exactly one east step between the lines $k$ and $k+1$, and, thus, they are on the same diagonal and $k\not=i$. Considering $\dinv(\gamma,w)=1$ we need $d_k(k+1)=0$ and $d_{k+1}(i+1)=0$. Therefore, $w_k>w_{k+1}$ and $w_{k+1}>w_{i+1}$ which is absurd. (See Figure~\ref{Fig : k,k+1,i meme diago}.)

For the case $w_k<w_{i+1}$, $p\geq 2$, $k=i-1$ and there is no east step between the lines $k$ and $k+1$. Notice that if $k=i-1$, there are $i+1-k=2$ north steps. Since $d_k(i+1)=1$, we know $k$ and $i+1$ are on the same diagonal. Hence, there is 2 east step between the north step at the line $k$ and the north step at the line $i+1$. In addition, letters on the same diagonal are separated by the same number of east steps than north steps. Thus, $p=2$.
Additionally, by Claim~\ref{Cl : max descent par colonne},  $T(\gamma)\geq\dd(w)+1$. Moreover, $w_i$ is read before $w_{i+1}$, since they are separated by more than one east step and $d_i(i+1)=0$. Thus,$w_i>w_{i+1}$. Furthermore, all descent in $w$ contribute to a different column. Hence, if $T(\gamma)>\dd(w)+1$ we must have a change of column at a line $l$, $l\not=i$, such that $l$ is not a descent. Ergo $w_l<w_{l+1}$ and because there is at most one east step between $w_l$ and $w_{l+1}$, by Lemma ~\ref{Lem : m=1 p>2 local}, we must have $d_l(l+1)=1$ which is absurd. Therefore, $T(\gamma)=\dd(w)+1$.

If $w_k<w_{i+1}$, $p\geq 2$, $k\not=i-1$ and there is no east step between the lines $k$ and $k+1$, then $k\not=i$ and $w_k<w_{k+1}$. Moreover, $d_k(i+1)=1$ implies the lines $k$ and $i+1$ are crossed by the same diagonal. Consequently, the north step at line $k+1$ is on the diagonal over the north step at  the line $i+1$. Hence, $w_{k+1}<w_{i+1}$, since  $d_{k+1}(i+1)=0$. Thus, $w_{k+1}$ is read before $w_{i+1}$ and $w_{i+1}$ is read before $w_k$ in $\R(\gamma,w)$. We know $\R(\gamma,w)\in \{n-d+1,\ldots,n\}\shuffle\{n-d,\ldots,1\}$, ergo, $w_{i+1}\in  \{n-d+1,\ldots,n\}$ and $w_{k+1},w_k \in  \{n-d,\ldots,1\}$. If $i+1\not=n$ and  $w_{i+2}$ is in the same column as $w_{i+1}$, then $w_{i+2}$ is read before $w_{i+1}$ and $w_{i+1}<w_{i+2}$. But this is impossible because $w_{i+1}$ is in the set $ \{n-d+1,\ldots,n\}$. Therefore, $w_{i+2}$ and $w_{i+1}$ are on the same diagonal and $w_{i+1}>w_{i+2}$, since $d_{i+1}(i+2)=0$. Due to $\dinv(\gamma,w)=1$ and $d_k(i+1)=1$, we have $d_{k}(i+2)=0$. The north step at line $i+2$ is on the same diagonal as the north step at the line $k$, ergo $w_k>w_{i+2}$.  For this reason, $w_{i+2}$ is read before $w_k$. But, $w_k \in  \{n-d,\ldots,1\}$ means $w_{i+2}>w_k$ which is impossible. So, $i+1=n$. (See Figure~\ref{Fig : k sous k+1 et k, i+1 meme diago}.)
\begin{figure}[!htb]
\begin{minipage}{5.5cm}
\centering
\begin{tikzpicture}[scale=.5]
\draw (0,0)--(0,1)--(1,1);
\draw[dotted] (1,1)--(2,1);
\draw (2,1)--(3,1)--(3,2);
\draw[red,dotted] (3.5,2)--(-2,-3.5);
\draw[red,dotted] (3.5,3)--(-1,-1.5);
\draw (-.5,-1.5)--(-.5,-.5);
\draw (-1.5,-3.5)--(-1.5,-2.5);
\node(i+1) at (4,1.5){$w_{i+1}$};
\node(i) at (0.5,0.5){$w_i$};
\node(k) at (0.1,-1){$w_k$};
\node(l) at (-.9,-3){$w_l$};
\end{tikzpicture}
\caption{ } \label{Fig : k diago au-dessus i+1}
\end{minipage}
\begin{minipage}{5.5cm}
\centering
\begin{tikzpicture}[scale=.5]
\draw (0,0)--(0,1)--(1,1);
\draw[dotted] (1,1)--(2,1);
\draw (2,1)--(3,1)--(3,2);
\draw[red,dotted] (3.5,2)--(-2,-3.5);
\draw (-1.5,-3.5)--(-1.5,-2.5)--(-.5,-2.5)--(-.5,-1.5);
\node(i+1) at (4,1.5){$w_{i+1}$};
\node(i) at (0.5,0.5){$w_i$};
\node(k) at (-0.9,-3){$w_k$};
\node(k+1) at (.5,-2){$w_{k+1}$};
\end{tikzpicture}
\caption{} \label{Fig : k,k+1,i meme diago}
\end{minipage}
\begin{minipage}{5.5cm}
\centering
\begin{tikzpicture}[scale=.5]
\draw (0,0)--(0,1)--(1,1);
\draw[dotted] (1,1)--(2,1);
\draw (2,1)--(3,1)--(3,2);
\draw[red,dotted] (3.5,2)--(-2,-3.5);
\draw[red,dotted] (3.5,3)--(-2,-2.5);
\draw (-1.5,-3.5)--(-1.5,-1.5);
\node(i+1) at (4,1.5){$w_{i+1}$};
\node(i) at (0.5,0.5){$w_i$};
\node(k) at (-0.9,-3){$w_k$};
\node(k+1) at (-.5,-2){$w_{k+1}$};
\end{tikzpicture}
\caption{} \label{Fig : k sous k+1 et k, i+1 meme diago}
\end{minipage}
\end{figure}

If $w_k<w_{i+1}$, $i+1=n$, $p>2$ and there is no east step between the lines $k$ and $k+1$, then, by Claim~\ref{Cl : max descent par colonne},  $T(\gamma)\geq\dd(w)+1$. As in the previous case $w_{i+1}\in  \{n-d+1,\ldots,n\}$. Additionally, $w_i$ is read before $w_{i+1}$, considering they are separated by more than one east step. Hence, $w_i<w_{i+1}$. Furthermore, all descent in $w$ contribute to a different column and the letters $w_{n-1}$, $w_{n}$ are in different columns (recall $i+1=n$ and there are $p$ east steps between $w_{n-1}$ and $w_n$). This means $T(\gamma)\geq\dd(w)+2$ if $w_{n-1}\in  \{n-d,\ldots,1\}$ and $T(\gamma)\geq\dd(w)+1$ if $w_{n-1}\in  \{n-d+1,\ldots,n\}$. In both cases if the inequality is strict, we must have a change of column at a line $l$ such that $l\not=n-1$ and $l$ is not a descent, ergo, $w_l<w_{l+1}$. Since there is at most one east step between $w_l$ and $w_{l+1}$, by Lemma ~\ref{Lem : m=1 p>2 local}, we must have $d_l(l+1)=1$ which is absurd. Therefore, $T(\gamma)=\dd(w)+2$ if $w_{n-1}\in  \{n-d,\ldots,1\}$ and $T(\gamma)=\dd(w)+1$ if $w_{n-1}\in  \{n-d+1,\ldots,n\}$.

Conversely, if all factors, $NE^pN$ of $\gamma$ are such that $p\leq1$ and $T(\gamma)=\dd(w)+2$, then there is exactly one line $i$ at the top of a column such that $i\not=n$ and $w_i<w_{i+1}$.  By Claim~\ref{Cl : min dinv Park}, $d_i(i+1)=1$ and $d_l(l+1)=0$ for all $l\not=i$ because $p\leq 1$. For the same reason, all lines $j$ and $k$ such that $k-j\geq 2$, have a number of north steps greater or equal to the number on east steps between them and $d_j(k)=0$, unless the north step on lines $j$ and $k$ are on the same diagonal. 
 But when the north step at the line $k$ and the north step at the line $j$ are on the same diagonal, $k>j$ and $p\leq 1$ we know the line $j$ is associated to a factor $NEN$ of $\gamma$ and is at the top of a column.  By Claim~\ref{Cl : read=w-1}, we have $\R(\gamma,w)=w^{-1}$, and, thus, $w_j\in \{n-d+1,\ldots,n\}$ and $w_k$ is read before $w_j$. Consequently, $w_j>w_k$ and $\dinv(\gamma,w)=1$.

If exactly one factor, $\gamma'=NE^pN$ of $\gamma$ is such that $p=2$, and all other such factors satisfy $p\leq1$ and $T(\gamma)=\dd(w)+1$, then by Claim~\ref{Cl : max descent par colonne}, all lines $j$ at the top of a column are such that $j$ is in the descent set of $w$. Hence, if $w_i$ and $w_{i+1}$ are associated to the factor $\gamma'=NE^2N$ of $\gamma$ $w_i$ is on the diagonal above $w_{i+1}$ and $w_i>w_{i+1}$, so $d_i(i+1)=1$. For the same reasons as in the previous case for all $j$ and $k$ such that $j\not=i$ and $k\not=i+1$, then $d_j(k)=0$. Therefore, $\dinv(\gamma,w)=1$.

Let us now consider the case when exactly one factor, $\gamma'=NE^pN$ of $\gamma$ is such that $p > 2$ and $\gamma'$ is associated to lines $n-1$, $n$ and all other such factors satisfy $p\leq1$ and $T(\gamma)=\dd(w)+2$. If we take out the last north step and the last east step and call the new path $\tilde\gamma$, then, by Proposition~\ref{Prop : critere dinv=0}, $\dinv(\tilde\gamma,w_1\cdots w_{n-1})=0$. Hence, for all $1\leq j<k\leq n-1$ we have $d_j(k)=0$. Because $d_j(k)$ is a local property, it is also true for $(\gamma,w)$. By continuity of the path, since $p>2$, there is a line $k$ such that $w_k$ and $w_{k+1}$ are in the same column and the north step a line $n$ is on the same diagonal than the north step at the line $k$. Thus, we read $w_{k+1}$ before $w_n$ and $w_n$ before $w_k$ in $\R(\gamma,w)$. Moreover, $w_{k+1}>w_k$, and, therefore, $w_k\in\{n-d,\ldots,1\}$ and $w_n>w_{k+1}>w_k$, is a consequence of $w_n\in\{n-d+1,\ldots,n\}$. So, $d_k(n)=1$ and $d_{k+1}(n)=0$. 
All other factors $\gamma''=NE^{p'}N$ satisfy $p'\leq1$, if there is $j$ distinct from $k$ such that $w_j$ is on the same diagonal than $w_n$, then $w_j,w_{j+1},\ldots,w_k$ are all on the same diagonal and  $w_j,w_{j+1},\ldots,w_{k-1}$ are at the top of their column. Furthermore, the condition $T(\gamma)=\dd(w)+2$ if $w_{i+1}\in  \{n-d,\ldots,1\}$ or $T(\gamma)=\dd(w)+1$ if $w_{i+1}\in  \{n-d+1,\ldots,n\}$ forces all letters of $w$ at the top of a column except for  $w_{n-1}$ and $w_n$ to be in  $\{n-d+1,\ldots,n\}$. Consequently, $w_j>w_{j+1}>\cdots>w_{k-1}>w_n>w_k$ and $d_l(n)=0$, for all $j\leq l\leq k-1$. Therefore, $\dinv(\gamma,w)=1$.
\end{proof}
The next corollary can also be deduced from the more general Lemma~\ref{Lem : m>2 dinv=1}. We only state it here, so one can notice that Proposition~\ref{Prop : m=1 dinv=1} is hiding a general statement  for $\sch_{n,d}^{(m)}$.  
\begin{cor}\label{Cor : m>1 dinv=1}
Let $m$ be an integer such that $m\geq 2$, $(\gamma,w)$ be in $\sch_{n,d}^{m}$ and let $T(\gamma)$ be the number of distinct columns. Then, $\dinv(\gamma,w)=1$ if and only if one of the following conditions applies:
\\
$\bullet$ All factors, $NE^pN$ of $\gamma$ are such that $p\leq1$ and $T(\gamma)=\dd(w)+2$. 
\\
$\bullet$ Exactly one factor, $\gamma'=NE^pN$ of $\gamma$ is such that $p=2$, and all other such factors satisfy $p\leq1$  and $T(\gamma)=\dd(w)+1$, then $\dinv(\gamma,w)=1$.
\end{cor}
\begin{proof}The proof of the previous proposition can be extended to $\tilde\gamma$, since $w_i=w_j$ only if they are in the same column. 
\end{proof}
The following is the restriction to unlabelled Dyck paths.
\begin{cor}\label{Cor : m>=1 dinv=1}
Let $m$ be an integer such that $m\geq 1$, $(\gamma,w)$ be in $\sch_{n,0}^{m}$ and let $T(\gamma)$ be the number of distinct columns. Then, $\dinv(\gamma,w)=1$ if and only if one of the following conditions applies:
\\
$\bullet$ All factors, $NE^pN$ of $\gamma$ are such that $p\leq1$ and $T(\gamma)=\dd(w)+2$. 
\\
$\bullet$ Exactly one factor, $\gamma'=NE^pN$ of $\gamma$ is such that $p=2$, and all other such factors satisfy $p\leq1$  and $T(\gamma)=\dd(w)+1$, then $\dinv(\gamma,w)=1$.
\end{cor}
\begin{proof}For $m=1$, the proof is a direct consequence of $\R(\gamma,w)=n\cdots1$ and Proposition~\ref{Prop : m=1 dinv=1}. For $m>1$ the proof follows from the last corollary.

Recall, from the proof of Corollary~\ref{Cor : Schroder->des=d}, that when $m\geq3$ a path $(\gamma,w)$ of $\sch_{n,d}{(m)}$ is such that $\dd(w)=d$. Hence the last two corollaries can be stated with nicer formulas.

\end{proof}
%%%%%%%%%%%%%%%%%%%%%%%%%%%%%%%%%%%%%%%%%%%%%%%
%%%%%%%%%%%%%%%            Bijections with tableaux
%%%%%%%%%%%%%%%%%%%%%%%%%%%%%%%%%%%%%%%%%%%%%%%
\section{Bijections With Tableaux}\label{Sec : bijections}
From Equation~\eqref{Eq : q=0 parking} of Proposition~\ref{Prop : 1 part parking}, one could wonder what tableau is associated to what path. In this section, we will first show a bijection between standard Young tableaux of shape $(d,1^{n-d})$ and the subset of Schr\"oder paths $\{\gamma \in \schP_{n,d-1} ~|~\A(\gamma)=0\}$. Afterwards, we exhibit a bijection between the set of paths $\{\gamma \in \schP_{n,d-1} ~|~\A(\gamma)=1\}$ and pairs containing a standard Young tableaux of shape $(d,1^{n-d})$ and a number $i$, $0\leq i\leq n-d$. This last bijection will allow us to write the sum over these paths with the $\A$ and $\B$ statistics in terms of hook-shaped Schur functions in the variables $q$ and $t$. In other words, we will obtain an explicit combinatorial formula for the expansion in Schur functions of $ \langle \nabla(e_n), s_{\mu}\rangle|_{\hooks}$
Before we start, we shall also notice that these bijections could easily be extended to paths ending with a diagonal step, by using the bijection between paths with $d$ diagonal steps that end with the factor $NE$ and paths with $d+1$ diagonal steps that end with a $D$ step.
\\
Recall in Section~\ref{Sec : combi chemins} we defined the touch points of a path. Notice that for a path $\gamma$ if $\A(\gamma)=0$ and $\TT(\gamma)=(\gamma_1,\gamma_2,\ldots,\gamma_k)$, then for all $i$, $\gamma_i$ is in $\{NE,D\}$. Let’s define the sets $\schP_{n,d, (i)}$ by:
\begin{equation*}
\schP{}_{n,d, (i)}=\{\gamma \in \schP{}_{n,d}~|~ \A(\gamma)=i\}. 
\end{equation*} 
Let $\{\mathcal{M}_{n,d}\}$ be a family of maps:
\begin{align*}\mathcal{M}_{n,d}:\SYT(d,1^{n-d}) &\rightarrow \schP{}_{n,d-1, (0)}
\\       \tau &\mapsto \gamma_1\gamma_2\cdots\gamma_n, 
\end{align*}
with $\gamma_n=NE$, $\gamma_{n-i}=NE$ if $i\in\DD(\tau)$ and $\gamma_{n-i}=D$ otherwise (see Figure~\ref{Fig : exemple de P} for an example).
Let $\{\mathcal{R}_{n,d}\}$ be a family of maps:
\begin{align*}\mathcal{R}_{n,d}: \schP{}_{n,d-1, (0)}&\rightarrow\SYT(d,1^{n-d}) 
\\       \gamma &\mapsto \tau, 
\end{align*}
with $\DD(\mathcal{R}_{n,d}(\gamma))=\{n-i ~|~1\leq i \leq n-1, \gamma_i=NE \in \TT(\gamma)\}$ (see Figure~\ref{Fig : exemple de R} for an example).
\begin{figure}[!htb]
\centering
\begin{minipage}{8.5cm}
\begin{tikzpicture}[scale=.5]
\draw (0,0)--(4,0)--(4,1)--(1,1)--(1,3)--(0,3)--(0,0);
\draw (1,0)--(1,1);
\draw (2,0)--(2,1);
\draw (3,0)--(3,1);
\draw (0,1)--(1,1);
\draw (0,2)--(1,2);
\node(1) at (.5,.5){$1$};
\node(2) at (1.5,.5){$2$};
\node(3) at (.5,1.5){$3$};
\node(4) at (.5,2.5){$4$};
\node(5) at (2.5,.5){$5$};
\node(6) at (3.5,.5){$6$};
\node(t) at (-1,1.5){$\tau=$};
\node(ds) at (2,-1){$\DD(\tau)=\{2,3\}$};
\node(fl) at (6,2){$\mapsto$};
\draw (7,-1)--(8,0)--(9,1)--(9,2)--(10,2)--(10,3)--(11,3)--(12,4)--(12,5)--(13,5);
\node(g1) at (8.5,-.5){$\gamma_1$};
\node(g2) at (9.5,.5){$\gamma_2$};
\node(g3) at (10.5,1.5){$\gamma_3$};
\node(g4) at (11.5,2.5){$\gamma_4$};
\node(g5) at (12.5,3.5){$\gamma_5$};
\node(g6) at (13.5,4.5){$\gamma_6$};
\end{tikzpicture}
\caption{An example of the application of map $\mathcal{M}_{6,4}$.}\label{Fig : exemple de P}
\end{minipage}
\begin{minipage}{8.5cm}
\begin{tikzpicture}[scale=.5]
\draw (8,0)--(12,0)--(12,1)--(9,1)--(9,3)--(8,3)--(8,0);
\draw (9,0)--(9,1);
\draw (10,0)--(10,1);
\draw (11,0)--(11,1);
\draw (8,1)--(9,1);
\draw (8,2)--(9,2);
\node(1) at (8.5,.5){$1$};
\node(3) at (9.5,.5){$3$};
\node(2) at (8.5,1.5){$2$};
\node(4) at (8.5,2.5){$4$};
\node(5) at (10.5,.5){$5$};
\node(6) at (11.5,.5){$6$};
\node(g) at (0,2.5){$\gamma=$};
\node(touch) at (4,-2){$\TT(\gamma)=\{D,D,NE,D,NE,NE\}$};
\node(fl) at (7,1.5){$\mapsto$};
\draw (0,-1)--(1,0)--(2,1)--(2,2)--(3,2)--(4,3)--(4,4)--(5,4)--(5,5)--(6,5);
\node(g1) at (1.5,-.5){$\gamma_1$};
\node(g2) at (2.5,.5){$\gamma_2$};
\node(g3) at (3.5,1.5){$\gamma_3$};
\node(g4) at (4.5,2.5){$\gamma_4$};
\node(g5) at (5.5,3.5){$\gamma_5$};
\node(g6) at (6.5,4.5){$\gamma_6$};
\end{tikzpicture}
\caption{An example of the application of map $\mathcal{R}_{6,4}$.}\label{Fig : exemple de R}
\end{minipage}
\end{figure}
\begin{lem}The families of maps $\{\mathcal{M}_{n,d}\}$ and $\{\mathcal{R}_{n,d}\}$ are well defined. 
\end{lem}
\begin{proof}
We have already seen that $\gamma$ in $\schP_{n,d-1,(0)}$ is represented by a word in $\{NE,D\}^{n-1}NE$ such that $|\gamma|_D=d-1$. Moreover, $\mathcal{M}_{n,d}(\tau)$ is in $\{NE,D\}^{n-1}NE$ by construction. For $\tau\in \SYT(d,1^{n-d}) $ the descent set of the tableau $\tau$ is a subset of $n-d$ elements in $\{1,\ldots,n-1\}$, since for $j\not=1$ in the first column $j-1$ is lower or to the right of $j$ by definition of hook-shaped standard tableaux. Hence, there are $n-d+1$, $i$'s such that $\gamma_{n-i}=NE$ and $|\mathcal{M}_{n,d}(\tau)|_D=d-1$. 

For the map $\mathcal{R}_{n,d}$, notice that hooked-shaped tableaux are uniquely defined by their descent set. Furthermore, there are $n-d+1$ north steps in  $\gamma$   and $\gamma_n=NE\in \TT(\gamma)$. Thus, for $1\leq i\leq n-1$ there are $n-d$ factors $\gamma_i$ such that $\gamma_i=NE$ and $\dd(\mathcal{R}_{n,d}(\tau))=n-d$. Ergo, $\mathcal{R}_{n,d}(\gamma)\in\SYT(d,1^{n-d})$.
\end{proof}
\begin{prop}\label{Prop : bij}Let $n$, $d$ be integers such that $0\leq d\leq n$. The map $\mathcal{M}_{n,d}$ is a bijection from the set of standard tableaux of shape $(d,1^{n-d})$ to the subset of Schr\"oder paths $\schP_{n,d-1,(0)} $,  of inverse $\mathcal{R}_{n,d}$. Moreover,the map $\mathcal{M}_{n,d}$ is such that $\m(\tau)=\B(\mathcal{M}_{n,d}(\tau))$ and the map $\mathcal{R}_{n,d}$ is such that $\m(\mathcal{R}_{n,d}(\gamma))=\B(\gamma)$.
\end{prop}
\begin{proof}For the first statement, we only need to show that $\mathcal{M}_{n,d}$ and $\mathcal{R}_{n,d}$ are inverse maps. Let $\tau$ be in $\SYT(d,1^{n-d}) $ if $i$ is in $\DD(\tau)$ (respectively, $i$ is not in $\DD(\tau)$), then the map $\mathcal{M}_{n,d}$ sends $i$ to $\gamma_{n-i}=NE$ in $\TT(\mathcal{M}_{n,d}(\tau))$ (respectively, to $\gamma_{n-i}=D$)  and $\mathcal{R}_{n,d}$ sends $\gamma_{n-i}=NE$ in $\TT(\mathcal{M}_{n,d}(\tau))$ to $i$ in $\DD(\mathcal{R}_{n,d}(\mathcal{M}_{n,d}(\tau)))$ (respectively,to $i$ not in $\DD(\mathcal{R}_{n,d}(\mathcal{M}_{n,d}(\tau)))$). Hence, $\mathcal{R}_{n,d}(\mathcal{M}_{n,d}(\tau))=\tau$, since hooked-shaped tableaux are uniquely determined by their descent set. For $\gamma$ in  $ \schP_{n,d-1,(0)}$, the proof of $\mathcal{M}_{n,d}(\mathcal{R}_{n,d}(\gamma))=\gamma$ is similar.
Additionally, in $ \schP_{n,d-1,(0)}$ all east steps are associated to a peak. For $\gamma=\gamma_1\gamma_2\cdots\gamma_n $, if $\gamma_{n-i}=NE$ and there are $k$ diagonal steps after  the factor $\gamma_{n-i}$, then the peak associated to that factor contributes $k$ to $\N$ and there is a return to the main diagonal after the factor $\gamma_{n-i}$ that contributes $i-k$ to $\B(\Gamma(\gamma))$. Moreover, the peak $\gamma_n=NE$ contributes nothing to $\B(\gamma)$ because it is the end of the path. Hence, $\m(\tau)=\B(\mathcal{M}_{n,d}(\tau))$ and $\m(\mathcal{R}_{n,d}(\gamma))=\B(\gamma)$, since both maps associate the factor $\gamma_{n-i}=NE$ in the touch sequence to $i$ in the descent set.
\end{proof}
This last proposition gives a combinatorial formula for $\langle\nabla e_n,s_{d,1^{n-d}} \rangle$ in terms of the major index as in Proposition~\ref{Prop : 1 part parking}. Although we now know to which paths the top weight are associated to (see Section~\ref{Sec : cristaux} for more on this).
\begin{cor} Let $n$, $d$ be positive integer such that $n\geq d$. Then:
\begin{equation*}\label{Eq : 1 part maj}\langle\nabla e_n,s_{d,1^{n-d}} \rangle|_{1 \pp}=\sum_{\tau\in \SYT(d,1^{n-d})} s_{\m(\tau)}(q,t)
\end{equation*}
\end{cor}
\begin{proof}By Proposition~\ref{Prop : sum algo 1 part} and Proposition~\ref{Prop : bij}.
\end{proof}
In order to get the same type of formula for the restriction on hook-shaped partitions, we will partition the set  $\schP_{n,d-1,(1)}$. We need maps to do so. Let $\tau$ be in $\SYT(d,1^{n-d})$, then we define the path $\gamma_\tau=\gamma_1\gamma_2\cdots\gamma_{n-1}$ where:
\begin{equation*}
\gamma_{n-i}= \begin{cases}NE & \text{ if } i=1 \text{ or } i\in\DD(\tau)\backslash\{\max(\DD(\tau))\}
\\      D & \text{ if } i\not\in \DD(\tau)\cup\{1\}
\\      NNEE & \text{ if } i=\max(\DD(\tau)) \text{ and } 1\in \DD(\tau)
\\      NDE & \text{ if }  i=\max(\DD(\tau)) \text{ and } 1\not\in \DD(\tau)
\end{cases}
\end{equation*}
Let $V_{n,d}$ be a collection of sets defined by:
\begin{equation*}
V_{n,d}=\{\gamma \in \schP{}_{n,d-1} ~|~\gamma=D^jNNEEu \text{ or } \gamma=D^jNDEu, j\geq0, u\in \{NE,D\}^*NE \}.
\end{equation*}
Our first family of maps $\mathcal{S}_{n,d}$ is defined as follows,  for $n-d\geq 1$:
\begin{align*} \mathcal{S}_{n,d} : \SYT(d,1^{n-d}) &\rightarrow V_{n,d}
\\         \tau & \mapsto  \gamma_\tau
\end{align*}
The second family of maps is defined as follows,  for $n-d\geq 1$:
\begin{align*} \mathcal{T}_{n,d} :V_{n,d}  &\rightarrow  \SYT(d,1^{n-d}) 
\\    \gamma & \mapsto  \tau_\gamma
\end{align*}
Where for $U=\{1\}$ if $\gamma=D^jNDEv$, $U=\emptyset$ if $\gamma=D^jNNEEu$, $u\in \{NE,D\}^*$, $v\in \{NE,D\}^*NE$ we have $\DD(\tau_\gamma)=\{n-i : \gamma_i=Nw\in \TT(\gamma), w \in \{NEE,E,DE\}^*\}\backslash\{U\}$.
\begin{figure}[!htb]
\centering
\begin{minipage}{8.5cm}
\begin{tikzpicture}[scale=.5]
\draw (0,0)--(4,0)--(4,1)--(1,1)--(1,3)--(0,3)--(0,0);
\draw (1,0)--(1,1);
\draw (2,0)--(2,1);
\draw (3,0)--(3,1);
\draw (0,1)--(1,1);
\draw (0,2)--(1,2);
\node(1) at (.5,.5){$1$};
\node(2) at (1.5,.5){$2$};
\node(3) at (.5,1.5){$3$};
\node(4) at (.5,2.5){$4$};
\node(5) at (2.5,.5){$5$};
\node(6) at (3.5,.5){$6$};
\node(t) at (-1,1.5){$\tau=$};
\node(ds) at (2,-1){$\DD(\tau)=\{2,3\}$};
\node(fl) at (6,2){$\mapsto$};
\draw (7,-1)--(8,0)--(9,1)--(9,2)--(10,3)--(11,3)--(11,4)--(12,4)--(12,5)--(13,5);
\node(g1) at (8.5,-.5){$\gamma_1$};
\node(g2) at (9.5,.5){$\gamma_2$};
\node(g3) at (10.75,1.75){$\gamma_3$};
\node(g4) at (12.5,3.5){$\gamma_4$};
\node(g5) at (13.5,4.5){$\gamma_5$};
\end{tikzpicture}
\caption{Example of the map $\mathcal{S}_{6,4}$}\label{Fig : exemple de S}
\end{minipage}
\begin{minipage}{8.5cm}
\begin{tikzpicture}[scale=.5]
\draw (8,0)--(12,0)--(12,1)--(9,1)--(9,3)--(8,3)--(8,0);
\draw (9,0)--(9,1);
\draw (10,0)--(10,1);
\draw (11,0)--(11,1);
\draw (8,1)--(9,1);
\draw (8,2)--(9,2);
\node(1) at (8.5,.5){$1$};
\node(3) at (9.5,.5){$3$};
\node(2) at (8.5,1.5){$2$};
\node(4) at (8.5,2.5){$4$};
\node(5) at (10.5,.5){$5$};
\node(6) at (11.5,.5){$6$};
\node(g) at (0,2.5){$\gamma=$};
\node(touch) at (4,-2){$\TT(\gamma)=\{D,D,NNEE,D,NE\}$};
\node(fl) at (7,1.5){$\mapsto$};
\draw (0,-1)--(1,0)--(2,1)--(2,3)--(4,3)--(5,4)--(5,5)--(6,5);
\node(g1) at (1.5,-.5){$\gamma_1$};
\node(g2) at (2.5,.5){$\gamma_2$};
\node(g3) at (3.75,1.75){$\gamma_3$};
\node(g4) at (5.5,3.5){$\gamma_4$};
\node(g5) at (6.5,4.5){$\gamma_5$};
\end{tikzpicture}
\caption{Example of the map $\mathcal{T}_{6,4}$}\label{Fig : exemple de T}
\end{minipage}
\end{figure}
\begin{lem}Let $n$, $d$ be positive integers such that $n-d\geq1$, the maps $\mathcal{S}_{n,d}$ and $\mathcal{T}_{n,d}$  are well defined. 
\end{lem}
\begin{proof}The factors $\gamma_i$ are all Schr\"oder paths ans the concatenation of Schr\"oder paths is a Schr\"oder path. Moreover, the factor $\gamma_{n-1}$ of $\gamma_\tau$ contains a north step and all the other north steps are related to an element in the descent set. Hence, we have $n-d+1$ north steps and $\gamma_\tau$ is an element of $\schP_{n,d-1}$. Additionally, the construction of $\gamma_\tau$ is based on four mutually exclusive conditions that define $\gamma_{n-i}$. If $i=\max(\DD(\tau))$,  then $n-i\leq n-k$ for all $k\in \DD(\tau)$.  Thus, by construction the path starts with $D^jNNEE$ or $D^jNDE$. Since there is only one maximum of the descent set the map $\mathcal{S}_{n,d}$ is well defined.
Notice that for all path in $V_{n,d}$ there is exactly one factor in $\{NNEE,NDE\}$ and all other are in $\{NE,D\}$. Hence, $\TT(\gamma)$ has $n-1$ factors because it is equivalent to counting the numbers of north and diagonal steps minus one. Therefore, the descent set is included in $\{1,\ldots,n-1\}$. Moreover, there are $n-d+1$ north step and one is in the factor $\gamma_{n-1}$. Ergo, $1\in\{n-i : \gamma_i=Nw\in \TT(\gamma), w \in \{NEE,E,DE\}^*\}$.
 
If $\gamma=D^jNNEEu$, $u\in \{NE,D\}^*$, then there are $n-d+1$ north steps for $n-d$ factors containing a north step. Consequently, $\dd(\tau_\gamma)=n-d$ and it uniquely determines a hooked-shaped tableau in $\SYT(d,1^{n-d})$. 

If $\gamma=D^jNDEv$, $v\in \{NE,D\}^*NE$, then there are $n-d+1$ north steps for $n-d+1$ factors containing a north step. But $U$ takes out one from the descent set. Hence, the descent set has $n-d$ elements and it uniquely determines a hooked-shaped tableau in $\SYT(d,1^{n-d})$. Consequently,$\mathcal{T}_{n,d}$  is well defined.
\end{proof}
The interesting thing about these maps is that they preserve statistics as shown in the next lemma.
\begin{lem}\label{Lem : map bounce=maj-des}Let $n$, $d$ be positive integers such that $n-d\geq1$. The maps $\mathcal{S}_{n,d}$ and $\mathcal{T}_{n,d}$ are bijective maps and $\mathcal{T}_{n,d}=\mathcal{S}_{n,d}^{-1}$. Moreover, the maps $\mathcal{T}_{n,d}$ and $\mathcal{S}_{n,d}$ preserve statistics in the following way: 
\begin{equation*}
 \B(\gamma)=\m(\mathcal{T}_{n,d}(\gamma))-\dd(\mathcal{T}_{n,d}(\gamma)),
\end{equation*}
\begin{equation*}
 \B(\mathcal{S}_{n,d}(\tau))=\m(\tau)-\dd(\tau).
\end{equation*}
\end{lem}
\begin{proof}Let $\tau$ be a Standard Young tableau in of shape $(d,1^{n-d})$, $\mathcal{S}_{n,d}(\gamma)=\gamma_{\tau}$, $\gamma_\tau=\gamma_1\gamma_2\cdots\gamma_{n-1}$. For $i\geq 2$, if the factor $\gamma_{n-i}=NE$, then the map $\mathcal{T}_{n,d}$ send that factor to $i\in\DD(\mathcal{T}_{n,d}(\gamma_\tau))$. But the map $\mathcal{S}_{n,d}$ gives us $\gamma_{n-i}=NE$ when $i\in\DD(\tau)$. If $\gamma_{n-i}=NNEE$, then  the map $\mathcal{T}_{n,d}$ sends that factor to $1, i\in\DD(\mathcal{T}_{n,d}(\gamma_\tau))$ and the map $\mathcal{S}_{n,d}$ gives $\gamma_{n-i}=NNEE$ when $1, i\in\DD(\mathcal{T}_{n,d}(\gamma_\tau))$. Finally, if $\gamma_{n-i}=NDE$, then  the map $\mathcal{T}_{n,d}$ sends that factor to $i\in\DD(\mathcal{T}_{n,d}(\gamma_\tau))$ and the image of map $\mathcal{S}_{n,d}$ is $\gamma_{n-i}=NDE$ when $i\in\DD(\mathcal{T}_{n,d}(\gamma_\tau))$. Thus,$\mathcal{T}_{n,d}(\mathcal{S}_{n,d}(\tau))=\tau$. The proof of $\mathcal{S}_{n,d}(\mathcal{T}_{n,d}(\gamma))=\gamma$ is similar.

The proof that $\B(\mathcal{S}_{n,d}(\tau))=\m(\tau)-\dd(\tau)$ is almost the same as the proof in Proposition~\ref{Prop : bij}. Since there are only $\gamma_{n-1}$ factors, we need to subtract one to each contribution to $\B(\Gamma(\gamma))$. Moreover, if $\gamma_{n-i}=NNEE$ and there are $k$ diagonal steps after  the factor $\gamma_{n-i}$, then the peak associated to that factor contributes $k$ to $\N$ and there is a return to the main diagonal after the factor $\gamma_{n-i}$ that contributes $i-k-1$ to $\B(\Gamma(\gamma))$. Finally, if $\gamma_{n-i}=NDE$ and there are $k$ diagonal steps after  the factor $\gamma_{n-i}$, then the peak associated to that factor contributes $k$ to $\N$ and there is a return to the main diagonal after the factor $\gamma_{n-i}$ that contributes $i-k-1$ to $\B(\Gamma(\gamma))$. But $1\not \in \DD(\mathcal{T}_{n,d}(\gamma)$ which mean we do not add $1-1=0$ to $\DD(\mathcal{T}_{n,d}(\gamma)$. Hence, $\m(\tau)-\dd(\tau)=\B(\mathcal{S}_{n,d}(\tau))$ and $\m(\mathcal{T}_{n,d}(\gamma))-\dd(\mathcal{T}_{n,d}(\gamma))=\B(\gamma)$, because both maps associate the factor $\gamma_{n-i}=NE$ to $i$ in the descent set. Consequently, $\B(\mathcal{S}_{n,d}(\tau))=\m(\tau)-\dd(\tau)$. The map $\mathcal{S}_{n,d}$ is a bijection of inverse $\mathcal{T}_{n,d}$, and, thus, we also have $\B(\gamma)=\m(\mathcal{T}_{n,d}(\gamma))-\dd(\mathcal{T}_{n,d}(\gamma))$.
\end{proof}
To extend the maps $\mathcal{S}_{n,d}$ and $\mathcal{T}_{n,d}$ we need to partition the paths of $\schP_{n,d,(1)}$. Notice that $\schP_{n,n, (1)}=\schP_{n,n-1, (1)}=\emptyset$ and $\schP_{n,n-2, (1)}=\{D^{n-2}NNEE, D^iNDED^jNE ~|~ i+j=n-3\}$.
\\
For $d=n-2$, we define $\Pi_{n,d-1}$ to be the identity map. For $n-d+1\geq 3$, let 
\begin{align*} \Pi{}_{n,d-1} : &\schP{}_{n,d-1, (1)} &&~\rightarrow~&& \schP{}_{n,d-1, (1)}
\\    & uNNEED^jNEv && \mapsto && uNED^jNNEEv
\\    &uNDED^jNEv && \mapsto & &uNED^jNDEv
\\    &uNNEED^jNE && \mapsto & &uNED^jNNEE
\\    &D^iNEuD^jNNEE && \mapsto & &D^iNNEEuD^jNE
\\    &D^iNEuNDED^jNE && \mapsto & &D^iNDEuNED^jNE
\end{align*}
For $u$ in $\{NE,D\}^*$ and $v$  in $\{NE,D\}^*NE$. 
\begin{figure}[!htb]
\begin{minipage}{8.5cm}
\centering
\begin{tikzpicture}[scale=.5]
\draw (0,-1)--(1,0)--(2,1)--(2,2)--(3,3)--(4,3)--(4,4)--(5,4)--(5,5)--(6,5);
\node(fl) at (6,2){$\mapsto$};
\draw (7,-1)--(8,0)--(9,1)--(9,2)--(10,2)--(10,3)--(11,4)--(12,4)--(12,5)--(13,5);
\end{tikzpicture}
\caption{Example of the map $\Pi_{6,3}$}\label{Fig : exemple de pi-1}
\end{minipage}
\begin{minipage}{8.5cm}
\centering
\begin{tikzpicture}[scale=.5]
\draw (0,-1)--(1,0)--(2,1)--(2,3)--(4,3)--(5,4)--(5,5)--(6,5);
\node(fl) at (7,1.5){$\mapsto$};
\draw (7,-1)--(8,0)--(9,1)--(9,2)--(10,2)--(11,3)--(11,5)--(13,5);
\end{tikzpicture}
\caption{Example of the map $\Pi_{6,3}$}\label{Fig : exemple de pi-2}
\end{minipage}
\end{figure}
\begin{prop}\label{Prop : partition}For all integers $n$ and $d$ such that $n-d\geq 1$. The sets $\{ \Pi_{n,d-1}^k(\gamma_\tau)\}$ over all $\tau\in\SYT(d,1^{n-d})$ are a partition of the set $\schP_{n,d-1, (1)}$. Additionally, $\Pi_{n,d-1}$ is cyclic of order $n-d$.
\end{prop}
\begin{proof}We shall first notice that the map is well  defined, since there is exactly one factor $NNEE$ or $NDE$ in path of $\A$ 1. Secondly, $\gamma_\tau$ is of shape $D^iNNEED^jNEuNE$ or $D^iNDED^jNEuNE$ and the action of $\Pi_{n,d-1}$ is to exchange the factor $NNEE$ (respectively, $NDE$) with the factor next $NE$ factor. Because there are $n-d+1$ north steps in $\gamma_\tau$, if $NNEE$ is a factor of $\gamma_\tau$  (respectively, $NDE$) there are $n-d-1$ (respectively, $n-d$)  factors $NE$ above the factor $NNEE$ (respectively, $NDE$)  and $\Pi^{n-d-1}_{n,d-1}(\gamma_\tau)=D^iNED^jNEuNNEE$ (respectively, $\Pi^{n-d-1}_{n,d-1}(\gamma_\tau)=D^iNED^jNEu'NDED^kNE$, where $u=u'NED^k$). Therefore, $\Pi^{n-d}_{n,d-1}(\gamma_\tau)=\gamma_\tau$.

Let $\gamma$ be a path in $\schP_{n,d-1,(1)}$. Then, there is a unique factor $NNEE$ or $NDE$ which corresponds to the line of $\A$ $1$. Hence, $\gamma=D^jNEu'NNEEu''NE$ (respectively, $\gamma=D^jNEu'NDEu''NE$, $\gamma=D^jNEu'NNEE$), with $u'$ and $u''$ in $\{NE,D\}$. Let $k$ be the number of north steps before the factor $NNEE$ (respectively,$NDE$, $NNEE$, with $k=n-d-1$). Consequently, $\Pi_{n,d-1}^k(\gamma_0)=\gamma$, for $\gamma_0=D^jNNEEu'NEu''NE$ (respectively,$\gamma_0=D^jNDEu'NEu''NE$, $\gamma_0=D^jNNEEu'NE$). Thus, $\mathcal{T}_{n,d}(\gamma_0)=\tau_{\gamma_0}$. So, $\gamma$ is in the set $\{ \Pi_{n,d-1}^k(\gamma_{\tau_{\gamma_0}})\}$.

Finally, let $\gamma$ be in $\{ \Pi_{n,d-1}^k(\gamma_\tau)\}\cap \{ \Pi_{n,d-1}^k(\gamma_\rho)\}$, then there is $k$ such that $\gamma=\Pi_{n,d-1}^k(\gamma_\tau)$ and $l$ such that $\gamma=\Pi_{n,d-1}^l(\gamma_\rho)$. Hence, $\Pi_{n,d-1}^{n-d-k+l}(\gamma_\rho)=\gamma_\tau$. By previous statements we know $k=l$ ; it is the number of north steps before the factor $NNEE$ of $NDE$ in $\gamma$. Ergo,$\gamma_\rho=\gamma_\tau$. Thus,by Lemma~\ref{Lem : map bounce=maj-des}, $\rho=\mathcal{T}_{n,d}(\gamma_\rho)=\mathcal{T}_{n,d}(\gamma_\tau)=\tau$. Therefore $\{ \Pi_{n,d-1}^k(\gamma_\tau)\}\cap \{ \Pi_{n,d-1}^k(\gamma_\rho)\}=\emptyset$ if $\rho\not=\tau$.
\end{proof}
One might see similarities between the next lemma and Lemma~\ref{Lem : +1-1}, since the bounce statistic increases by exactly one with each iteration of $\Pi_{n,d-1}$. But it is worth mentioning that in this case the area remains one. Thus, this is not an extension of the algorithm seen in Section~\ref{Sec : algo} for paths with a bounce statistic value of zero.
\begin{lem} Let $U=\{uNED^jNNEE, uNDED^jNE ~|~ u\in \{NE,D\}^*\}$. The map $\Pi_{n,d-1}$ is such that $\B(\Pi_{n,d-1}(\gamma))=\B(\gamma)+1$ for all $\gamma \in \schP_{n,d, (1)}\backslash U$.
\end{lem}
\begin{proof}If $\gamma$ as a factor $NNEE$, then the map $\Pi_{n,d-1}$ swaps the factor $NNEE$ with the next factor $NE$ and all factors return to the main diagonal in $\Pi_{n,d-1}(\gamma)$. Hence, the number of peak under each diagonal step is left unchanged and $\N(\gamma)=\N(\Pi_{n,d-1}(\gamma))$.  Moreover,  $\A(\gamma)=1$, and, therefore, there is some $k$ such that $\Gamma(\gamma)=(NE)^{n-d-k}NNEE(NE)^{k-1}$. Thus, we obtain $\B(\Gamma(\gamma))= \binom{n-d+2}{2}-k$ and $\B(\Gamma(\Pi_{n,d-1}(\gamma)))= \binom{n-d+2}{2}-k+1$.

If $\gamma$ as a factor $NDE$, then $\B(\Pi_{n,d-1}(\Gamma(\gamma)))=\B(\Gamma(\gamma))$. Since the map $\Pi_{n,d-1}$ swaps the factor $NDE$ with the next factor $NE$ and all factors return to the main diagonal, in $\Pi_{n,d-1}(\gamma)$ there is one more peak below the diagonal step coming from the factor $NDE$ than in $\gamma$. Moreover, the number of peaks below all the other diagonal steps remains unchanged. Ergo, $\N(\gamma)+1=\N(\Pi_{n,d-1}(\gamma))$.
\end{proof}
With this partition we can put forward a bijection between tableaux and Schr\"oder paths.
\begin{prop}\label{Prop : bij chemins tableaux}Let $\mathcal{Q}_{n,d}$ be a map between the product set $\SYT(d,1^{n-d})\times \{0,1,\cdots,n-d-1\}$ and the set $\schP_{n,d-1, (1)}$, defined by $\mathcal{Q}_{n,d}(\tau,i)=\Pi_{n,d-1}^i(\mathcal{S}_{n,d}(\tau))$. Then, the map $\mathcal{Q}_{n,d}$ is a bijection. Moreover, $\B(\mathcal{Q}_{n,d}(\tau,i))=\m(\tau)-\dd(\tau)+i$ for all $i$ in $\{0,1,\cdots,n-d-2\}$.
\end{prop}
\begin{proof} By Lemma~\ref{Lem : map bounce=maj-des} and Proposition~\ref{Prop : partition}, this map is well defined. Furthermore, for $\gamma$ in $\schP_{n,d-1, (1)}$ there is a unique $\tau$ in $\SYT(d,1^{n-d})$ an a unique integer $i$ such that $0\leq i\leq n-d-1$ and $\Pi_{n,d-1}^i(\gamma_\tau)=\gamma$. Moreover, $\mathcal{S}_{n,d}$ is a bijection of inverse $\mathcal{T}_{n,d}$. Hence, $\mathcal{T}_{n,d}(\gamma_\tau)=\tau$ and the pre-image of $\gamma$ is unique. Ergo $\mathcal{Q}_{n,d}$ is a bijection. 

By the previous lemma $\B(\Pi_{n,d-1}^i(\gamma_\tau))-i=\B(\gamma_\tau)$ if $\Pi_{n,d-1}^i(\gamma_\tau)$ is not an element of $U=\{uNED^jNNEE, uNDED^jNE ~|~ u\in \{NE,D\}^*\}$. The proof of Proposition~\ref{Prop : partition} shows that   $\Pi_{n,d-1}^i(\gamma_\tau)$ is an element of $U$ if and only if $i=n-d-1$. Hence, we only need to show $\B(\gamma_\tau)=\m(\tau)-\dd(\tau)$. But this is true by Lemma~\ref{Lem : map bounce=maj-des}.
\end{proof}
We now restate and prove our main theorem.
\begin{theo*}[1] If  $\mu\in\{(d,1^{n-d})~|~1\leq d\leq n\}$ and $\nu\vdash n$, then:
  \begin{equation} \langle \nabla(e_n), s_{\mu}\rangle|_{\hooks}=\sum_{\tau\in \SYT(\mu)}  s_{\m(\tau)}(q,t)+\sum_{i=2}^{\dd(\tau)} s_{\m(\tau)-i,1}(q,t),
  \end{equation}
    \begin{equation}\langle \nabla^m(e_n), s_\nu \rangle|_{1 \pp}=\sum_{\tau\in \SYT(\nu)}  s_{m\binom{n}{2}-\m(\tau')}(q,0)=\sum_{\tau\in \SYT(\nu)}  s_{m\binom{n}{2}-\m(\tau')}(0,t),
  \end{equation}
  and:
  \begin{equation}\langle \nabla^m(e_n), e_n\rangle|_{\hooks}=  s_{m\binom{n}{2}}(q,t)+\sum_{i=2}^{n-1} s_{m\binom{n}{2}-i,1}(q,t).
  \end{equation}
\end{theo*}
Note that $\m(1^n)=\binom{n}{2}$, $\binom{n}{2}-\m(\tau')=\m(\tau)$ and $\dd(\tau)=n-1-\dd(\tau')$.
\begin{proof}Equation~\eqref{Eq : 2 de the principale} is true by Proposition~\ref{Prop : 1 part parking}. For $m=1$ Equation~\eqref{Eq : 3 de the principale} is a direct consequence of Equation~\eqref{Eq : 1 de the principale}.  Notice that if $\gamma$ is in  $\sch_{n,d}^{(1)}$, then $\gamma E^{(m-1)n}$ is in $\sch_{n,d}^{(m)}$. Furthermore, the difference between the area of $\gamma$ and the area of $\gamma E^{(m-1)n}$ is exactly $(m-1)\binom{n}{2}$. Ergo, for $m>1$ Equation~\eqref{Eq : 3 de the principale} follows from Corollary~\ref{Cor : m>=1 dinv=1}.
\\
The first sum on the right side of Equation~\eqref{Eq : 1 de the principale} follows from Proposition~\ref{Prop : 1 part parking}. For the second sum on the right side of Equation~\eqref{Eq : 1 de the principale}, notice that,  by Proposition~\ref{Prop : bij chemins tableaux}, there is a bijection between paths of area value equal to one with $d$ diagonal steps ending with $NE$ or $NNEE$ and the product set of standard Young tableaux of shape $(d,1^{n-d})$ and the set $\{0,\ldots,n-d-1\}$.
 Moreover,one can see that the cyclic action of the map $\Pi_{n,d-1}$, proven in Proposition~\ref{Prop : partition}, puts $\Pi_{n,d-1}^{n-d-1}(\mathcal{S}_{n,d}(\tau))$  in $\{uNED^jNNEE, vNDED^jNE ~|~ u\in \{NE,D\}^{n-d-2},  v\in \{NE,D\}^{n-d-1}\}$, since it is of order $n-d$. So, by Corollary ~\ref{Cor : aire 1 vs 1 part} and Proposition~\ref{Prop : 1 part parking}, the set $\{\Pi_{n,d-1}^{n-d-1}(\mathcal{S}_{n,d}(\tau))\}$ contains the only paths that contribute to Schur functions indexed by partition of length $1$.  Consequently, we need only to consider the set $\{0,\ldots,n-d-2\}$ because the second sum relates to partition of length $2$. Additionally, $\dd(\tau)=n-d$, for $\tau\in \SYT(d,1^{n-d})$, hence $2\leq i\leq \dd(\tau)$ if and only if $n-d-2 \geq \dd(\tau)-i\geq 0$. Therefore, we sum from $2$ to $\dd(\tau)$. Except for the paths already contributing to Schur function having only one part, the restriction to a value of one for the area corresponds to hook-shaped Schur functions. Indeed, $s_{a,b}(q,t)=(qt)^b(q^{a-b}+q^{a-b-1}t+\cdots+qt^{a-b-1}+t^{a-b})$, and, thus, the monomial $q^ct$ can only be found when $b\in\{0,1\}$. Ergo, by Proposition~\ref{Prop : bij chemins tableaux}, we have the stated result.
\end{proof}
We conjecture this is true for all $\mu$ when $m=1$. We also know, by Lemma~\ref{Lem : m>2 dinv=1}, that this is not true for $m>1$.
\begin{conj}\label{Conj : nabla(e_n)} For all  $\mu\vdash  n$:
  \begin{equation*} \langle \nabla(e_n), s_{\mu}\rangle|_{\hooks}=\sum_{\tau\in \SYT(\mu)}  s_{\m(\tau)}(q,t)+\sum_{i=2}^{\dd(\tau)} s_{\m(\tau)-i,1}(q,t),
  \end{equation*}
\end{conj}
%%%%%%%%%%%%%%%%%%%%%%%%%%%%%%%%%%%%%%%%%%%%%%%
%%%%%%%%%%%%%                        inclusion exclusion sur chemins
%%%%%%%%%%%%%%%%%%%%%%%%%%%%%%%%%%%%%%%%%%%%%%%
\section{Inclusion Exclusion}\label{Sec : inclu-exclu}
In this section we will see that half the paths in $\schP_{n,d,(1)}$ are related to $\schP_{n,d-1,(1)}$ and the other half is related to $\schP_{n,d+1,(1)}$. Thereafter, this will be used to find a positive formula for an alternating sum. The outcome is needed to prove results on multivariate diagonal harmonics in \cite{[Wal2019a]}. 
Recall that $\schP_{n,d, (1)}=\{\gamma \in \schP{}_{n,d}~|~ \A(\gamma)=1\}$ and  $\schP_{n,n, (1)}=\schP_{n,n-1, (1)}=\emptyset$. Now, let:
\begin{equation*}
\schT{}_{n,d}=\{\gamma \in \schP{}_{n,d,(1)}~|~ \gamma=D^jNNEE\gamma'NENE \text{ or } \gamma=\gamma'NED^jNNEE\gamma'', j\geq 0 \},
\end{equation*}
and:
\begin{equation*}
\schB{}_{n,d}=\{\gamma \in \schP{}_{n,d,(1)}~|~ \gamma=D^jNNEE\gamma'DNE \text{ or } \gamma=\gamma'NDED^jNE\gamma'', j\geq 0\}.
\end{equation*}
\begin{lem}\label{Lem : top bottom}
For $1\leq d\leq n-4$ we have the following equality  $\schP_{n,d, (1)}=\schT{}_{n,d}\cup \schB{}_{n,d}$. Additionally, $\schP_{n,0, (1)}=\schT{}_{n,0}$, $\schP_{n,n-3, (1)}=\schT{}_{n,n-3}\cup \schB{}_{n,n-3}\cup \{D^{n-3}NNEENE\}$ and $\schP_{n,n-2, (1)}= \schB{}_{n,n-2}\cup\{D^{n-2}NNEE\}$. Furthermore, for all $d$, $\schT{}_{n,d}\cap \schB{}_{n,d}=\emptyset$.
 \end{lem}
 \begin{proof}A simple check shows that the four cases of $\schT_{n,d}$ and  $\schB_{n,d}$ are mutually exclusive. Hence, $\schT{}_{n,d}\cap \schB{}_{n,d}=\emptyset$.
 The cases $d=0$, $d=n-2$ and $d=n-3$ are related to the maximal number of north steps and diagonal steps. For $d$ general, $1\leq d\leq n-4$, let $\gamma$ be in $\schP_{n,d,(1)}$. When a Schr\"oder path has an area value of  $1$, there is a factor $\pi=NNEE$ or $\pi=NDE$ such that $\gamma=u\pi v$, $u\in \{NE,D\}^*$ and $v\in \{NE,D\}^*NE\cup \{\varepsilon\}$. By definition of $\schP_{n,d,(1)}$, if $v=\varepsilon$, then $\pi=NNEE$ and $\gamma$ is in $\schT{}_{n,d}$. Moreover, if $\pi=NDE$, then there is a factor $NE$ in $v$ and $\gamma$ is in $\schB{}_{n,d}$.  If $\pi=NNEE$ and $u=D^j$, then $v$ has at least two factors $NE$, since $d\leq n-4$,  and $v$ can end with $NENE$, in which case $\gamma$ is in $\schT{}_{n,d}$ or $v$ end with $DNE$ and $\gamma$ is in $\schB{}_{n,d}$. If $\pi=NNEE$ and $u\not=D^j$, then there is a factor $NE$ in $u$ and $\gamma$ is in $\schT{}_{n,d}$. \end{proof}
 
 Let $d$ be an integer such that $0\leq d \leq n-1$. For each $d$ let:
 \begin{align*} \rho_{n,d} :&\schT{}_{n,d}&&\rightarrow && \schB{}_{n,d+1}
 \\    & \gamma'NED^jNNEE\gamma''&&\mapsto &&\gamma'NDED^jNE\gamma''
 \\    &D^jNNEE\gamma'NENE&&\mapsto&&D^jNNEE\gamma'DNE
 \end{align*}

\begin{lem}\label{Lem : inclu exclu chemins}Let $d$ be an integer such that $0\leq d \leq n-1$. Then, for all $d$, $\rho_{n,d}$ is a bijection such that $\B(\rho(\gamma))=\B(\gamma)-1$.
\end{lem}
\begin{proof}Notice that $\gamma'NED^jNNEE\gamma''\not=D^kNNEE\gamma'NENE$ for all $k$ and $j$, since  one has a $NE$ factor before its $NNEE$ factor and  not the other. The path as an area value of one, ergo there is only one factor $NNEE$ in $\gamma$. Moreover, the map $\rho_{n,d}$ increases the number of diagonal steps by one. Therefore, for all $d$ the map $\rho_{n,d}$ is well defined. For the same reasons the map $\rho_{n,d}^{-1}$ defined by $\rho_{n,d}^{-1}(\gamma'NDED^jNE\gamma'')=\gamma'NED^jNNEE\gamma''$ and $\rho_{n,d}^{-1}(D^jNNEE\gamma'DNE)=D^jNNEE\gamma'NENE$ is well defined. Thus, it is the inverse of $\rho_{n,d}$ and we have a bijection.

Recall that the numph statistic is related to the number of peaks that are in a lower row than the diagonals. When the area value is one, all factors $NE$ and $NNEE$ contain a peak; they always return to the diagonal. If we compare the path $\gamma=\gamma'NED^jNNEE\gamma''$ and the path $\rho_{n,d}(\gamma)$, we see that only one diagonal step was added in $\rho_{n,d}(\gamma)$ and the factor $NDE$ in which it was added has its peak above the diagonal step. Hence, $\N(\rho_{n,d}(\gamma))$ is equal to $\N(\gamma)$ plus the number of peaks in $\gamma'$. Considering that  $\Gamma(\gamma'NED^jNNEE\gamma'')=(NE)^{|\gamma'|+1}NNEE(NE)^{|\gamma''|}$, the value of $\B(\Gamma(\gamma))$ is $\binom{n-d}{2}-|\gamma''|_N-1$. Moreover, $\Gamma(\gamma'NDED^jNE\gamma'')=(NE)^{n-d-1}$; therefore, $\B(\Gamma(\rho_{n,d}(\gamma)))=\binom{n-d-1}{2}$. Because the number of peaks in $\gamma'$ is equal to $|\gamma'|_E$ which is equal to $|\gamma'|_N$ and $|\gamma'|_N+|\gamma''|_N=n-d-3$, we obtain $\B(\gamma)-\B(\rho_{n,d}(\gamma))=1$.

Now comparing the paths $\gamma=D^jNNEE\gamma'NENE$ and $\rho_{n,d}(\gamma)$ we see that $\N(\rho_{n,d}(\gamma))$ is equal to $\N(\gamma)$ plus the number of peaks in $D^jNNEE\gamma'$. This is equivalent to counting the number of non-consecutive east steps in  $D^jNNEE\gamma'$, that is $n-d-3$, since there are $n-d$ east steps in $\gamma$. We know $\Gamma(D^jNNEE\gamma'NENE)=NNEE(NE)^{n-d-2}$. The value of $\B(\Gamma(\gamma))$ is $\binom{n-d}{2}-(n-d-1)$. Moreover, in we consider the restriction to north steps and east steps is $\Gamma(D^jNNEE\gamma'DNE)=NNEE(NE)^{n-d-3}$, then $\B(\Gamma(\rho_{n,d}(\gamma)))=\binom{n-d-1}{2}-(n-d-2)$. Hence, $\B(\gamma)-\B(\rho_{n,d}(\gamma))=1$.
\end{proof}
We now know some Schr\"oder paths of area 1 are associated to Schur functions in the variables $q$ and $t$ indexed by partitions of length one. This corollary is merely to show that the bijection in the previous lemma sends paths associated to Schur functions indexed by a length one partition to another path associated to Schur functions indexed by a length one partition.
\begin{cor}\label{Cor : enlever chemin 1 part} Let $\gamma$ be a path of $\schT_{n,d}$. Then, $\gamma$ contributes to a Schur function  indexed by a partition of length $1$ in $\langle \nabla(e_n), s_{d+1,1^{n-d-1}}\rangle$ if and only if $\rho_{n,d}(\gamma)$ contributes to a Schur function  indexed by a partition of length $1$ in $\langle \nabla(e_n), s_{d+2,1^{n-d-2}}\rangle$. Moreover, if $\gamma$ satisfies this property, then $\gamma=\gamma'NED^jNNEE$ and $\rho_{n,d}(\gamma)=\gamma'NDED^jNE$, with $\gamma'\in  \{NE,D\}^*$.
\end{cor}
\begin{proof}By Corollary~\ref{Cor : aire 1 vs 1 part}, the definition of $\rho_{n,d}$ and the definition of $\schT_{n,d}$.
\end{proof}
\begin{lem}\label{Lem : autre somme}
Let $d$ and $n$ be integers such that $0\leq d \leq n-2$. Then, $\mathcal{T}_{n,d+2}\circ \rho_{n,d}\circ \mathcal{S}_{n,d+1}$ is a bijection between the sets of tableaux: 
\begin{equation*}
\{\tau \in \SYT(d+1,1^{n-d-1})~|~ \{1,2\} \subseteq\DD(\tau)\}\simeq \{\tau \in \SYT(d+2,1^{n-d-2})~|~1\in\DD(\tau), 2\not\in\DD(\tau)\},
\end{equation*}
such that $\m(\tau)-\dd(\tau)=\m(\mathcal{T}_{n,d+2}\circ \rho_{n,d}\circ \mathcal{S}_{n,d+1}(\tau))-\dd(\mathcal{T}_{n,d+2}\circ \rho_{n,d}\circ \mathcal{S}_{n,d+1}(\tau))+1$.
Additionally, $\mathcal{Q}_{n,d+2}^{-1}\circ \rho_{n,d}\circ \mathcal{Q}_{n,d+1}$ is a bijection between the set :
\begin{equation*}
\{(\tau,i) \in \SYT(d+1,1^{n-d-1})\times\{1,\ldots,n-d-3\}~|~ 1\in \DD(\tau)\},
\end{equation*}
 and  the set :
 \begin{equation*}
 \{(\tau,i) \in \SYT(d+2,1^{n-d-2})\times\{0,\ldots,n-d-4\}~|~1\not\in\DD(\tau)\},
 \end{equation*}
 such that $\m(\tau)-i=\m(\mathcal{Q}_{n,d+2}^{-1}\circ \rho_{n,d}\circ \mathcal{Q}_{n,d+1}(\tau,i))-(i-1)$.
\end{lem}
\begin{proof}By Corollary~\ref{Cor : aire 1 vs 1 part} we know paths of shape $\gamma'NDED^jNE$  or $\gamma'NED^jNNEE$ are associated to a Schur function of length one. Notice that by definition $\rho_{n,d}$ sends paths of theses shapes to other paths of theses shapes. Hence, $\mathcal{Q}_{n,d}^{-1}\circ \rho_{n,d}\circ \mathcal{Q}_{n,d}(\tau, n-d-2)=(\tau',n-d-3)$, where $\DD(\tau')=\DD(\tau)\backslash\{1\}$. At this point one only needs to notice that $\mathcal{Q}_{n,d+2}^{-1}\circ \rho_{n,d}\circ \mathcal{Q}_{n,d+1}(\tau,i)=(\tau',i-1)$, $\DD(\tau')=\DD(\tau)\backslash\{1\}$ and $\DD(\mathcal{T}_{n,d+2}\circ \rho_{n,d}\circ \mathcal{S}_{n,d+1}(\tau))=\DD(\tau)\backslash\{2\}$ . The rest of the proof follows Proposition~\ref{Prop : bij chemins tableaux} from the definition of the maps $\rho_{n,d}$, $\mathcal{Q}_{n,d}$, $\mathcal{S}_{n,d}$ and $\mathcal{T}_{n,d}$ .
\end{proof}
\begin{prop}
Let $k$ be an integer such that $0\leq k\leq n-3$ and  $\psi : \Lambda_\mathbb{Q} \rightarrow \mathbb{Q}[q,t]$ be a linear map defined by $\psi(s_\lambda)=q^{|\lambda|}t^{\ell(\lambda)-1}$. If:
\begin{equation*}
h_k(q):=\psi\left(\sum_{d=0}^{k} (-1)^{k-d} \langle \nabla(e_n), s_{d+1,1^{n-d-1}}\rangle |_\textrm{pure hook}\right)q^{-k+d}t^{-1},
\end{equation*}
then:
\begin{align*}h_k(q)&=\underset{\{1,2\}\subseteq \DD(\tau)}{\sum_{\tau \in \SYT(k+1,1^{n-k-1})}} q^{\m(\tau)-\dd(\tau)}
+\underset{1\in \DD(\tau)}{\sum_{\tau \in \SYT(k+1,1^{n-k-1})}} \sum_{i=2}^{n-k-2} q^{\m(\tau)-i},
\\    &=\underset{1\in \DD(\tau), 2\not\in\DD(\tau)}{\sum_{\tau \in \SYT(k+2,1^{n-k-2})}} q^{\m(\tau)-\dd(\tau)}
+\underset{1\not\in \DD(\tau)}{\sum_{\tau \in \SYT(k+2,1^{n-k-2})}} \sum_{i=2}^{n-k-2} q^{\m(\tau)-i},
\end{align*}
where the restriction $|_\textrm{pure hook}$ is the restriction to Schur function indexed by partitions $(a,1)$, with $a\geq 1$.
\end{prop}
\begin{proof}Let $\tilde h_k(q)=\psi\left(\sum_{d=0}^{k} (-1)^{k-d} \langle \nabla(e_n), s_{d+1,1^{n-d-1}}\rangle |_\textrm{hook}\right)q^{-k+d}t^{-1}$. Due to Haglund's theorem, Lemma~\ref{Lem : inclu exclu chemins} and Lemma~\ref{Lem : top bottom}, for an integer $k$ we have:
\begin{align*}
\tilde h_k(q)&=\sum_{d=0}^{k} (-1)^{k-d}\sum_{\gamma \in \schT_{n,d}} q^{\B(\gamma)-k+d}+\sum_{\gamma \in \schB_{n,d}} q^{\B(\gamma)-k+d},
\\  &=\sum_{d=0}^{k} (-1)^{k-d}\sum_{\gamma \in \schT_{n,d}} q^{\B(\rho_{n,d}(\gamma))+1-k+d}+\sum_{d=1}^{k} (-1)^{k-d}\sum_{\gamma \in \schB_{n,d}} q^{\B(\gamma)-k+d}, %\text{ since } \schB{}_{n,0}=\emptyset,
\\  &=\sum_{d=0}^{k} (-1)^{k-d}\sum_{\gamma \in \schB_{n,d+1}} q^{\B(\gamma)+1-k+d}-\sum_{d=1}^{k} (-1)^{k-d+1}\sum_{\gamma \in \schB_{n,d}} q^{\B(\gamma)-k+d}, 
\\  &=\sum_{d=0}^{k} (-1)^{k-d}\sum_{\gamma \in \schB_{n,d+1}} q^{\B(\gamma)+1-k+d}-\sum_{d=0}^{k-1} (-1)^{k-d}\sum_{\gamma \in \schB_{n,d+1}} q^{\B(\gamma)-k+d+1}, 
\\  &=\sum_{\gamma \in \schB_{n,k+1}} q^{\B(\gamma)+1},
\\  &=\sum_{\gamma \in \schT_{n,k}} q^{\B(\gamma)}.
\end{align*}
The polynomial $h_k(q)$ only takes into account pure hooks, so we only need to consider the paths $\gamma'NED^jNNEE\gamma''$ and $D^jNNEE\gamma'NENE$, with $\gamma'$, $\gamma''\in \{NE,D\}$, as shown in Corollary~\ref{Cor : enlever chemin 1 part}. Thus, by Lemma~\ref{Lem : map bounce=maj-des}, the map $\mathcal{T}_{n,k+1}$, for $\gamma=D^jNNEE\gamma'NENE$, gives us  $\mathcal{T}_{n,k+1}(\gamma)$ is a tableau containing $\{1,2\}$ in its descent set and such that $\B(\gamma)=\m(\mathcal{T}_{n,k+1}(\gamma))-\dd(\mathcal{T}_{n,k+1}(\gamma))$. Additionally, for  $\gamma=\gamma'NED^jNNEE\gamma''$, $\mathcal{Q}_{n,k+1}^{-1}(\gamma)$ is a tableau containing $\{1\}$ in its descent set. By Proposition~\ref{Prop : bij chemins tableaux}, $\gamma$ is such that $\B(\gamma)=\m(\mathcal{Q}_{n,k+1}^{-1}(\gamma))-i$, $2\leq i \leq \dd(\mathcal{Q}_{n,k+1}^{-1}(\gamma))$. Moreover, if $\{1\}$ is in the descent set of $\tau$, then the map $\mathcal{Q}_{n,k+1}$ send $(\tau,0)$ to $D^jNNEE\gamma'NENE$ or $D^jNNEE\gamma'DNE$ the first one was already considered and the last one is in $\schB_{n,d}$. Finally, if $\{1\}$ is in the descent set of $\tau$, then the map $\mathcal{Q}_{n,k+1}$ send $(\tau,i)$ to $\schT_{n,d}$ for all $1\leq i\leq n-d-3$. Hence, we sum over $2\leq i \leq \dd(\mathcal{Q}_{n,k+1}^{-1}(\gamma))-1$. The second sum is a consequence of Lemma~\ref{Lem : autre somme}.
\end{proof}
Considering what is known about multivariate diagonal harmonics, it should be possible to extend the results of this section to the case for $\schP_{n,d,(i)}$. This generalization would lead to more results on multivariate diagonal harmonics.
%%%%%%%%%%%%%%%%%%%%%%%%%%%%%%%%%%%%%%%%%%%%%%%
%%%%%%%%%%%%%%%               Partial Crystal Decomposition
%%%%%%%%%%%%%%%%%%%%%%%%%%%%%%%%%%%%%%%%%%%%%%%
 \section{Partial Crystal Decomposition}\label{Sec : cristaux}
This section is mainly to explain the underlying idea throughout this paper. We can see this as finding the crystal decomposition of the Schr\"oder paths and the parking function (since their weighted sum relate to modules). We basically found some of the top weight and for some of them gave a map that gives the remainder of the crystal. In that setting we can say that for $m=1$, we can describe all the crystals in the case where the Schur functions are indexed by partitions of length one. When $m>1$, we can characterize only the top weights. For hooked-shaped Schur functions, we can only depict the top weight, when $m=1$. 

More precisely, the maps $\mathcal{R}_{n,d}$ and $\mathcal{T}_{n,d}$ determine in which crystal the paths lie. The maps $\tilde\varphi$, defined by the map $\varphi$ in Section~\ref{Sec : algo}, give the decomposition according to the top weight. This also is well defined, since for all $\gamma\in\{NE,D\}^{n-1}NE$  we have $\mathcal{T}_{n,d}\circ\Pi\circ\tilde\varphi(\gamma)=\mathcal{R}_{n,d}(\gamma)$ (see Figure~\ref{Fig : cristaux bien def} for an example). Notice that map $\mathcal{M}_{n,d}$ (respectively,$\mathcal{Q}_{n,d}$) tells us in which crystal component are the Schr\"oder paths of area value $0$ (respectively,area value $1$).
 
  \begin{figure}[!htb]
  \centering
  \begin{tikzpicture}[scale=.5]
\draw (0,0)--(4,0)--(4,1)--(1,1)--(1,3)--(0,3)--(0,0);
\draw (1,0)--(1,1);
\draw (2,0)--(2,1);
\draw (3,0)--(3,1);
\draw (0,1)--(1,1);
\draw (0,2)--(1,2);
\node(1) at (.5,.5){$1$};
\node(2) at (1.5,.5){$2$};
\node(3) at (.5,1.5){$3$};
\node(4) at (.5,2.5){$4$};
\node(5) at (2.5,.5){$5$};
\node(6) at (3.5,.5){$6$};
\node(fl) at (6,2){$\stackrel{\mathcal{R}_{n,d}}{\longleftarrow}$};
\draw (7,-1)--(8,0)--(9,1)--(9,2)--(10,2)--(10,3)--(11,3)--(12,4)--(12,5)--(13,5);
\draw (0,-9)--(1,-8)--(2,-7)--(2,-6)--(3,-5)--(4,-5)--(4,-4)--(5,-4)--(5,-3)--(6,-3);
\node(fl2) at (6,-6){$\stackrel{\Pi}{\longleftarrow}$};
\draw (7,-9)--(8,-8)--(9,-7)--(9,-6)--(10,-6)--(10,-5)--(11,-4)--(12,-4)--(12,-3)--(13,-3);
\node(flmonte) at (2,-2){$\stackrel{{\mathcal{T}_{n,d}}}{}\uparrow$};
\node(fldescend) at (10,-2){$\downarrow\tilde\varphi$};
\end{tikzpicture}
 \caption{}\label{Fig : cristaux bien def}
 \end{figure}
 
 Using the zeta map, so far we know the top weights for all crystals containing a parking function having a diagonal inversion statistic value of $0$, and for all hook-shaped partitions we know all top weights for all crystals containing a parking function having a diagonal inversion statistic value of $1$. For crystals containing a parking function having a diagonal inversion statistic value of $0$ that are not associated to a hook-shaped partition, we do not know the exact paths. Although, we do know in which subset of parking functions the lowest weight lie. Figure~\ref{Fig : cristaux} gives an overview of what is known so far.
 
 \begin{figure}[!htb]
 \begin{tikzpicture}[scale=.25]
 \node(3p1) at (-2,2){$\cdots$};
 \draw (1,0)--(6,0)--(6,1)--(4,1)--(4,3)--(2,3)--(2,4)--(1,4)--(1,0);
  \node(3p2) at (9,2){$\cdots$};
  \node(3p3) at (24,2){$\cdots$};
  \draw (26,0)--(26,5)--(27,5)--(27,1)--(34,1)--(34,0)--(26,0);
  \node(3p4) at (36,2){$\cdots$};
  \draw (40,0)--(51,0)--(51,1)--(40,1)--(40,0);
  %%%%%%%% pas hook
     \draw (-3,-5)--(-3,-13);
   \draw (-3,-19)--(-3,-23);
  \node(0a) at (-3,-5){$\bullet$};
  \node(0b) at (-3,-9){\textcolor{red}{$\bullet$}};
 \node(0c) at (-3,-13){\textcolor{red}{$\bullet$}};
   \node(v3p0) at (-3,-15.5){$\vdots$};
 \node(0d) at (-3,-19){\textcolor{red}{$\bullet$}};
  \node(0e) at (-3,-23){\textcolor{red}{$\bullet$}};
  \draw (-7.5,-5)--(-7.5,-13);
   \draw (-7.5,-19)--(-7.5,-31);
     \node(3p0000) at (-5,-11){$\cdots$}; 
  \node(0000a) at (-7.5,-5){$\bullet$};
  \node(0000b) at (-7.5,-9){\textcolor{red}{$\bullet$}};
 \node(0000c) at (-7.5,-13){\textcolor{red}{$\bullet$}};
   \node(v3p0000) at (-7.5,-15.5){$\vdots$};
 \node(0000d) at (-7.5,-19){\textcolor{red}{$\bullet$}};
  \node(0000e) at (-7.5,-23){\textcolor{red}{$\bullet$}};
   \node(0000f) at (-7.5,-27){\textcolor{red}{$\bullet$}};
   \node(0000g) at (-7.5,-31){\textcolor{red}{$\bullet$}};
     \node(3p0) at (3.5,-11){$\cdots$}; 
   \draw (1,-9)--(1,-13);
   \draw (1,-19)--(1,-27);
    \node(00a) at (1,-9){\textcolor{red}{$\bullet$}};
 \node(00b) at (1,-13){\textcolor{red}{$\bullet$}};
   \node(v3p00) at (1,-15.5){$\vdots$};
 \node(00c) at (1,-19){\textcolor{red}{$\bullet$}};
  \node(00d) at (1,-23){\textcolor{red}{$\bullet$}};
    \node(00e) at (1,-27){\textcolor{red}{$\bullet$}};
  \draw (5.5,-9)--(5.5,-13);
  \node(000a) at (5.5,-9){\textcolor{red}{$\bullet$}};
   \node(000b) at (5.5,-13){\textcolor{red}{$\bullet$}};
   \node(v3p000) at (5.5,-15.5){$\vdots$};
 \node(000c) at (5.5,-19){\textcolor{red}{$\bullet$}};
      \node(3p0) at (8,-16.5){$\cdots$}; 
     \draw (10,-15)--(10,-19);
   \node(0000000) at (10,-19){\textcolor{red}{$\bullet$}};
     \node(0000000a) at (14,-19){\textcolor{red}{$\bullet$}};
        \node(0000000) at (10,-15){\textcolor{red}{$\bullet$}};
   %%%%%%%% hook 1
  \node(1) at (20,-5){$\bullet$};
  \node(1a) at (20,-9){$\bullet$};
  \node(1b) at (20,-13){$\bullet$}; 
  \node(v3p1) at (20,-15.5){$\vdots$};
  \node(1c) at (20,-19){$\bullet$};
  \node(1d) at (20,-23){$\bullet$};
  \node(1e) at (20,-27){$\bullet$};
  \node(1f) at (20,-31){$\bullet$};
  \node(1g) at (20,-35){$\bullet$};
  \node(1h) at (20,-39){$\bullet$};
  \draw (20,-5)--(20,-14);
  \draw (20,-18)--(20,-39);
    \node(11) at (24.5,-5){$\bullet$};
  \node(11a) at (24.5,-9){$\bullet$};
  \node(11b) at (24.5,-13){$\bullet$}; 
  \node(v3p11) at (24.5,-15.5){$\vdots$};
  \node(11c) at (24.5,-19){$\bullet$};
  \node(11d) at (24.5,-23){$\bullet$};
  \node(11e) at (24.5,-27){$\bullet$};
  \draw (24.5,-5)--(24.5,-14);
  \draw (24.5,-18)--(24.5,-27);
    \draw (33,-9)--(33,-14);
  \draw (33,-18)--(33,-19);
  \node(2) at (33,-9){$\bullet$};
  \node(2b) at (33,-13){\textcolor{red}{$\bullet$}}; 
  \node(v3p2) at (33,-15.5){$\vdots$};
  \node(2c) at (33,-19){\textcolor{red}{$\bullet$}};
  \node(3p6) at (22.5,-11){$\cdots$}; 
    \node(3p7) at (31,-11){$\cdots$}; 
    \draw (28.5,-9)--(28.5,-14);
  \draw (28.5,-18)--(28.5,-31);
  \node(3) at (28.5,-9){$\bullet$};
  \node(3b) at (28.5,-13){\textcolor{red}{$\bullet$}}; 
  \node(v3p3) at (28.5,-15.5){$\vdots$};
  \node(3c) at (28.5,-19){\textcolor{red}{$\bullet$}};
    \node(2d) at (28.5,-23){\textcolor{red}{$\bullet$}};
  \node(2e) at (28.5,-27){\textcolor{red}{$\bullet$}};
  \node(2f) at (28.5,-31){\textcolor{red}{$\bullet$}};
    \node(3p2) at (35.5,-16.5){$\cdots$}; 
     \draw (37.5,-15)--(37.5,-19);
     \node(222) at (37.5,-19){\textcolor{red}{$\bullet$}};
     \node(222a) at (41.5,-19){\textcolor{red}{$\bullet$}};
        \node(222b) at (37.5,-15){\textcolor{red}{$\bullet$}};
  %%%%%%%%%%%%%%%%% 1 part
  \node(7) at (46,-5){$\bullet$};
  \node(b0) at (3,-4){$\overbrace{\hspace{170pt}}$};
  \node(b1) at (30,-4){$\overbrace{\hspace{170pt}}$};
  \node(b3) at (46,-4){$\overbrace{\hspace{10pt}}$};
  \node(P) at (-12,3){Partition of  $n$};
   \node(d) at (-12,-3){area};
\node(d0) at (-12,-5){$0$};
\node(d1) at (-12,-9){$1$};
\node(d2) at (-12,-13){$2$};
\node(vd) at (-12,-15.5){$\vdots$};
\node(d0) at (-15,-5){$s_a$};
\node(d1) at (-15,-9){$s_{a,1}$};
\node(d2) at (-15,-13){$s_{a,2}$};
\node(vd) at (-15,-15.5){$\vdots$};
\node(d2) at (-15,-19){$s_{a,i}$};
 \end{tikzpicture}
 \caption{The nodes represent paths. Each chain is associated to a Schur function in the variables $q$ and $t$. The height of the first node determines which Schur function. The partitions determine the Schur function in the variables $X$. Each chain can be associated to a Standard Young tableau corresponding to the shape of the partition. More than one chain can be associated with the same tableau. When nodes are in black, we know which paths they relate to, in red we do not.}\label{Fig : cristaux}
 \end{figure}
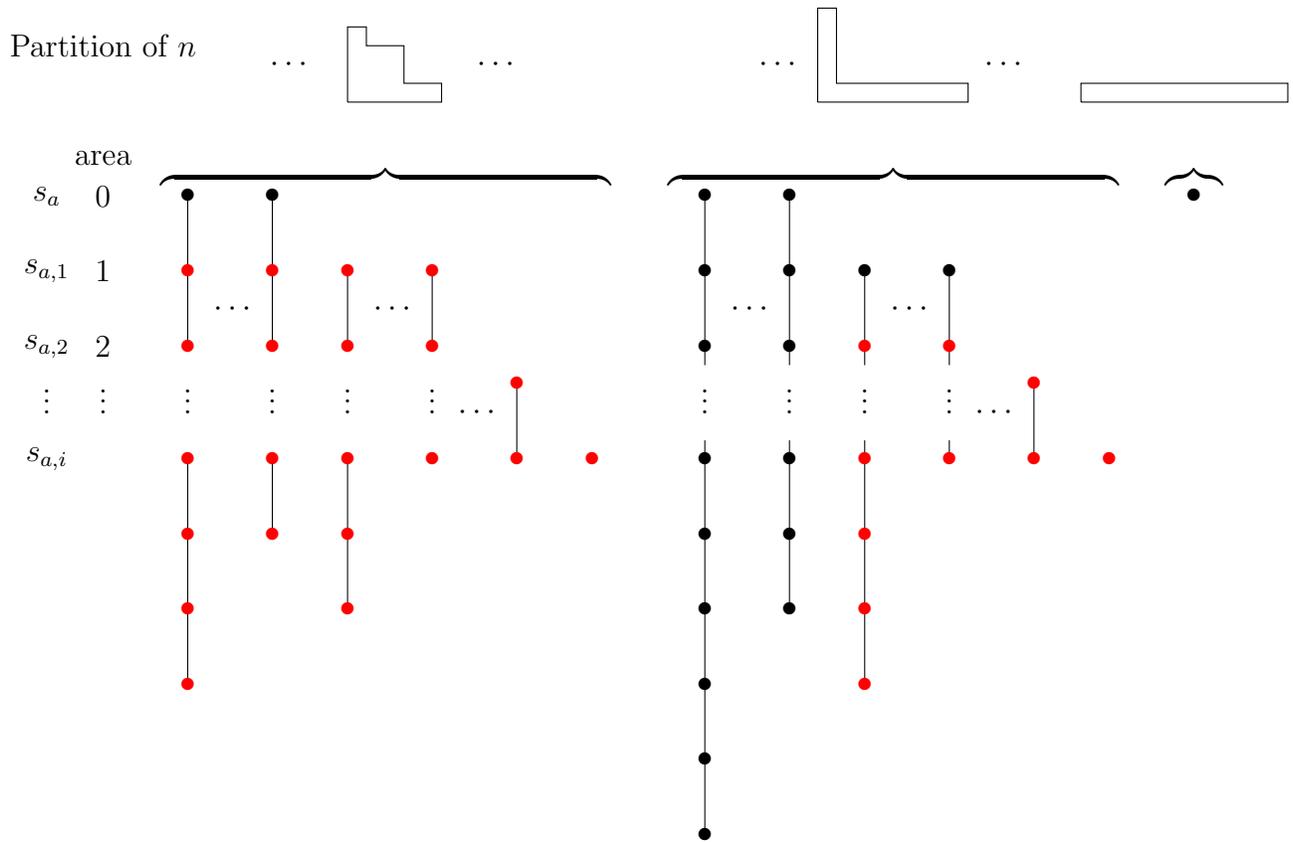

%%%%%%%%%%%%%%%%%%%%%%%%%%%%%%%%%%%%%%%%%%%%%%%
%%%%%%%%%%%%%%%%%%  Conclusion and further questions
%%%%%%%%%%%%%%%%%%%%%%%%%%%%%%%%%%%%%%%%%%%%%%%
\section{Conclusion and Further Questions}\label{Sec : conclu}
Proving Conjecture~\ref{Conj : nabla(e_n)} would be a great start. Moreover, can one describe nicely the algorithm described in Section~\ref{Sec : algo}, in terms of diagonal inversions and extend it to $m$-Schr\"oder paths? Here we are looking for more than just applying the zeta map.  This would allow us to know exactly what paths contribute to each Schur function, even when $m\not=1$. It might be easier to start by the following problem:
\begin{prob} Using the bounce statistic, generalize the algorithm in Section~\ref{Sec : algo} to all Schr\"oder paths.
\end{prob} 
This would give the Schr\"oder paths associated to all the Schur functions and not only the one with one part. It would also answer completely Haglund's open problem 3.11 of \cite{[H2008]}. One could also generalize the algorithm for labelled Dyck paths, relating to the Delta conjecture, and get a partial decomposition in Schur functions in the variables $q$ and $t$ indexed by partitions of length one. 
Using Corollary~\ref{Cor : m>1 dinv=1} it should be possible to decompose $\nabla^m(e_n)$ into the basis $s_\mu(q,t)s_\lambda(X)$. 
The following problems could lead to finding a partial decomposition $\nabla^m(e_n)$ into Schur functions in the $X$, $s_\lambda(X)$, when $\lambda$ is not a hook. Which is a known hard problem. 
\begin{prob}Using Lemma~\ref{Lem : m>2 dinv=1} decompose $\nabla^m(e_n)$ into the basis $s_\mu(q,t)F_c(X)$, for $\mu$ a hook and $c$ a composition.
\end{prob}
Even if the decomposition is in fundamental quasisymmetric functions, it would help get a partial decomposition $\nabla^m(e_n)$ into Schur functions, since it should be easier to regroup the fundamental quasisymmetric functions into Schur functions because there will be fewer coefficients.
Extending, the maps in this paper from $m$-Schr\"oder paths to tableaux with the multiplicity of the descent set, would help decompose completely the Schr\"oder paths into crystals. Of course the extended map must somewhat preserve the area and diagonal inversion statistic or area and bounce statistic through the major index and the number of descents. Another related problem is:
  \begin{prob}\label{Prob : solution conj}Find a bijection between parking functions $(\gamma,w)$, in an $n\times n$, grid, having diagonal inversion statistic value of $1$ and  Standard Young tableaux, $\tau$, with a multiplicity related to the major index of the tableau ($\m(\tau, i)-i-1$) such that  $\m(\tau, i)-i=\A(\mathcal{B}(\tau,i))$.
 \end{prob}
 Using the zeta map, we already have a bijection for such $n\times n$ parking functions $(\gamma,w)$ such that $\R(\gamma,w)\in \{n-d+1,\ldots,n\}\shuffle\{n-d,\ldots,1\}$. The idea here is to “extend” that bijection, with multiplicity. We need the multiplicity because there are more than just the parking functions $(\gamma,w)$ such that $\R(\gamma,w)\in \{n-d+1,\ldots,n\}\shuffle\{n-d,\ldots,1\}$, that contribute to $\nabla (e_n)$, seen as a sum of parking functions. The said paths are merely representatives. If the solution to Problem~\ref{Prob : solution conj} is indeed an extension of $\zeta\circ\mathcal{Q}_{n,d}$ this would  solve Conjecture~\ref{Conj : nabla(e_n)}.
 
As mentioned in Section~\ref{Sec : inclu-exclu} the insight coming from multivariate diagonal harmonics foresees a solution to the problem:
\begin{prob} Find a general map $\rho_{n,d}^{(i)}$ that partitions $\schP_{n,d,(i)}$.
\end{prob}  
This generalization would lead to more results on combinatorial formulas for multivariate diagonal harmonics, like the one in \cite{[Wal2019d]}. Actually, any explicit decomposition in terms of Schur functions of parking functions could be lifted with the tools discussed in \cite{[Wal2019a]} (the long version of \cite{[Wal2019d]}) and give an explicit combinatorial formula for a partial Schur decomposition of the multivariate diagonal harmonics.
   Finally, Proposition~\ref{Prop : 1 part parking} suggests a bijection between permutations and tableaux with a multiplicity such that $\binom{n}{2}-\dd(w)n+\m(w)=\m(\tau)$. Research on this last problem could lead to a decomposition of $\nabla^m(e_n)$ altogether. Since it further our knowledge of how fundamental quasi-symmetric functions index by permutations relate to Schur functions.
   \begin{prob}Find a combinatorial proof of Proposition~\ref{Prop : 1 part parking}.
   \end{prob}
\section*{Acknowledgments} 
Thank you to my advisor, Fran\c cois Bergeron, for introducing me to this beautiful subject and other helpful comments.
%\nocite{*}
\bibliographystyle{alpha}
\bibliography{bibliographie}

\newcommand{\etalchar}[1]{$^{#1}$}
\begin{thebibliography}{HHL{\etalchar{+}}05}

\bibitem[BG99]{[BG1999]}
F.~Bergeron and A.~M. Garsia.
\newblock Science fiction and {M}acdonald's polynomials.
\newblock In {\em Algebraic methods and {$q$}-special functions ({M}ontr\'eal,
  {QC}, 1996)}, volume~22 of {\em CRM Proc. Lecture Notes}, pages 1--52. Amer.
  Math. Soc., Providence, RI, 1999.

\bibitem[BGLX16]{[BGLX2016]}
Francois Bergeron, Adriano Garsia, Emily~Sergel Leven, and Guoce Xin.
\newblock Some remarkable new plethystic operators in the theory of {M}acdonald
  polynomials.
\newblock {\em J. Comb.}, 7(4):671--714, 2016.

\bibitem[CM18]{[CM2015]}
E.~Carlson and A.~Mellit.
\newblock A proof of the schuffle conjecture.
\newblock {\em J. Amer. Math. Soc.}, 31(8), 2018.

\bibitem[GH96]{[GH1996a]}
A.~M. Garsia and M.~Haiman.
\newblock A remarkable {$q,t$}-{C}atalan sequence and {$q$}-{L}agrange
  inversion.
\newblock {\em J. Algebraic Combin.}, 5(3):191--244, 1996.

\bibitem[Hag04]{[H2004]}
J.~Haglund.
\newblock A proof of the {$q,t$}-{S}chr\"oder conjecture.
\newblock {\em Int. Math. Res. Not.}, 2004(11):525--560, 2004.

\bibitem[Hag08]{[H2008]}
James Haglund.
\newblock {\em The {$q$},{$t$}-{C}atalan numbers and the space of diagonal
  harmonics}, volume~41 of {\em University Lecture Series}.
\newblock American Mathematical Society, Providence, RI, 2008.
\newblock With an appendix on the combinatorics of Macdonald polynomials.

\bibitem[Hai02]{[H2002]}
Mark Haiman.
\newblock Vanishing theorems and character formulas for the {H}ilbert scheme of
  points in the plane.
\newblock {\em Invent. Math.}, 149(2):371--407, 2002.

\bibitem[HHL{\etalchar{+}}05]{[HHLRU2005]}
J.~Haglund, M.~Haiman, N.~Loehr, J.~B. Remmel, and A.~Ulyanov.
\newblock A combinatorial formula for the character of the diagonal
  coinvariants.
\newblock {\em Duke Math. J.}, 126(2):195--232, 2005.

\bibitem[Mac95]{[Mac1995]}
I.~G. Macdonald.
\newblock {\em Symmetric functions and {H}all polynomials}.
\newblock Oxford Mathematical Monographs. The Clarendon Press, Oxford
  University Press, New York, second edition, 1995.
\newblock With contributions by A. Zelevinsky, Oxford Science Publications.

\bibitem[Mel16]{[M2016]}
C.~Mellit.
\newblock Toric braids and (m,n)-parking functions, 2016.

\bibitem[Son05]{[S2005]}
Chunwei Song.
\newblock The generalized {S}chr\"{o}der theory.
\newblock {\em Electron. J. Combin.}, 12:Research Paper 53, 10, 2005.

\bibitem[Sta79]{[S1979]}
Richard~P. Stanley.
\newblock Invariants of finite groups and their applications to combinatorics.
\newblock {\em Bull. Amer. Math. Soc. (N.S.)}, 1(3):475--511, 1979.

\bibitem[TW18]{[TW2018]}
Hugh Thomas and Nathan Williams.
\newblock Sweeping up zeta.
\newblock {\em Selecta Math. (N.S.)}, 24(3):2003--2034, 2018.

\bibitem[Wal19a]{[Wal2019a]}
N.~Wallace.
\newblock Explicit combinatorial formulas for some irreducible characters of
  the $gl_k\times \mathbb{S}_n$-module of multivariate diagonal harmonics,
  2019.

\bibitem[Wal19b]{[Wal2019d]}
N.~Wallace.
\newblock Explicit combinatorial formulas for some irreducible characters of
  the $gl_k\times \mathbb{S}_n$-module of multivariate diagonal harmonics.
\newblock {\em Séminaire Lotharingien de Combinatoire: Proceedings of the 31st
  Conference on Formal Power Series and Algebraic Combinatorics (Ljubljana)},
  82B(71):1--12, 2019.

\end{thebibliography}
\end{document}